\documentclass[12pt]{article}
\input epsf.tex
\usepackage{amsmath}
\usepackage{amsthm, textcomp}
\usepackage{amsfonts}
\usepackage{amssymb}
\usepackage{graphicx}
\usepackage{tikz}
\usepackage{lscape}
\usepackage{color}
\usetikzlibrary{automata}
\usepackage{tkz-euclide}
\usetikzlibrary{matrix}

\usepackage{latexsym}
\usepackage[left=2cm, right=2cm, top = 2cm, bottom=2 cm]{geometry}
\usepackage[english]{babel}
\usepackage[utf8]{inputenc}
 \usepackage{lipsum}% http://ctan.org/pkg/lipsum
\makeatletter
\g@addto@macro{\endabstract}{\@setabstract}
\newcommand{\authorfootnotes}{\renewcommand\thefootnote{\@fnsymbol\c@footnote}}%
\makeatother

\usepackage{amsmath,amsthm,amsfonts,amssymb}

\usepackage{epsfig}
\usepackage{hyperref}
\usepackage{amssymb}
\usepackage[all]{xy}
\xyoption{poly}
\usepackage{fancyhdr}
\usepackage{wrapfig}
\usepackage[T1]{fontenc}
\usepackage{amssymb,amsmath,amsthm}
\usepackage{amsmath,amssymb,amsthm}
\usepackage{graphicx}

\newcommand{\ncmd}{\newcommand}

\ncmd{\htop}{h_\mathrm{top}}
\ncmd{\hpol}{h_\mathrm{pol}}
\ncmd{\E}{\mathbb{E}}
\ncmd{\Oc}{\mathbb{O}}
\ncmd{\Ha}{\mathbb{H}}
\ncmd{\R}{\mathbf{R}}
\ncmd{\C}{\mathbf{C}}
\ncmd{\Z}{\mathbf{Z}}
\ncmd{\N}{\mathbf{N}}
\ncmd{\Sph}{\mathbb{S}}
\ncmd{\T}{\mathbb{T}}
\ncmd{\D}{\mathbb{D}}
\ncmd{\Q}{\mathbf{Q}}
\ncmd{\PP}{\mathbf{P}}

\newcommand{\one}{\boldsymbol{\bar{1}}}
\newtheorem{theorem}{Theorem}
\newtheorem{lemma}{Lemma}
\newtheorem{proposition}{Proposition}
\newtheorem{conjecture}{Conjecture}
\theoremstyle{definition}
\newtheorem{definition}{Definition}
\newtheorem*{remark}{Remark}
\newtheorem*{problem}{Problem}
\newtheorem*{example}{Example}

\newtheorem*{answer}{Answer to the problem}
\ncmd{\A}{\mathcal{A}}
\ncmd{\Ftau}{\mathcal{F}^{\tau}}
\ncmd{\SAF}{\mathrm{SAF}}
\ncmd{\IETn}{\mathrm{IET}^n}
\ncmd{\IETFn}{\mathrm{IETF}^n}
\ncmd{\IETFplus}{\mathrm{IETF}^{n+1}}
\ncmd{\IETFtwo}{\mathrm{IETF}^2}
\ncmd{\IETFthree}{\mathrm{IETF}^3}
\ncmd{\IETFfour}{\mathrm{IETF}^4}
\ncmd{\IETFfive}{\mathrm{IETF}^5}
\ncmd{\FETthree}{\mathrm{FET}^3}
\ncmd{\FETn}{\mathrm{FET}^n}
\ncmd{\FETfive}{\mathrm{FET}^5}
\ncmd{\FETplus}{\mathrm{FET}^{n+1}}
\ncmd{\CETn}{\mathrm{CET}^n_{\tau}}
\ncmd{\CETthree}{\mathrm{CET}^3_{\tau}}
\ncmd{\CETthreehalf}{\mathrm{CET}^3_{\frac{1}{2}}}
\ncmd{\Rauzy}{\mathcal{R}}
\ncmd{\CEThalf}{\mathrm{CET}^{3}_{\frac{1}{2}}}
\ncmd{\CETfour}{\mathrm{CET}^4_{\tau}}

\ncmd{\re}{\mathrm{Re}}
\ncmd{\im}{\mathrm{Im}}
\ncmd{\sing}{\mathrm{sing}}
\ncmd{\reg}{\mathrm{reg}}
\ncmd{\red}{\mathrm{red}}
\ncmd{\ttop}{\mathrm{top}}
\ncmd{\bbot}{\mathrm{bot}}
\ncmd{\bs}{\backslash}
\ncmd{\ov}{\overline}
\ncmd{\noi}{\noindent}
\ncmd{\di}{\displaystyle}
\ncmd{\ra}{\rightarrow}
\ncmd{\lra}{\longrightarrow}

\newcommand{\Addresses}{{% additional braces for segregating \footnotesize
  \bigskip
  \footnotesize

  Pascal~Hubert, \textsc{Aix Marseille Universit\'e, CNRS, Centrale Marseille, I2M - UMR 7373, F-13453 Marseille, France}\par\nopagebreak
  \textit{E-mail address}, P.~Hubert: \texttt{pascal.hubert@univ-amu.fr}

  \medskip

  Olga~Paris-Romaskevich, \textsc{Univ Rennes, CNRS, IRMAR - UMR 6625, F-35000 Rennes}\par\nopagebreak
  \textit{E-mail address}, O.~Paris-Romaskevich: \texttt{olga.romaskevich@univ-rennes1.fr, olga@pa-ro.net}
}}

\begin{document}
\title
%[Triangle tiling billiards and the exceptional family of escaping trajectories]
{Triangle tiling billiards and the exceptional family of their escaping trajectories: circumcenters and Rauzy gasket}
\author{Pascal Hubert, Olga Paris-Romaskevich}

%%\thanks{Aix Marseille Universit\'e, CNRS, Centrale Marseille, I2M - UMR 7373, F-13453 Marseille, France; email: pascal.hubert@univ-amu.fr},
% Olga Paris-Romaskevich
%%\thanks{Univ Rennes, CNRS, IRMAR - UMR 6625, F-35000 Rennes, France; email: olga.romaskevich@univ-rennes1.fr, olga@pa-ro.net}
%}
%\author{Pascal Hubert}
%\address{Aix Marseille Universit\'e, CNRS, Centrale Marseille, I2M - UMR 7373, F-13453 Marseille, France}
%\email{pascal.hubert@univ-amu.fr}
%\author{ Olga Paris-Romaskevich}
%\address{Univ Rennes, CNRS, IRMAR - UMR 6625, F-35000 Rennes}
%\email{olga.romaskevich@univ-rennes1.fr, olga@pa-ro.net}
%

\maketitle
\date{}

\begin{center}
\textbf{Abstract.}
Consider a periodic tiling of a plane by equal triangles obtained from the equilateral tiling by a linear transformation. We study a following \emph{tiling billiard}: a ball follows straight segments and bounces of the boundaries of the tiles into neighbouring tiles in such a way that the coefficient of refraction is equal to $-1$. We show that almost all the trajectories of such a billiard are either closed or escape linearly, and for closed trajectories we prove that their periods belong to the set $4\N^*+2$. We also give a precise description of the exceptional family of trajectories (of zero measure) : these trajectories escape non-linearly to infinity and approach fractal-like sets. We show that this exceptional family is parametrized by the famous Rauzy gasket. This proves several conjectures stated previously on triangle tiling billiards. In this work, we also give a more precise understanding of fully flipped minimal exchange transformations on $3$ and $4$ intervals by proving that they belong to a special hypersurface. Our proofs are based on the study of Rauzy graphs for interval exchange transformations with flips.
\end{center}


\section{Introduction.}
\subsection{What are tiling billiards and how do they behave ? : background and main result.}

Take a tiling (decomposition) of a plane by the shapes (possibly of infinite volume) with piece-wise smooth boundary. Consider a following billiard in such a tiling. A particle on the plane follows a straight line till it hits the boundary of one of the tiles. Then the trajectory continues in the neighbouring tile, following the rule of negative refraction with coefficient $-1$. In other words, the oriented angle that the trajectory makes with the side of the tile, changes its sign but keeps the same absolute value. We call the dynamical system defined in this way a \emph{tiling billiard}. See for example Figure \ref{fig:Matisse} for a tiling trajectory in a tiling coming from Matisse's collage \emph{The Snail}. 

In this article we will restrict ourselves to the case where the tiles are polygons or generalized polygons (infinite volume tiles with boundaries consisting of the union of straight line segments or rays). Tiling billiards in the square tiling and the equilateral triangle tilings have been first studied in a preprint \cite{MF} by Mascarenhas and Fluegel from the point of view of physics of negative refraction of light. Unfortunately, this paper has never been published and is not accesible on-line. Although, this study seems to be quite relevant since recently discovered materials can exhibit negative indices of refraction, see \cite{SSS01, SPW04, VZZ08}.
% and help to construct such wonderful objects as invisible cloacks ()

Tiling billiards got their name and were first presented as an interesting mathematical object in \cite{DDRSL16}, where the first non-trivial case of \emph{periodic triangle tilings} was considered. These tilings by congruent triangles are obtained by cutting the plane by three families of equidistant parallel lines. In this article, we mostly concentrate on the dynamics of such negative refraction billiards in these periodic triangle tilings: \emph{triangle tilling billiards}. Our work is inspired by \cite{BDFI18}, where the connection of these billiards with interval exchange transformations with flips is pointed out. In their work, Baird-Smith, Davis, Fromm and Iyer show that the dynamics of triangle tiling billiards has a first integral: the (oriented) distance between a segment of a trajectory in each crossed triangle and its circumcenter.
In particular, if a trajectory passes through a circumcenter of one of the triangles that it crosses, it passes through the circumcenters of all the crossed triangles. 

The existence of the first integral is crucial in order to show the (very fruitful!) connection of these billiards with interval exchange transformations with flips. It also helps to prove following results about triangle tiling billiards : first, each trajectory passes through any tile at most once and second, all bounded trajectories are closed. Even more surprisingly, the authors manage to construct, as a corollary of results in \cite{LPV07}, a singular trajectory (with a branching point in some vertex of the tiling) in a triangle tiling billiard that exhibits fractal behavior and passes through \emph{all} of the triangles in the tiling. 

In this work, we give a full description of the qualitative behavior of trajectories of triangle tiling billiards. Our main result is the following

%\begin{theorem}\label{thm:intro_main_theorem}
%Consider a trajectory $\delta$ of a triangle tiling billiard with the tiles congruent to the triangle $\Delta$. Suppose that $\Delta$ has the angles in its vertices equal to $\alpha, \beta$ and $\gamma$. Let $\mathcal{R}$ be the set of triangles $\Delta$ such that the point $\left(1-\frac{2}{\pi}\alpha, 1-\frac{2}{\pi}\beta, 1-\frac{2}{\pi}\gamma \right) \in \R^3_+$ belongs to the Rauzy gasket\footnote{For the reminder on the Rauzy gasket, see paragraph \ref{subs:HW_subsec}, Definition \ref{defn:Rauzy_gasket}.}. 
%Let $\mathcal{C}$ be the (well defined) set of trajectories that pass through the circumcenters of the crossed triangles.
%Then exactly one of the following four cases holds for $(\delta, \Delta)$: 
%\begin{itemize}
%\item[1.] a trajectory $\delta$ is closed and stable under perturbation and $(\delta,\Delta) \notin \mathcal{C} \times \mathcal{R}$. Furthermore, the period of $\delta$ is equal to $4n+2, n \in \N^*$;
%\item[2.] a trajectory $\delta$ is drift-periodic (is linearly escaping with a translation symmetry), the angles of $\Delta$ are dependent over $\mathbb{Q}$ (and, automatically, $\Delta \notin \mathcal{R}$)
%\item[3.] a trajectory $\delta$ is linearly escaping and its symbolic dynamics can be described as a sturmian sequence and $(\delta,\Delta) \notin \mathcal{C} \times \mathcal{R}$;
%\item[4.] a trajectory $\delta$ is non-linearly escaping, $\Delta$ is acute and $(\delta,\Delta) \in \mathcal{C} \times \mathcal{R}$.
%\end{itemize}
%\end{theorem}

\begin{theorem}
Consider a trajectory $\delta$ of a triangle tiling billiard with the tiles congruent to the triangle $\Delta$. Suppose that $\Delta$ has the angles in its vertices equal to $\alpha, \beta$ and $\gamma$. Let $\mathcal{R}$ be the set of triangles $\Delta$ such that the point $p:=\left(1-\frac{2}{\pi}\alpha, 1-\frac{2}{\pi}\beta, 1-\frac{2}{\pi}\gamma \right) \in \R^3_+$ belongs to the Rauzy gasket, $p \in \boldsymbol{\mathcal{R}}$. 
Let $\mathcal{C}$ be the (well defined) set of trajectories that pass through the circumcenters of the crossed triangles.
Then exactly one of the following four cases holds for $(\delta, \Delta)$: 
\begin{itemize}
\item[1.] a trajectory $\delta$ is closed and stable under perturbation and $(\delta,\Delta) \notin \mathcal{C} \times \mathcal{R}$. Furthermore, the period of $\delta$ is equal to $4n+2, n \in \N^*$;
\item[2.] a trajectory $\delta$ is drift-periodic (is linearly escaping with a translation symmetry), the angles of $\Delta$ are dependent over $\Q$ (and, automatically, $\Delta \notin \mathcal{R}$)
\item[3.] a trajectory $\delta$ is linearly escaping and its symbolic dynamics can be described as a sturmian sequence and $(\delta,\Delta) \notin \mathcal{C} \times \mathcal{R}$;
\item[4.] a trajectory $\delta$ is non-linearly escaping \footnote{To be precise, the result in this form is not yet proven but we formulate it like this for simplicity of exposition. We indeed prove the necessary condition of point 4. For sufficient condition, we prove a little less than we would like to. The non-linearly escaping behaviour holds for almost all points $(\delta,\Delta) \in \mathcal{C} \times \mathcal{R}$ with respect to the natural measure on the Rauzy gasket but we strongly believe the non-lineraly escaping beahvior holds for all points in $\mathcal{C} \times \mathcal{R}$. See Proposition \ref{thm:when_escape} for the exact statement.}, $\Delta$ is acute and $(\delta,\Delta) \in \mathcal{C} \times \mathcal{R}$.
\end{itemize}
\label{thm:intro_main_theorem}
\end{theorem}

This result gives a positive answer to  Conjectures 4.19  and 5.1 in \cite{BDFI18}, about the behaviour of exceptional trajectories as well as about the periods of all closed periodic trajectories. 

As a corollary we get that \emph{almost any triangle tiling trajectory is either closed or linearly escaping}. This property is related to the notion of \emph{integrability} of an interval exchange transformation with flips that we define and study throughout this article. We borrowed the name from the terminology for the studies on Novikov's problem on the semiclassical motion of an electron \cite{N82, Z84, D97}, since we think that these problems are related and hope to study them in future work.  

The proof of our main theorem uses in a crucial way a powerful tool of the modified Rauzy induction for interval exchange transformations with flips. Such a modification of the Rauzy induction was first introduced by Nogueira in \cite{N89}. He proved that almost any interval exchange transformation with flips has a periodic sub-interval :  Rauzy induction amost always stops. We implicitely use this result all along our work: triangle tiling billiards have abundant and stable closed trajectories exactly thanks to the phenomenon noticed by Nogueira. Nogueira's theorem shows how different interval exchange transformations with flips are from classic interval exchange transformations that do preserve orientation : the first are almost never minimal, the second are almost always minimal\footnote{Except for the cases when they obviously are not ! See classical Keane's theorem \cite{K75}}.

The proof of Theorem \ref{thm:intro_main_theorem} proceeds in two major steps. First, we prove that there exists an invariant hyperspace for the Rauzy induction procedure that corresponds exactly to the space of trajectories hitting the circumcenters. This step uses as a key result Lemma \ref{lemma:main_vector_preserved_lemma} which we have proven by some (not very heavy but still...) computer assisted calculations of Rauzy graphs for interval exchange transformations with flips. For now, Lemma \ref{lemma:main_vector_preserved_lemma} seems quite miraculous and one of the major goals of our future work is to understand reasons behind its claim. The second step of the proof of Theorem \ref{thm:intro_main_theorem} is based on a more precise understanding of the structure of permutations corresponding to the stopping points of Rauzy induction. 

This article also contains links to the Rauzy graphs that were drawn by the program written by Paul Mercat in order to conclude the proof of Lemma \ref{lemma:main_vector_preserved_lemma} as well as to illustrate some of the arguments. We think that the study of the Rauzy graphs for interval exchange transformations with flips is a very interesting area for future research. As already mentionned above, these graphs are very different from Rauzy graphs for classical interval exchange transformations. The connections between these two worlds may be interesting to explore. We speak about this more in the last Section \ref{sec:perspectives}, as well as about other perspectives and open questions related to tiling billiards and interval exchange transformations with flips. 

To conclude the introduction, we would like to say that the area of tiling biliards is a very young and small (for now...) niche of dynamics of billiards. We find it very attractive.
As far as we know, there are only a very few works on this subject. In addition to the works already mentionned, we are aware of the existence of two more works. First, in \cite{DH18} Davis and Hooper study tiling billiards on the trihexagonal tiling and their ergodic properties. Second, in a slightly larger setting, Glendinning in \cite{G16} studies the dynamics of tiling billiards on the standard infinite checkerboard. He supposes that the black and white tiles have different refraction coefficient indices, $k_1$ and $k_2$, and relates the dynamics to interval exchange transformations in the case when $\frac{k_1}{k_2}>\sqrt{2}$. 

As far as we know, the first published (in $2016$) works on tiling billiards are \cite{DDRSL16} and \cite{G16}. Although, one could say that the story of tiling billiards starts almost thirty years earlier, in $1989$ with the work of Nogueira \cite{N89} on interval exchange transformations with flips. Indeed, on the last page of his work, Nogueira defines \emph{the billiards with flips in polygons}. One can see that the study of such a billiard in a square is equivalent to the study of the square periodic tiling billiard, and a study of such a billiard in a triangle is equivalent to the study of triangle tiling billiards. In this work this connection will be made explicit. \footnote{For the explanation of this connection, see the triangletangent system in paragraph \ref{subs:definitions}} Although, the system of Nogueira can't be generalized to the tilings with different tiles, for example.

\begin{figure}
\includegraphics*[scale=0.17]{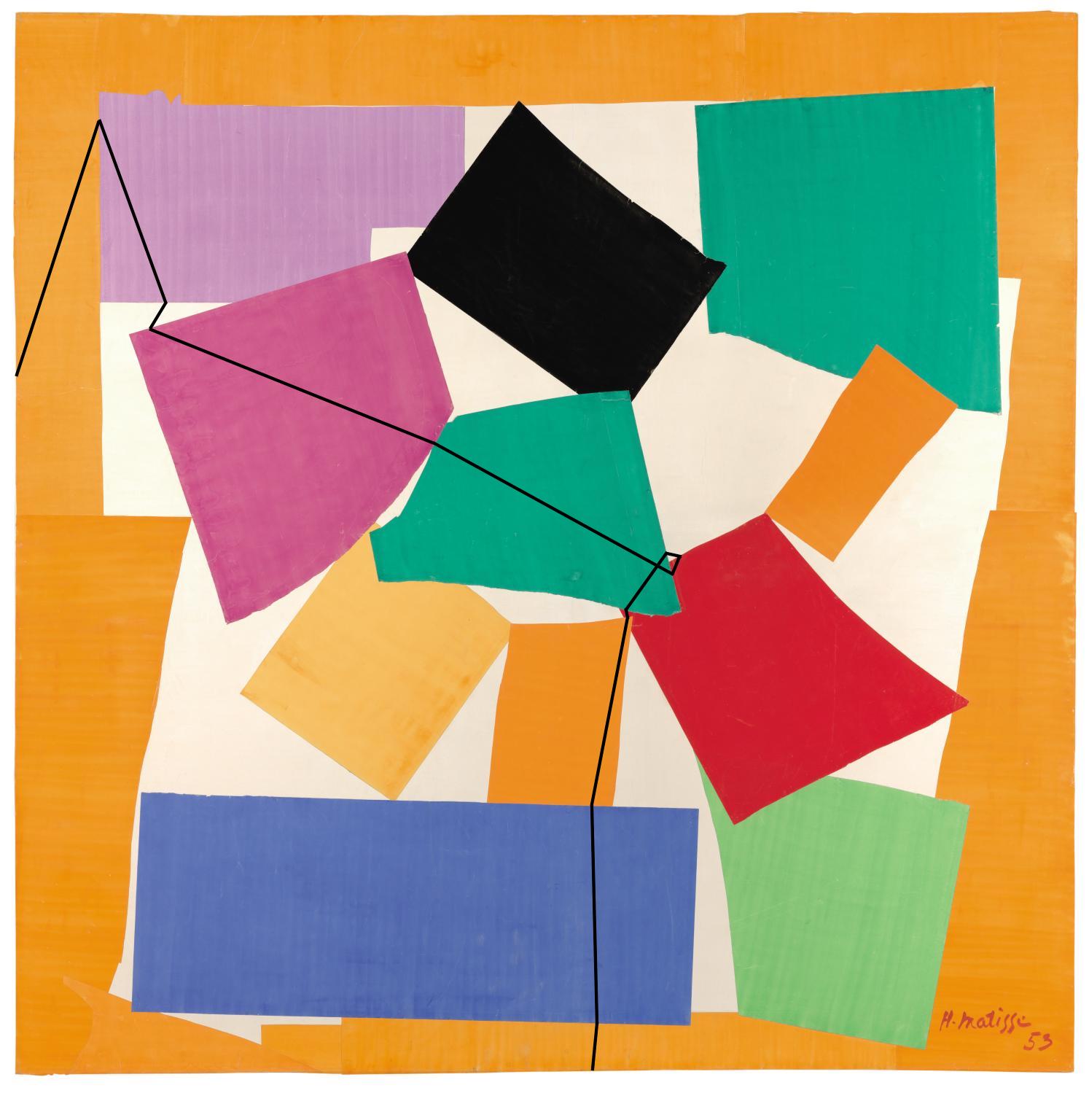}
\centering
\caption[]{A part of a tiling billiard trajectory in the tiling defined by one of Matisse's collages. 
Initital painting (except for a black tiling billiard trajectory): H. Matisse, \emph{L'Escargot (Snail)}, 1953,  Gouache on paper, cut and pasted on paper mounted on canvas, Tate Museum, Succession Henri Matisse/DACS 2018}\label{fig:Matisse}
\end{figure}

\subsection{Plan of the paper.}
Different dynamical systems equivalent to the triangle tiling billiard are defined in Section \ref{sec:definition}. One of them  is a family of  interval exchange transformations with flips (fully flipped  $3$-interval exchange transformations on $\Sph^1$).  Qualitative behavior of orbits in triangle tiling billiards are discussed in Section \ref{subs:different}.  A modified Rauzy induction for interval exchange transformations with flips is precisely defined in Section \ref{sec: modified Rauzy}. This is a crucial tool in this work. Rauzy graphs are also introduced and the work or Nogueira is revisited. Section \ref{sec:one-half} gives a necessary condition for minimality. The proof uses the modified Rauzy induction. Section \ref{sec:integrability_section} explores integrability for interval exchange transformations with flips. In Section \ref{sec:properties_tiling_billiards}, properties of the orbits of triangle tiling billiards are derived from tools and ideas introduced in the previous sections: symbolic dynamics, periodic orbits, generic and exotic dynamics are considered. Further remarks and open questions are mentioned in Section \ref{sec:perspectives}.

We highlight that the statement of Theorem \ref{thm:intro_main_theorem} is a union of statements of Propositions \ref{thm:generic}, \ref{thm:period} and \ref{thm:when_escape}. These propositions will be proved separately in the paper.

\section{Different approaches of triangle tiling billiards.}\label{sec:definition}
\subsection{Four seemingly different dynamical systems.}\label{subs:definitions}

Take some triangle $\Delta$. Suppose that this triangle $\Delta$ has its angles equal to $\alpha, \beta$ and $\gamma$. From now on till the end of the article the corresponding vertices are denoted by $A,B$ and $C$ and the sides facing these vertices - by $a,b$ and $c$. This definition was given in the introduction but we repeat it here for completeness.

\begin{definition}[\textbf{\emph{Triangle tiling billiards}}]\label{def:triangle_tiling_billiards_1}
Consider a periodic tiling of the plane by triangular tiles which are all congruent to $\Delta$ which is obtained by cutting the plane by three families of equidistant parallel lines. A \emph{triangle tilling billiard} is a dynamical system of a motion of a point particle in such a tiling defined in a following way. A particle follows a straight line till it hits the side of some tile. The trajectory continues in the neighbouring tile, following the rule of negative refraction with coefficient $-1$, see Figure \ref{fig:triangletiling}. 
\end{definition}

Throughout this article we will be interested in studying the dynamics of triangle tiling billiards for different triangles $\Delta$ and different initial conditions of the trajectory. Note that the tiles can be rescaled in such a way that $\Delta$ has area $1$ - the dynamics is invariant under homothety. The parameters of the dynamics are hence the angles $\alpha, \beta, \gamma$. We denote the sides of $\Delta$, corresponding to these angles, as $a,b$ and $c$. 

\smallskip

Now, let us give three more definitions of other, seemingly unrelated, dynamical systems. Then we will clarify their connection to the triangle tiling billiards untill the end of this Section.

\begin{figure}
\centering
\includegraphics*[scale=0.4]{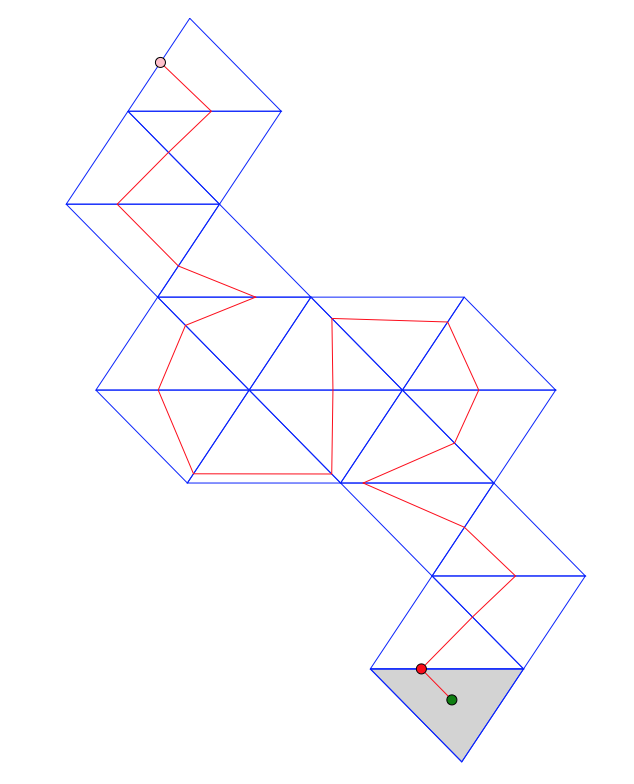}
\caption[]{
%\textbf{PASCAL! CAN WE DRAW A BETTER PICTURE ? This is a first picture of a tiling billiard... And also a picture explaining how the tiling is made !}
The first $20$ segments of a piecewise linear trajectory in a negative triangle tiling billiard. The refraction coefficient $k$ is equal to $-1$.}\label{fig:triangletiling}
\end{figure}

%
%\begin{center}
%\textbf{(II) Reflection in a circumcircle.}
%\end{center}

\begin{definition}[\textbf{\emph{Reflection in a circumcircle.}}]\label{def:reflection_in_a_circumcircle:2}
Consider a circle on the complex plane centered at the origin. Fix a triangle $\Delta$ that can be inscribed in this circle. For any $\tau \in (0, 1)$ define an oriented chord $l$ in the circle that connects the point with argument $2 \pi \tau$ with the point of argument $0$ and is headed to the later. Let us now define a following dynamical system on some subset of this circle. For any $X \in [0,1)$, inscribe a triangle $\Delta_X$ congruent to $\Delta$ in such a way that the vertices $A, B$ and $C$ are placed on the circle in a counter-clockwise manner and that the argument of the vertex $A$ as of a complex number is equal to $2 \pi X$.

The dynamical system will be defined for a subset of such $X \in [0,1)$ such that the chord $l$ intersects the corresponding triangle $\Delta_X$. Take the last (following the orienation of the chord $l$) side $s_X$ of triangle $\Delta_X$ that the chord $l$ intersects. Define $\bar{\Delta}_X$ as a triangle congruent to $\Delta$, inscribed in the same circle, sharing the side $s_X$ with $\Delta_X$ and having an opposite orientation. Now define $\Delta'_X$ as a triangle obtained by reflecting $\bar{\Delta}_X$ with respect to the diameter of the circle perpendicular to the chord $l$. The orientation of $\Delta'_X$ is the same as that of the initial triangle $\Delta_X$. The map $F_{\Delta, l}: X \mapsto X', X \in \Sph^1$ is a \emph{reflection in a circumcircle}, see Figure \ref{fig:circumreflection}. 
\end{definition}

Let us make a couple of important remarks. First, for the map $\bar{F}_{\Delta, l}: X \mapsto \bar{X}$ defined analogically, we see that its square is equal to that of the map $F_{\Delta, l}$ of the reflection in a circumcircle: $\bar{F}^2=F^2$. Second, the reflection in a circumcircle is not necessarily defined on the full circle $\Sph^1$. For example,for the obtuse triangle $\Delta$, the map $F_{\Delta, l}$ is never defined on the full circle for any $\tau \in [0,1)$. On the contrary, for $\Delta$ acute and $\tau=\frac{1}{2}$, the map $F_{\Delta, l}$ is defined on the full circle. In this case $l$ is a diameter.

%
%This system may seem quite tricky, and one could ask why we just don't reflect with respect to the sides crossed by the line $l$ each time and consider $X \mapsto \bar{X}$ as a dynamical system: the passage to the triangle $\Delta'$ is needed in order to make the direction of the line $l$ an invariant of the system - for this one has to change the coordinate on the circle. 
%
%Indeed, for the trajectory in the triangle $\bar{\Delta}$ the direction of the line $l$ is changed to the opposite. This change is just a technical detail since the squares of these two maps coincide: $\bar{\bar{X}}=X'$, see Figure \ref{fig:circumreflection}.

%Note that the reflection in a circumcircle is defined on some subset of the circle of initial conditions $X \in [0,1)$ since for some values of parameter $X$ the line $l$ won't intersect the triangle $\Delta$.
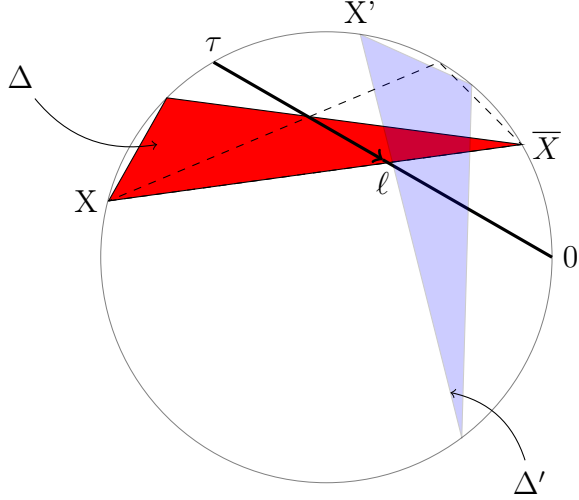
\begin{figure}

\centering
\begin{tikzpicture}[scale = 1.5]

\draw[gray,very thin](0,0)circle (2);

\draw[fill =  red] (-1.93, 0.5)--(-1.414, 1.414)--(1.73,1)--cycle;

\draw[dashed] (-1.93, 0.5)--(1.73,1)--(1,1.73)--cycle;

\draw[->, very thick] (-1,1.73)--(0.5,0.865) node[below]{$\ell$};
\draw[ very thick] (0.5,0.865)--(2,0);

\draw (-1,1.73) node[above]{$\tau$};

\draw (-1.93, 0.5) node[left]{X};

\draw (1.73,1) node[right]{$\overline{X}$};

%\draw(-2,-2.8)--(2,2.8);

\draw[ fill=blue, opacity=0.2](1.2, -1.6)--(1.288,1.53)--(0.3,1.977)--cycle;

\draw (0.3,1.977) node[above]{X'};

\draw (2,0) node[right]{0};

\path[->] (-2.5,1.6) node[left]{$\Delta$} edge [bend right] (-1.5,1);
\path[->] (1.8,-2) node{$\Delta'$} edge [bend right] (1.1,-1.2);

\end{tikzpicture}   

\caption[]{\emph{Reflection in a circumcircle} $F_{\Delta, l}$: the initial triangle $\Delta$ is mapped to the triangle $\Delta'$ of the same orientation. The intermediary step of the process passes by the dotted triangle $\bar{\Delta}_X.$ The oriented chord $l$ intersects (the last time) the triangle $\Delta_X$ in the side $s_X$, and the triangle $\Delta'_X$ in the side with another label, $s_{X'}$. } \label{fig:circumreflection}
\end{figure}

\begin{definition}[\textbf{\emph{Triangletangent system}}]\label{def:triangletangent_system_3}
Fix a triangle $\Delta$ inscribed in its circumcircle, and fix a number $\tau \in (0,1)$ for a parameter. Consider a smaller circle $\T$ homothetic to the initial circumcircle, with a coefficient of homothety equal to $\left| \cos \pi \tau\right|$. The \emph{triangletangent system} is a map defined on the (one-dimensional) space of oriented segments connecting two sides of the triangle $\Delta$ and tangent to $\T$. These segments are parametrized by a subset of points on the sides of the triangle $\Delta$ which correspond to their end points. Then for any segment $X$ one associates a segment with a point $X'$ on the same side of the triangle but symmetrical with respect to the middle of the side. The map $F_{\Delta, \T}: X \mapsto X'$ defined in such a way on the space of oriented tangent segments is a \emph{reflection in a triangle with respect to a circle of tangency}, see Figure \ref{fig:triangletangent}.
\end{definition}

The map $F_{\Delta, \T}$ here is well-defined. Indeed, if there exists a segment tangent to $\T$ with an end-point in $X$, then the segment tangent to $\T$ with an end-point in a symmetrical point $X'$ exists as well. This folows from the fact that the circle $\T$ has its center in the circumcenter of $\Delta$, and the circumcenter is placed on the intersection of line segment bisectors. Note that as in the case of the reflection in a circumcircle, the system $F_{\Delta, \T}$ is not necessarily defined for any point $X$ on the sides of the triangle. For $\Delta$ acute, $F_{\Delta, \T}$ is defined everywhere if $\tau$ is small enough - smaller than any of the distances between the circumcenter to the sides of the triangle $\Delta$.

This system in this form was suggested to us by Shigeki Akiyama. This system is also a restriction on some subset of a billiard with flips that was defined by Nogueira for \emph{any} polygon (and not only a triangle) in \cite{N89}. Although in his approach, Nogueira didn't mention the first integral that appears in the case of triangle billiard (in this definition, this invariant is represented by the circle $\T$). The triangletangent system is exactly the restriction of Nogueira's billiard in a triangle to the subset of trajectories with the same value of the first integral.

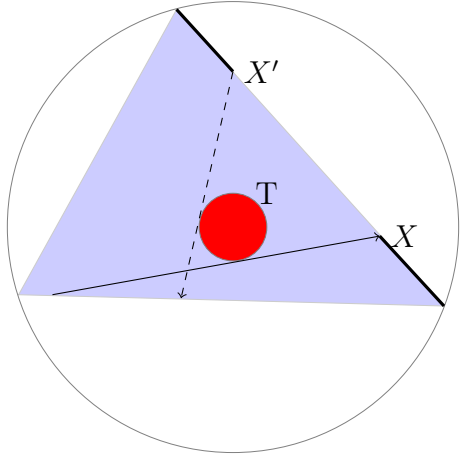
\begin{figure}
\centering
\begin{tikzpicture}[scale = 1.5]

\draw[gray,very thin](0,0)circle (2);

\draw[fill=blue, opacity=0.2] (-1.9,-0.6)--(1.87, -0.7)--(-0.5,1.93)--cycle;

\draw[gray,very thin, fill = red](0,0)circle (0.3);
\draw (0.3, 0.3) node{T};

\draw[<-] (1.3, -0.08) node[above, right]{$X$} --(-1.6, -0.6) ;

\draw[dashed, ->] (0, 1.38) node[above, right]{$X'$} --(-0.46, -0.63);

\draw[very thick] (0, 1.38)  --(-0.5,1.93);

\draw[very thick] (1.3, -0.08)   --(1.87, -0.7);

\end{tikzpicture}   
\caption[]{\emph{Triangletangent system} $F_{\Delta, \T}$: the oriented segment with its end-point $X$ is mapped to the oriented (dotted) segment with its end-point $X'$.}\label{fig:triangletangent}
\end{figure}

\begin{definition}[\textbf{\emph{Fully flipped  }}$3$\emph{\textbf{-interval exchange transformations on }}$\Sph^1$.]\label{def:fully_flpped_3_interval_exchange_transformations_on_the_circle:4}
Fix the numbers $\tau \in [0,1)$ and $l_j \in \R_+, j=1,2,3$ such that $\sum_{j=1}^3 l_j=1$. Define a map $F: \Sph^1 \rightarrow \Sph^1$ by a following explicit formula : 
\begin{equation*}
F_{\tau, l_1, l_2, l_3}(x):=
  \begin{cases}
 -x+l_1+\tau \; \; \; \; \mathrm{mod} \; 1 & \; \; \textit{if}  \; \; \; x \in I_a:= [0,l_1) \\
      -x+l_2 +\tau   \; \; \; \; \mathrm{mod}\;  1  & \; \; \textit{if} \; \; \; x \in I_b: = [l_1, l_1+l_2)\\
      -x+l_3 +\tau  \; \; \; \; \mathrm{mod}\;  1 & \; \; \textit{if} \; \; \; x \in I_c: =[l_1+l_2,1)
    \end{cases}.
\end{equation*}
We call such a transformation\emph{a fully flipped }$3$-\emph{interval exchange transformation on the circle with a trivial permutation of intervals.} We denote the set of all such transformations by $\CETthree$.
\end{definition}

The letter C in the name  $\CETthree$ corresponds to the word \emph{circle}, the letters E and T - to (interval) \emph{exchange transformation}, and the number $3$ to the number of continuity intervals on the circle, $\tau$ being a parameter. A map $F_{\tau, l_1, l_2, l_3} \in \CETthree$ is a continuous transformation of the circle outside a three-point set $\{l_1, l_1+l_2,1\} \in \Sph^1$. Each one of the intervals of continuity is shifted by $\tau$ and then \emph{flipped}. See Figure \ref{fig:fullyflipped} for the illustration. 

Let us make a couple of remarks: first, for $F \in \CETthree$ its square $T=F^2$ is a standard interval exchange transformation (without flips). For almost all the values of parameters, the map $T$ is a $6$-interval exchange transformation on the circle $\Sph^1$. Second, the space $\CETthree$ is parametrized by a three-dimensional space which is a direct product of the simplex of lengths and the circle of parameter $\tau$.

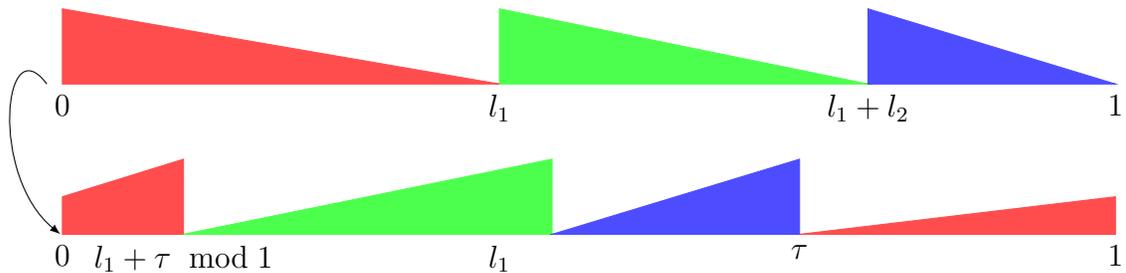
\begin{figure}
\centering
\begin{tikzpicture}[xscale=14]
\path[draw, fill=red,red, opacity=0.7] (0,0)--(0,1)--(0.415,0)--cycle;
\path[draw, fill=green, green, opacity=0.7] (.415,0)--(.415,1)--(0.765,0)--cycle;
\path[draw, fill=blue, blue, opacity=0.7] (.765,0)--(1,0)--(.765,1)--cycle;
\path[draw] (0.765,0) node[below]{$l_1+l_2$} --cycle;
\path[draw] (0,0) node[below] (1){$0$} --cycle;
\path[draw] (1,0) node[below]{$1$} --cycle;
\path[draw] (.415,0) node[below]{$l_1$} --cycle;

 \tikzset{
        arrow/.style={
            color=black,
            draw=black,
            -latex,
                font=\fontsize{12}{12}\selectfont},
        }

\path[draw] (0,-2) node[below] (2){$0$} --cycle;
\path[draw] (1,-2) node[below]{$1$} --cycle;
\path[draw, fill=red,red, opacity=0.7] (0,-2)--(0,0.5-2)--(.115,1-2)--(.115,0-2)--cycle;
\path[draw, fill=green, green, opacity=0.7] (.115,0-2)--(.465,1-2)--(0.465,0-2)--cycle;
\path[draw, fill=blue, blue, opacity=0.7] (.465,0-2)--(.7,1-2)--(.7,0-2)--cycle;
\path[draw] (0.7,0-2) node[below]{$\tau$} --cycle;
\path[draw, fill=red,red, opacity=0.7] (0.7,-2)--(1,-2)--(1,0.5-2)--cycle;
\path[draw] (.115,0-2) node[below]{$l_1+\tau \;\; \mathrm{mod} \; 1$} --cycle;
\path[draw] (.415,0-2) node[below]{$l_1$} --cycle;
\draw[arrow](1) to [out=94,in=95]  (2);
\end{tikzpicture}
\caption[]{\emph{Fully flipped }$3$\emph{-IET.} Along this paper we draw the geometric shapes (triangles and quadrilaterals, for most of the time) over the intervals of continuity of interval exchange transformations. This is done in order to simplify the understanding of the action of the maps by making it more visual: the orientation of the form drawn over an interval changes from the interval to its image if the interval is flipped by an interval exchange transformation. Even though the idea of drawing geometric shapes over the intervals exchanged by an interval exchange transformation is very old, a nice idea to flip the geometric forms above the flipped intervals comes from \cite{BDFI18} (see Figure 9 there).}\label{fig:fullyflipped}
\end{figure}

\subsection{The same system and its four faces: different tools for understanding.}\label{subs:connections}
Here are the four dynamical systems defined in a previous paragraph: 
\begin{itemize}
\item[1.] Triangle tiling billiards guided by the negative refraction in the tiling with tiles congruent to $\Delta$ (Definition \ref{def:triangle_tiling_billiards_1});
\item[2.] reflection $F_{\Delta, l}$ of a triangle in its circumcircle by following an oriented fixed chord (Definition \ref{def:reflection_in_a_circumcircle:2});
\item[3.] triangletangent system $F_{\Delta, \T}$ (Definition \ref{def:triangletangent_system_3});
\item[4.] a family of maps in $\CETthree$ (Definition \ref{def:fully_flpped_3_interval_exchange_transformations_on_the_circle:4}).
\end{itemize}
What do these systems have in common? All of the systems use a triangle $\Delta$ as a parameter: for all of them except for the system\ref{def:fully_flpped_3_interval_exchange_transformations_on_the_circle:4} it is explicit, and for this one the lengths of the intervals of continuity $l_j, j=1,2,3$ can be reparametrized to correspond to the angles of some triangle $\Delta$. The systems \ref{def:reflection_in_a_circumcircle:2}, \ref{def:triangletangent_system_3} and \ref{def:fully_flpped_3_interval_exchange_transformations_on_the_circle:4}) are one-dimensional systems (they are all defined on the circle or on its subsets), and all of them have a parameter $\tau$ in them (defining either the chord $l$, the circle $\T$ or the shift $\tau$ in the action of a fully flipped IET). A triangle billiard \ref{def:triangle_tiling_billiards_1} is a $2$-dimensional system. 

\smallskip

A simple and crucial remark ("\emph{folding observation}") from \cite{BDFI18} helps to reduce the triangle tiling billiards to a $1$-dimensional system by finding the first integral $d(\delta)$ for its trajectories. Indeed, for any trajectory $\delta$ of a triangle tiling billiard and two consecutive triangles $\Delta, \Delta'$ that it crosses, these triangles can be folded one onto another (along the crossed edge). In this way, the segments of the trajectory fold onto \emph{one line }and the triangles fold in such a way that their images have a \emph{common circumcircle}. This permits to prove

\begin{proposition}[\cite{BDFI18}]\label{prop:Dianacomeback}
In a triangle tiling billiard the following holds:
\begin{itemize}
\item[1.] Every trajectory crosses each triangle in the tiling at most once;
\item[2.] the distance $d(\delta, \Delta)$ between a segment of a trajectory in any triangle $\Delta$ crossed by it and the circumcenter of $\Delta$ is an invariant of the trajectory $d(\delta)$ (doesn't depend on $\Delta$). Moreover, the circumcenter of each crossed by the trajectory $\delta$ triangle $\Delta$ stays on the same side from the (oriented) segment of $\delta$;
\item[3.] all bounded trajectories are closed.
\end{itemize}
\end{proposition}

\begin{figure}
\centering
\includegraphics*[scale=0.18]{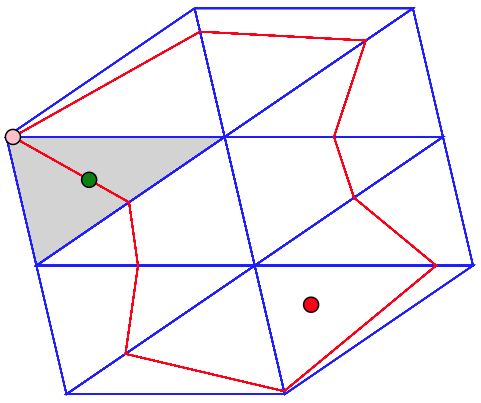}
\includegraphics*[scale=0.3]{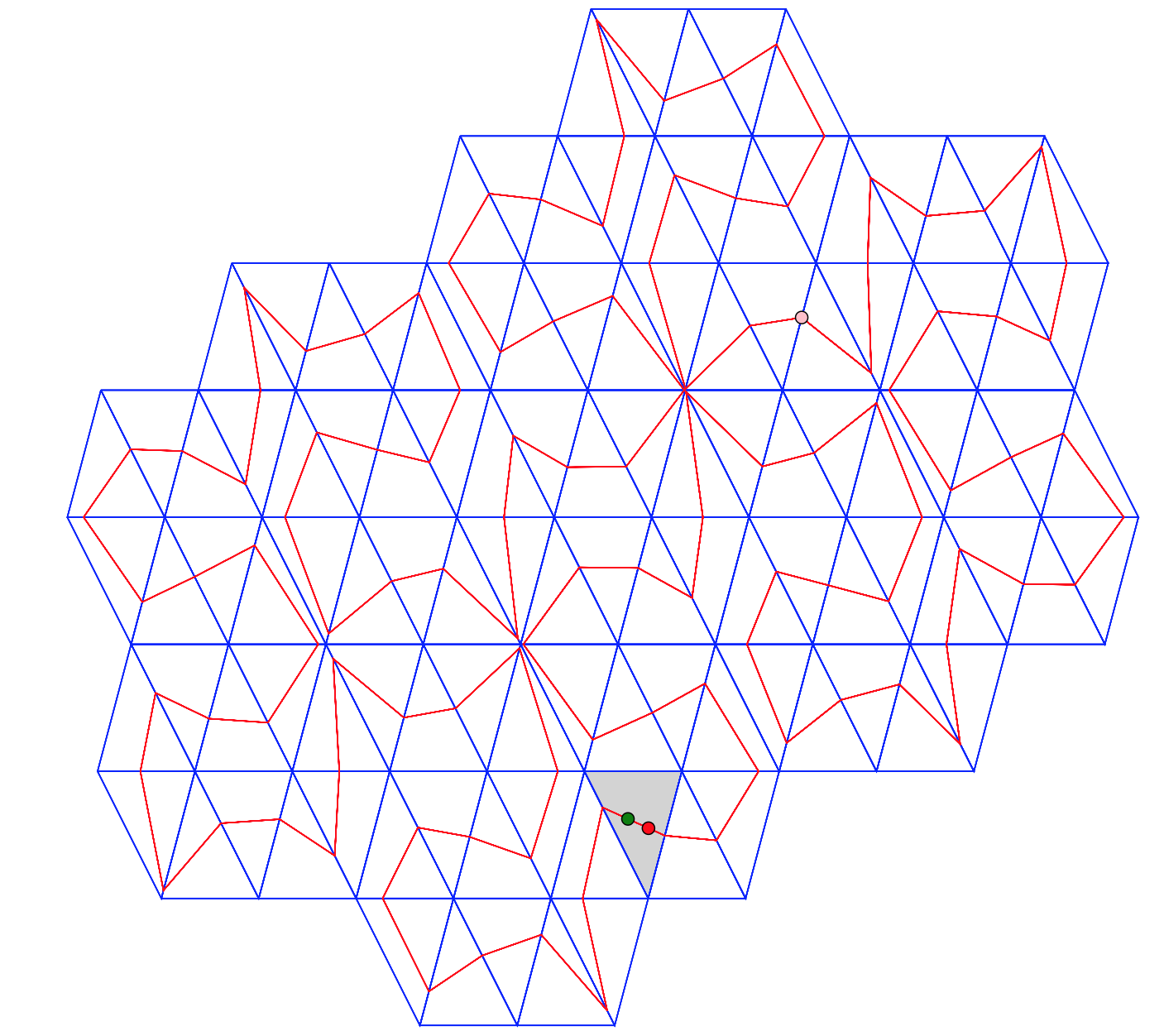}
\includegraphics*[scale=0.3]{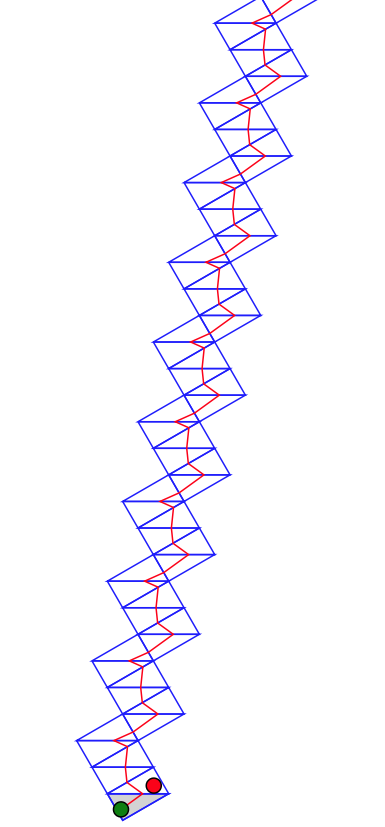}
\includegraphics*[scale=0.26]{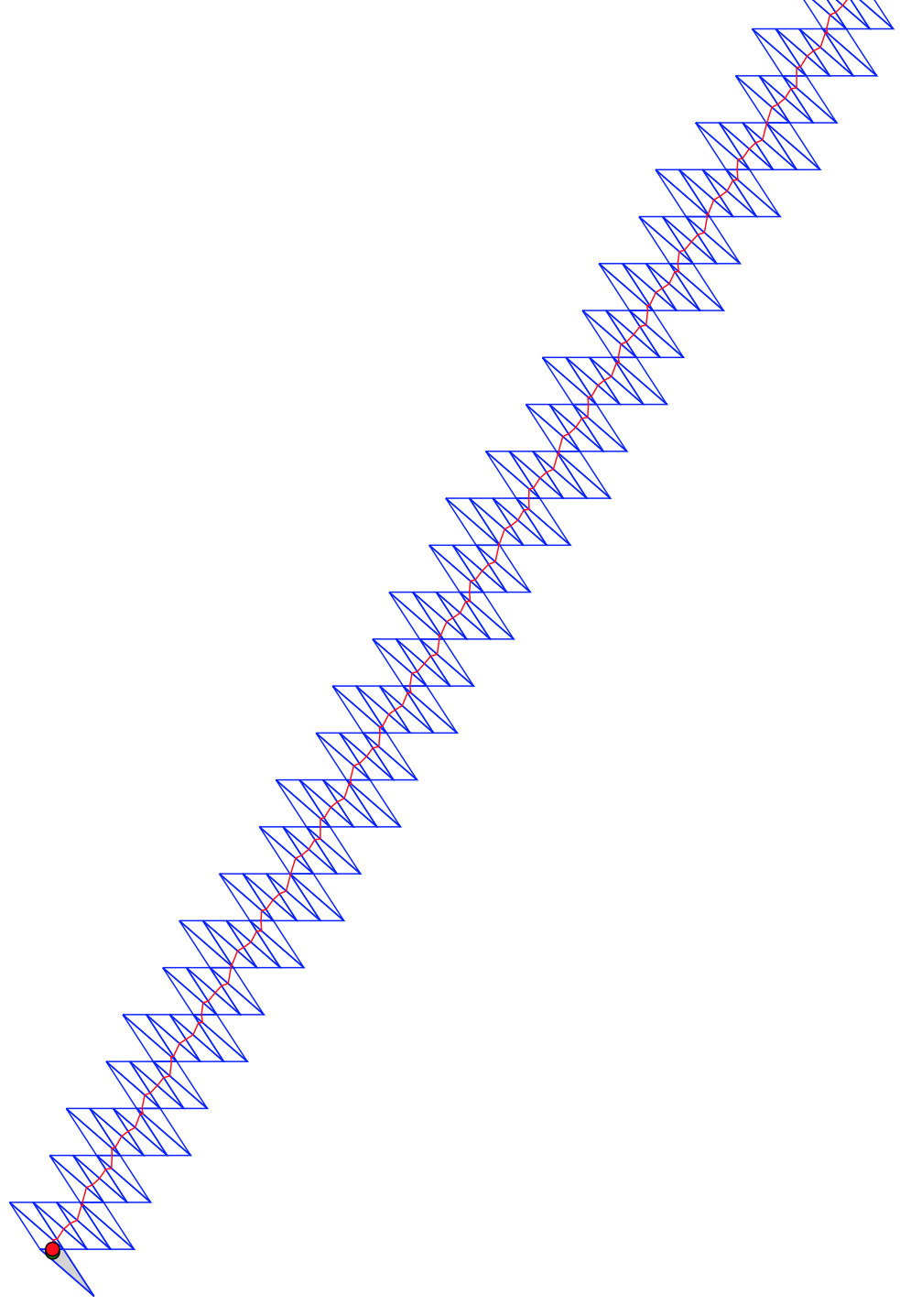}
\includegraphics*[scale=0.3]{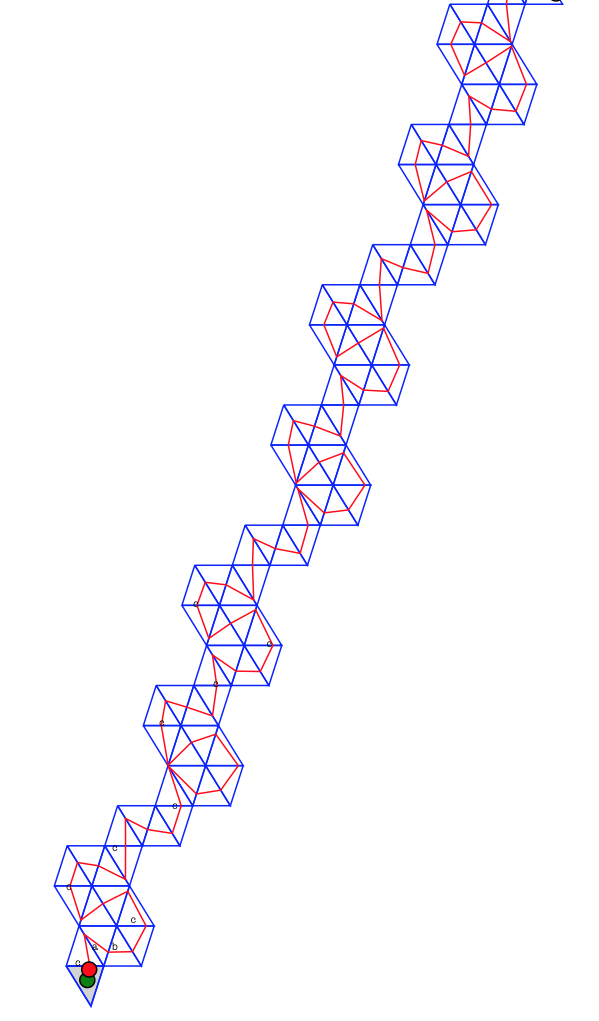}
\includegraphics*[scale=0.3]{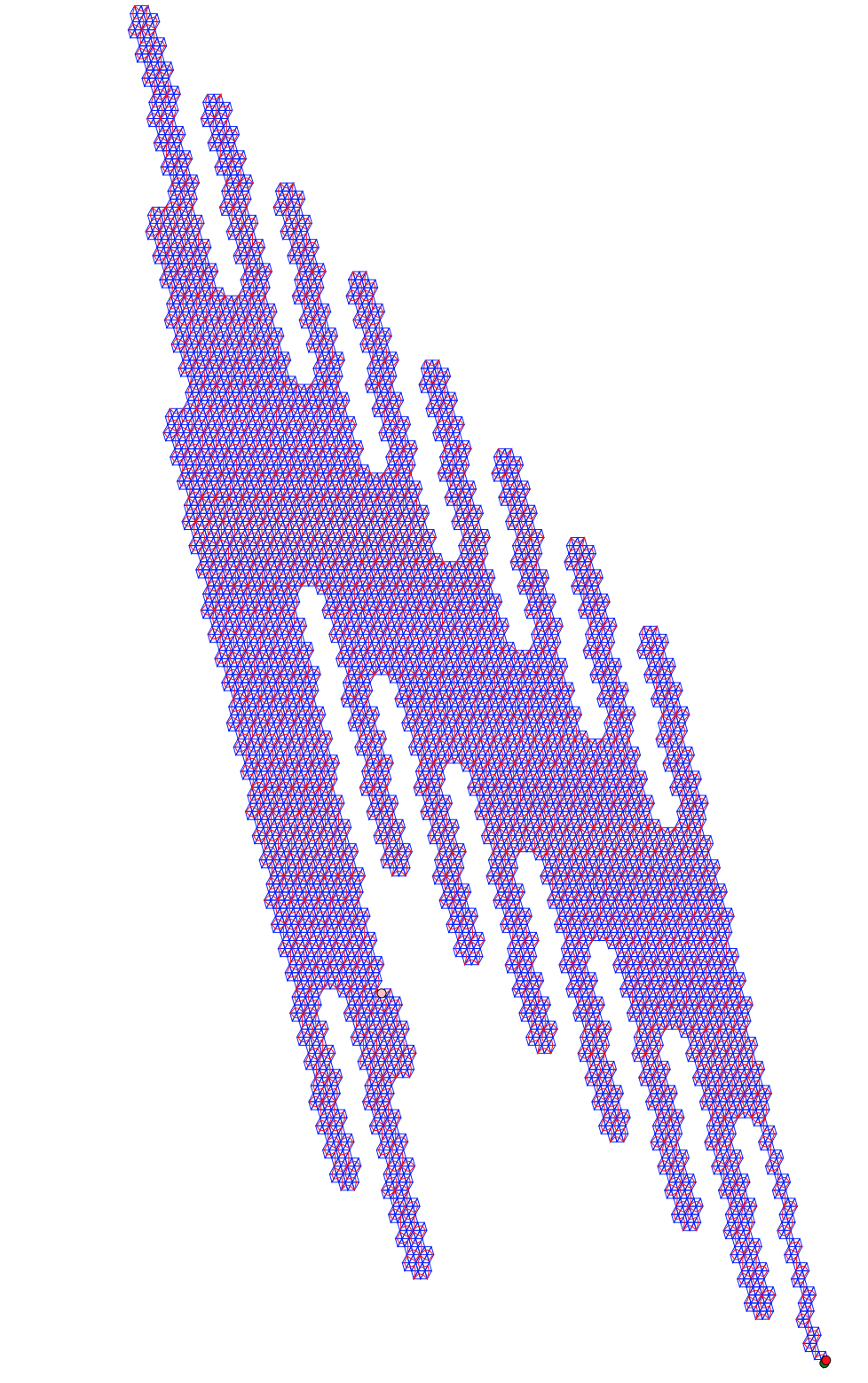}
\caption[]{Examples of trajectories in triangle tiling billiards. From left to right, from up to down: a periodic trajectory of period $10$; a periodic trajectory of period $114$; a drift-periodic trajectory (having a translation symmetry) of drift period $6$ in the $30-60-90$-triangle; linear escaping trajectory; another linear escaping trajectory (no translation symmetry, coding of crossed sides corresponds to a sturmian sequence, and a trajectory stays in a bounded distance from some fixed line in $\mathbb{R}^2$ with an irrational slope); a non-linearly escaping trajectory which spirals out to infinity, the first $15000$ segments of the trajectory are drawn. These pictures are drawn by the program \cite{HSL} authored by P. Hooper and A. St Laurent, accessible on-line. Remark: finding the last trajectory by randomly adjusting the parameters on the computer is impossible - to find it, we had first to know how to search for it, see Theorem \ref{thm:when_escape} first.}\label{figure: different behaviors}
\end{figure}

This proposition gives an understanding of relationships between the first three dynamical systems in our list. For a fixed triangle $\Delta$, the dynamics of the system \ref{def:reflection_in_a_circumcircle:2} for the same $\Delta$ and for all $\tau$ depicts a general dynamics of the triangle tiling billiard \ref{def:triangle_tiling_billiards_1}. The parameter $\tau(l)$ which is defining the chord $l$ in the system \ref{def:reflection_in_a_circumcircle:2} is the invariant of triangle tiling billiard trajectories, with the relationship $d(\delta)=|\cos \left( \pi \tau(l) \right)\vert$. For example, understanding the trajectories passing through the circumcenters of the crossed triangles ($d(\delta)=0$) is equivalent to the understanding of the system \ref{def:reflection_in_a_circumcircle:2} for $l$ being a diameter ($\tau=\frac{1}{2}$). The passage from the first system to the second is made by folding the triangles of the tiling along the trajectory $\delta$.
The third system is the same as the second: the connection can be seen by adding on the Figure \ref{fig:circumreflection} a circle of radius $|\cos \left( \pi \tau(l) \right)\vert$ concentric to an already drawn circle, and to notice that the chord $l$ is tangent to it. In other words, in the system \ref{def:reflection_in_a_circumcircle:2} the chord is fixed and a triangle moves, and in the system \ref{def:triangletangent_system_3} a chord moves while the triangle is fixed. Then, Theorem 3.3 in \cite{BDFI18} shows that the system \ref{def:reflection_in_a_circumcircle:2} written out explicitely as a map of the circle gives a map $F_{\tau, l_1, l_2, l_3} \in \CETthree$ with parameters $l_j$ corresponding to normalized angles and $\tau=\tau(l)$ (at least, where the map from $F_{\Delta, l}$ is defined).

We are interested in the understanding of the behavior of trajectories in triangle tiling billiards. The other three dynamical systems of this Section give different approaches of this same problem. Most of our tools (Sections \ref{sec: modified Rauzy}-- \ref{sec:properties_tiling_billiards}) are applied directly to the systems in \ref{def:fully_flpped_3_interval_exchange_transformations_on_the_circle:4}. 
%Although, some of the recent advancements on the conjectures related to this system (see, e.g. the paragraph \ref{subs:tree} for the discussion of a so-called \emph{Tree conjecture}) were made by the direct use of the approach \ref{def:reflection_in_a_circumcircle:2} of the reflection in a circumcircle. We do think that this approach is probably very fruitful, see more discussion in the end of the Section \ref{sec:perspectives}, especially paragraphs \ref{subs:tree}.

\subsection{Qualitative behavior of orbits in triangle tiling billiards.}\label{subs:different}
As a corollary of Proposition \ref{prop:Dianacomeback}, one obtains a following classification (that we will use as well as definition) of different orbit behavior in a triangle tiling billiard.

\begin{definition}[\emph{\textbf{Types of trajectories of triangle tiling billiards}}]\label{def:different_types_of_trajectories}
Any trajectory of a triangle tiling billiard is one of these three types. 
\begin{itemize}
\item[•] \emph{Closed (periodic) orbits. } A trajectory is closed if it is a closed piecewise linear curve in the plane, without self-intersections.
\item[•] \emph{Drift-periodic orbits.} A trajectory is drift-periodic if it is invariant under a translation of the plane.
\item[•] \emph{Escaping orbits.} A trajectory is escaping if it is not periodic (neither closed nor drift-periodic).
\end{itemize}
\end{definition}

See Figure \ref{figure: different behaviors} for some examples of trajectories. 

Here are some remarks on the qualitative behavior that were clear from the previous work on triangle tiling billiards \cite{DDRSL16, BDFI18}. First, drift-periodic orbits occur only when $\Delta$ has rationally dependent angles (i.e. the ratios of the angles $\frac{\alpha}{\beta}, \frac{\beta}{\gamma} \in \Q$). Even more, if the angles are rationally dependent, any orbit is periodic (closed or drift-periodic), see \cite{BDFI18}. Let us note that both closed and drift-periodic orbits correspond to the periodic orbits of the system \ref{def:reflection_in_a_circumcircle:2}. Second, an important corollary of the "folding observation" (see paragraph \ref{subs:connections}) is that closed orbits come in open families. If $\delta$ is a closed trajectory then for a small enough perturbation of a point $(\delta, \Delta)$ the obtained trajectory is still closed. Indeed, an orbit close to a periodic one, by continuity, will continue the path in the same sequence of folded triangles (for a more quantitative explanation, see Theorem 4.2 in \cite{BDFI18}). Third, escaping orbits happen to have two types of qualitatively different behaviors (and these two behaviors do occur). We will distinguish between \emph{linearly escaping orbits} and \emph{non-linearly escaping orbits}. 

\begin{definition}[\emph{\textbf{Escaping trajectories}}]
A trajectory of a tiling billiard is \emph{linearly escaping} if it stays at a bounded distance from some fixed line in the plane and is not drift-periodic. An escaping trajectory which is not linearly escaping, is \emph{non-linearly escaping}. 
\end{definition}

As we will show in this article, linearly escaping orbits come in sets of positive measure in the space $\{(\delta, \Delta)\}$ of triangle tiling billiard trajectories,  although non-linearly escaping orbits are truly exceptional (and correspond to a zero measure set in this space). The first examples of non-linearly escaping orbits was given in \cite{BDFI18} as a corollary of the noticed connection between triangle tiling billiards and Arnoux-Yoccoz $6$-IET on the circle. These orbits were constructed as passing through the circumcenters of the triangles in the tiling corresponding to $\Delta$ with angles $\pi \frac{1-x}{2}, \pi \frac{1-x^2}{2}$ and $\pi \frac{1-x^3}{2}$, where $x+x^2+x^3=1, x \in \R$ is the Tribonacci number. Our main goal in the following is to show that\emph{ all} of the non-linearly escaping trajectories are passing through circumcenters of the crossed triangles in the tiling. And even more, the triangles $\Delta$ that permit (as forms of tiles) the existence of non-escaping trajectories, are parametrized by the Rauzy gasket $\boldsymbol{\mathcal{R}}$(see its Definition below, Definition \ref{defn:Rauzy_gasket}). This proves the Conjecture 4.19 from \cite{BDFI18}.

\section{Modified Rauzy induction and its Rauzy graphs.}\label{sec: modified Rauzy}
The Rauzy induction is a powerful tool in the study of interval exchange transformations. It was introduced in \cite{R79} by Rauzy. For the study of IETs with flips the standard Rauzy induction procedure has to be modified. This was first done by Nogueira in \cite{N89}. The first two paragraphs of this Section don't contain any new results but only introduces the core objects. Our presentation follows the notations analogue to those chosen by Delecroix in his \emph{Lecture notes on IETs} \cite{D16}, where he in his turn follows \cite{MMY10}. We adjust the notations in order to present the modified Rauzy induction for interval exchange transformations with flips.

\subsection{Interval exchange transformations with and without flips and their dynamics.}\label{subs:IETS_with_without}
From now on, we denote $\boldsymbol{\bar{1}}:=(1, \ldots, 1) \in \R^n$.

\begin{definition}[\textbf{\emph{Interval exchange transformations}}]\label{def:IET}
\emph{An interval exchange transformation (IET)} of the interval $[a,b)$ is a bijection $T: [a,b) \rightarrow [a,b)$ such that there exist the points $a=a_0<a_1< \ldots < a_n=b$ on the interval and the numbers $t_1, \ldots t_n \in \R$ (\emph{shifts}) such that $\left.T\right|_{[a_{i-1},a_i)} (x) = x+t_i, 1 \leq i \leq n$. The set of all such transformations is denoted by $\IETn [a,b)$. We write simply $\IETn$ for the set of all IETs on all intervals.
\end{definition}
\begin{definition}[\textbf{\emph{Interval exchange transformations with flips}}]\label{def:IETF}
Fix the points $a=a_0<a_1< \ldots < a_n=b$, the numbers $t_1, \ldots t_n \in \R$ as well as the vector $\textbf{k}=(k_1, \ldots , k_n)$ with $k_i \in \{-1,1\}$.
Then define a map  $F: [a,b) \rightarrow [a,b)$ in a following way: 
\begin{equation*}
\left.F\right|_{[a_{i-1},a_i)} (x) = k_i x+t_i, 1 \leq i \leq n.
\end{equation*}
If $F: [a,b) \rightarrow [a,b)$ is a bijection between the sets $[a,b) \setminus \cup_i \{a_i\}$ and $[a,b) \setminus \cup_i F(a_i)$ and if $\boldsymbol{k} \neq\boldsymbol{\bar{1}}$, then $F$ is \emph{an interval exchange transformation with flips (IETF)}. If $k_i=-1$ (or $1$), then we say that the interval $[a_{i-1},a_i)$ is \emph{flipped} (or \emph{not flipped}). 

The set of all such transformations is denoted by $\IETFn[a,b)$ or, in the case when the interval is not specified, simply $\IETFn$. The vector $\textbf{k}$ is called a \emph{vector of flips.}
\end{definition}

Note that the Definition \ref{def:IET} is a part of Definition \ref{def:IETF} when $\textbf{k} = \boldsymbol{\bar{1}}$.
\begin{definition}[\emph{\textbf{Fully flipped interval exchange transformations}}]\label{def:FETT}
We say that $F \in \IETFn[a,b)$ is a \emph{fully flipped interval exchange transformation (fully flipped IET)} if $\textbf{k}=- \boldsymbol{\bar{1}}.$ We denote the set of all such transformations by $\FETn[a,b)$, or simply $\FETn$ for the set of all fully flipped IETs.
\end{definition}

\begin{definition}[\emph{\textbf{Associating permutation to an interval exchange transformation (with or without flips)}}]\label{def:permutation}
Any map $F \in \IETn [a,b) \cup \IETFn[a,b) $ is well defined as a bijection between the sets $[a,b)  \setminus \cup_i \{a_i\}$ and $[a,b)  \setminus \cup_i \{F(a_i)\}$. Define a labeling map $\mathcal{L}: \{ [a_{i}, a_{i+1}), i=0, \ldots, n-1\} \rightarrow \{1, \ldots, n\}$ such that $\mathcal{L} ([a_i, a_{i+1}))=i+1$. Then, one associates to $F$ a permutation $\sigma \in S_n$ in a natural way by reading off the labels of the intervals in the set $[a,b) \setminus \cup_i \{F(a_i)\}$ in the order of their positions on the interval $[a,b) $. 
A standard graphic representation of a permutation is a $2 \times n$ matrix with a first line containing the numbers $1, \ldots, n$ in a natural order. For a map $F \in \IETFn [a,b) $ and a permutation corresponding to it, we will add additional information to this graphic representation that is a graphic representation of the vector of flips $\textbf{k}$ by drawing bars over the labels which correspond to the intervals which are flipped, in the first as well as in the second row. In this way, a permutation corresponding to a map in $\FETn [a,b) $ has a graphic representation with bars over all the entries.
\end{definition}

Any map $F \in \IETFn$ defines the triple $(\sigma, \textbf{k}, \{\lambda_i\}_{i=1}^n)$ with $\lambda_i=a_i-a_{i-1}, i=1, \ldots, n$ and vice-versa : the combinatorial information and the lengths of the intervals define an interval exchange transformation with flips.

\smallskip

The three definitions of classes of the maps on the interval given above can be generalized in a natural way to the $n$-interval exchanges on the circle (by replacing $[a,b) $ with $\Sph^1$ and identifying $a_0=a_n$). We denote the obtained classes correspondingly as $\mathrm{CET}^n$ (Definition \ref{def:IET}), $\mathrm{CETF}^n$ (Definition \ref{def:IETF}) and $\mathrm{FCET}^n$ (Definition \ref{def:FETT}). The results of this work mostly concern a subfamily of the set $\mathrm{FCET}^n$ of fully flipped interval exchange transformations on the circle that has already been defined for $n=3$ in the previous Section (see Definition \ref{def:fully_flpped_3_interval_exchange_transformations_on_the_circle:4}) in relation to triangle tiling billiards, and that we define in full generality now.

\begin{definition}\label{def:CETCET}
Fix $l_1, l_2, \ldots, l_n \in  \R_{>0}$ such that $l_1+\ldots+l_n=1$. Take $\Sph^1=\R / \Z$ and define $a_0, a_1, \ldots, a_n \in \Sph^1$ as $a_0=0, a_1=l_1, a_2=l_1+l_2, \ldots, a_i=l_1+\ldots+l_i, \ldots, a_{n-1}=1-l_n, a_n=a_0$. 
A map $F\in \mathrm{FCET}^n [0,1)$ is said to belong to the set $\CETn$ if there exists $\tau \in \Sph^1$ such that
\begin{equation*}
F(x)=      -x+l_i +\tau   \;  \; \mathrm{mod}\; 1 \; \; \forall x \in [a_{i-1}, a_i) \; \; \forall i =1, \ldots n.
\end{equation*}
\end{definition}

\begin{remark}\label{remark:CET_as_IET_plus_one}
As we defined the circle interval exchange transformations by demanding $a_0=a_n$, the circle has a marked point $a_0=a_n$. Hence, the sets
 $\mathrm{CET}^n$, $\mathrm{CETF}^n$, $\mathrm{FCET}^n$ (and $\CETn$) are in natural bijection with the subsets of, correspondingly, $\mathrm{IET}^{n+1}$, $\IETFplus$ and $\FETplus$. Indeed, these bijections are obtained by cutting the interval $[a_i, a_{i+1})$ such that $a_0 \in F(a_i, a_{i+1})$, into two intervals. So one can also define corresponding permutations for the maps on the circle, by applying the Definition \ref{def:permutation} to the images of the maps on the circles in the set of transformations of the interval. The corresponding permutation belongs to $S_{n+1}$ and not $S_n$ (which occurs only if $a_0 \in \cup_i F(a_i)$), see Figure \ref{fig:defdif}. 
\end{remark}

\begin{remark}
For the interval exchange transformations on a circle without a marked point, the definition of a corresponding permutation is problematic. Although, we want to point out that the set $\CETn$ of fully flipped interval exchange transformations on the circle correspond to a "trivial permutation"(when the circle is seen as a circle without a marked point) of the indices of intervals, and hence defined by one parameter which is $\tau \in \Sph^1$. But after a cut, a permutation of a corresponding map in $\FETplus$ is no longer a trivial permutation, see Figure \ref{fig:defdif}.
\end{remark}

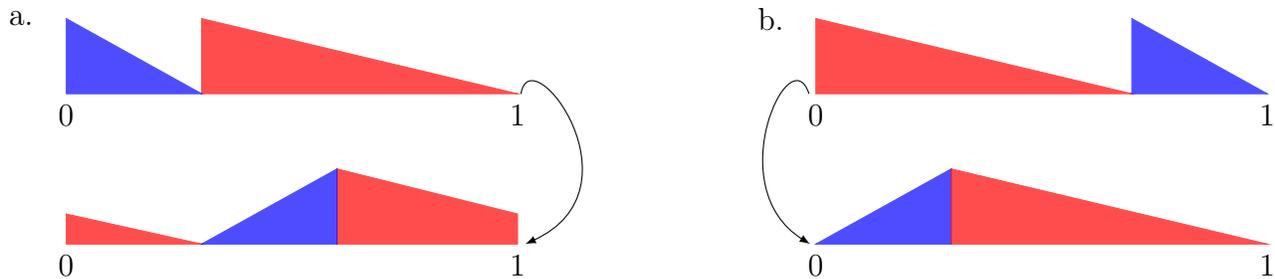
\begin{figure}
\centering
\begin{tikzpicture}[xscale=6]
\path[draw](0,-2)  node[below] {$0$} --cycle;
\path[draw](1,-2) node[below] (2) {$1$} --cycle;
\path[draw, fill=red,red, opacity=0.7] (.3,0)--(.3,1)--(1,0)--cycle;
\path[draw, fill=blue, blue, opacity=0.7] (0,0)--(.3,0)--(0,1)--cycle;
\path[draw] (0,0) node[below]{$0$} --cycle;
\path[draw] (1,0) node[below] (1) {$1$} --cycle;
\path[draw, fill=red,red, opacity=0.7] (0,-2) --(0.3,-2)--(0,-1.6)--cycle;
\path[draw, fill=red,red, opacity=0.7] (0.6,-2)--(0.6,-1)--(1,-1.6)--(1,-2) --cycle;
\path[draw, fill=blue,blue, opacity=0.7] (.3,-2)--(.6,-2)--(0.6,-1)--cycle;
\node at (-0.1,1) {a.};
 \tikzset{
        arrow/.style={
            color=black,
            draw=black,
            -latex,
                font=\fontsize{13}{13}\selectfont},  
        }

\draw[arrow](1) to [out=88,in=70]  (2);
\end{tikzpicture}
\qquad
\begin{tikzpicture}[xscale=6]
\path[draw, fill=red,red, opacity=0.7] (0,0)--(0,1)--(0.7,0)--cycle;
\path[draw, fill=blue, blue, opacity=0.7] (.7,0)--(.7,1)--(1,0)--cycle;
\path[draw] (0,0) node[below] (1){$0$} --cycle;
\path[draw] (1,0) node[below]{$1$} --cycle;
\path[draw] (0,-2) node[below] (2) {$0$} --cycle;
\path[draw] (1,-2) node[below] {$1$} --cycle;
\path[draw, fill=blue,blue, opacity=0.7] (0,-2)--(0.3,-2)--(.3,-1)--cycle;
\path[draw, fill=red,red, opacity=0.7] (.3,-2)--(.3,-1)--(1,-2)--cycle;
\node at (-0.1,1) {b.};

 \tikzset{
        arrow/.style={
            color=black,
            draw=black,
            -latex,
                font=\fontsize{12}{12}\selectfont},
        }

\draw[arrow](1) to [out=94,in=105]  (2);

\end{tikzpicture}

\caption[]{For a fixed $F\in \mathrm{CETF}^2$ on the circle represented here, after the marking $a_0=a_n=0 \in \Sph^1$ and passing to an IETF on the interval, three cases can occur. First, as in a., when $0$ belongs to the image of a non-flipped interval, one obtains a map $F_1 \in \mathrm{IETF}^3 [0,1)$ with the combinatorics
$\begin{pmatrix} \bar{1} & 2 & 3\\
3 & \bar{1} & 2
\end{pmatrix}$. Second, when $0$ belongs to the image of a flipped interval, the obtained map  $F_2 \in \mathrm{IETF}^3 [0,1)$ has the combinatorics
$\begin{pmatrix} \bar{1} & \bar{2} & 3\\
\bar{1} & 3 & \bar{2}
\end{pmatrix}$. Third, as in b., when $0$ coincides with the image of one of the singularities, a corresponding map $F_3 \in \mathrm{IETF}^2 [0,1)$, with combinatorics $\begin{pmatrix} \bar{1} & 2 \\
\bar{1} & 2
\end{pmatrix}$ or $\begin{pmatrix} \bar{1} & 2 \\
2 & \bar{1}
\end{pmatrix}$.
} \label{fig:defdif}
\end{figure}

The dynamics of the maps in $\IETFn$ is very different from the dynamics of the maps in $\mathrm{IET}^n$. Indeed, $\mathrm{IET}^2$ is a set of rotations of the circle (almost all of which - irrational rotations - are minimal), although any map in the set of $\mathrm{IETF}^2$ is completely periodic \cite{K75}. From the first sight, this fact seems quite surprising, as seems the answer to a following problem.

\begin{problem}[\emph{\textbf{Cake problem}}]
Take a chocolate cake with ice frosting on top. Cut out a slice of the cake of size $\alpha$, and put it back in place but with a flip. Then, the chocolate comes to the surface and the frosting is hidden on the lower level of the cake. Now reiterate the procedure, starting from the place where the cake was cut last time.
Question: for what values of angle $\alpha$, after some finite number of iterations, all of the frosting will come back ? 
%See Figure \ref{fig:movie}.
\end{problem}

\begin{answer}
For \emph{\textbf{all}} $\alpha$.
\end{answer}

%\begin{figure}
%\includegraphics[scale=0.72]{cake_storyboard.pdf}
%\caption{Here is a four-step iterative recipe of the cutting of the cake mentionned in the cake problem. First, starting from the point where the knife is placed, cut a piece of cake of size $\alpha$. Second, take this piece of a cake on a spatula and lift it up. Third, flip and put a piece of cake back onto its place but change its orientation (the lower layer goes on top and the top layer goes on bottom). Fourth, put a knife at the end of the piece (clock-wise) and reiterate. These illustrations are snapshots from a short movie\cite{movie}, co-created by C. Gourdon and the second author.}\label{fig:movie}
%\end{figure}

The second author made a short movie abouth this problem, co-created with C. Gourdon \cite{GPR} . The cake problem is, of course, just a reformulation of the fact that a map from $\mathrm{IETF}^2$ presented on Figure \ref{fig:defdif} is completely periodic, see Proposition \ref{prop:keane} in the following.  The analogue of this periodicity property holds in the family $\IETFn$ even for $n>2$.

\begin{theorem}[\cite{N89}]\label{thm:Nogueira}
Fix a permutation $\sigma \in S_n$ and a vector $\textbf{k} \in \{-1,1\}^n, \textbf{k} \neq \boldsymbol{\bar{1}}$. Then for almost any choice of points $a_0=0, a_1, \ldots, a_n=1$ (with respect to Lebesgue measure), for the map $F \in \IETFn$ defined by $\sigma, \textbf{k}$ and $\{a_i\}_{i=0}^n$ (as in Definition \ref{def:IETF}) there exists $N\in \N^*$ and $I \subset [0,1)$ such that the first return map of $F$ in restriction to $I$ coincides with $F^N$ and is a flip of the interval $I$ onto itself.
\end{theorem}

\begin{remark} The conclusion of this theorem was recently sharpened by Skripchenko and Troubetzkoy who proved that the set of parameters with minimal dynamics  does not have maximal Hausdorff dimension \cite{ST18}. 

\end{remark}

\subsection{Modified Rauzy induction for interval exchange transformations with flips: definition and notations.}\label{subs:MRI}
The proof of Theorem \ref{thm:Nogueira} is based on the tool of modified Rauzy induction that we introduce in this paragraph, and that we use extensively throughout the article. 

When in the previous paragraph the permutations characterizing the IETs had a first row fixed and trivial (see Definition \ref{def:permutation}), in this paragraph we will label the intervals by letters and not by numbers, and we will allow any order on the top and bottom line of the permutation. More formally, now let $\A$ be an alphabet of cardinality $n$. Let the maps $\sigma^{\mathrm{top}}$ and $\sigma^{\mathrm{bot}}$ be two orders on the alphabet $\A$. 

\begin{definition}
A \emph{generalized permutation} on the alphabet $\A$ corresponding to two bijections $\sigma^{\mathrm{top}}, \sigma^{\mathrm{bot}}: \A \rightarrow \{1, \ldots, n\}$ is 
defined as
\begin{equation}\label{eq:matrix_representation_gen_permutation}
\sigma=\begin{pmatrix}
\left(\sigma^{\ttop}\right)^{-1}(1)&\ldots&\left(\sigma^{\ttop}\right)^{-1}(n) \\
\left(\sigma^{\bbot}\right)^{-1}(1)&\ldots&\left(\sigma^{\bbot}\right)^{-1}(n) 
\end{pmatrix}.
\end{equation}
We denote the set of all generalized permutations on the alphabet $\A$ with $n$ elements by $S_n^{\A}$.
\end{definition}

Now let $\boldsymbol{\lambda}:=(\lambda_i)_{i \in \A}$ be a vector of positive real numbers, with $\lambda=|\boldsymbol{\lambda}|:=\sum_{i \in \A} \lambda_i$. Suppose $\boldsymbol{k}=(k_i)_{i \in \A}$, $k_i \in \{-1,1\}$. Then, one associates a map $F \in \mathrm{IET}^n [0, \lambda) \cup \IETFn [0, \lambda) $ to the data $\{\sigma, \textbf{k}, \boldsymbol{\lambda}\}$ in a following way. First, set for every $j \in \{0, \ldots, n\}$ the following quantities:

\begin{equation*}
\alpha_j^{\mathrm{top}}:=\sum_{i: \sigma^{\mathrm{top}}(i) \leq j} \lambda_i, \;\;\; \alpha_j^{\mathrm{bot}}:=\sum_{i: \sigma^{\mathrm{bot}}(i) \leq j} \lambda_i.
\end{equation*}

These quantities define the points that give two partitions of $[0, \lambda)$. Namely, we set for each $i \in \A$: 

\begin{equation*}
I_i^{\mathrm{top}}:=\left(
 \alpha^{\mathrm{top}}_{\sigma^{\mathrm{top}}(i)-1}, \alpha^{\mathrm{top}}_{\sigma^{\mathrm{top}}(i)} 
 \right), \;\;\; I_i^{\mathrm{bot}}:=\left(
 \alpha^{\mathrm{bot}}_{\sigma^{\mathrm{bot}}(i)-1}, \alpha^{\mathrm{bot}}_{\sigma^{\mathrm{bot}}(i)} 
 \right).
\end{equation*}

Define $F$ in restriction to each interval $I_i^{\ttop}$ by first, $F\left(I_i^{\ttop}\right)=I_i^{\bbot}$ and moreover, if $k_i=1$, it is a translation; and otherwise it is a translation in composition with a flip. This defines $F$ completely on $[0, \lambda)$ outside the extremities of the intervals $I_i^{\ttop}, i \in \{1, \ldots, n \}$. We call $\alpha_i^{\mathrm{top}}$ top singularities and $\alpha_i^{\mathrm{bot}}$ bottom singularities of $F$, $i=1, \ldots, n-1$. Any map $F \in \mathrm{IET}^n [0, \lambda) \cup \IETFn [0, \lambda)$ is represented in such a way, and such a representation is unique (up to the re-labelling of the first row in $\sigma$). In the following, we will identify $F \in \IETFn [0, \lambda)$ with the corresponding data $\{\sigma, \boldsymbol{k}, \boldsymbol{\lambda}\}$. The conventions of graphic representation are transmitted from the previous paragraph, see Figure \ref{fig:defintervalexchange} for illustration by example. 

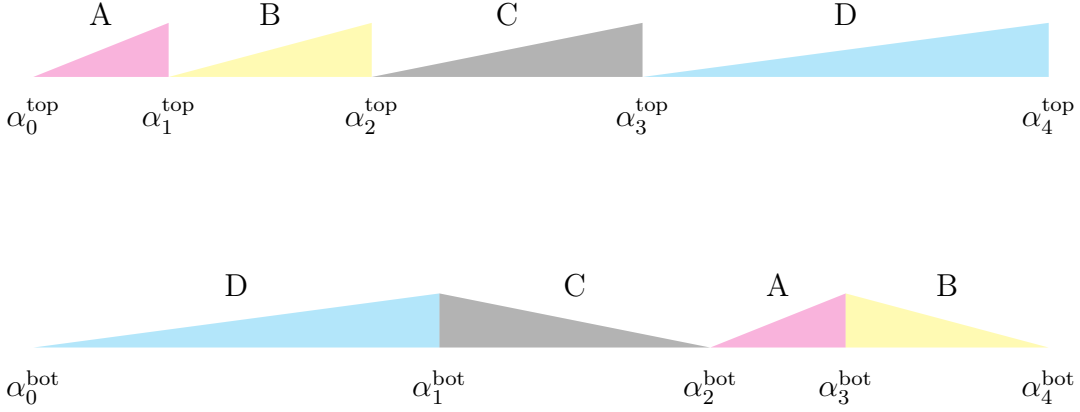
\begin{figure}
\centering
\begin{tikzpicture}[scale=1.8]

\path[fill=magenta, opacity=0.3] (0,0)--(1,0)--(1,0.4)--cycle;
\path[fill=yellow, opacity=0.3] (1,0)--(2.5,0)--(2.5,0.4)--cycle;
\path[fill=black, opacity=0.3] (2.5,0)--(4.5,0)--(4.5,0.4)--cycle;
\path[fill=cyan, opacity=0.3] (4.5,0)--(7.5,0)--(7.5,0.4)--cycle;

\path[fill=cyan, opacity=0.3] (0,-2)--(3,-2)--(3,-1.6)--cycle;
\path[fill=black, opacity=0.3] (3,-1.6)--(3,-2)--(5,-2)--cycle;
\path[fill=magenta, opacity=0.3] (5,-2)--(6,-2)--(6,-1.6)--cycle;
\path[fill=yellow, opacity=0.3] (6,-1.6)--(6,-2)--(7.5,-2)--cycle;

\draw(0.5, 0.3) node[above]{A};
\draw(1.75, 0.3) node[above]{B};
\draw(3.5, 0.3) node[above]{C};
\draw(6, 0.3) node[above]{D};
\draw(0, -0.5) node[above]{$\alpha_0^{\textrm{top}}$};
\draw(1, -0.5) node[above]{$\alpha_1^{\textrm{top}}$};
\draw(2.5, -0.5) node[above]{$\alpha_2^{\textrm{top}}$};
\draw(4.5, -0.5) node[above]{$\alpha_3^{\textrm{top}}$};
\draw(7.5, -0.5) node[above]{$\alpha_4^{\textrm{top}}$};

\draw(1.5, -1.7) node[above]{D};
\draw(4, -1.7) node[above]{C};
\draw(5.5, -1.7) node[above]{A};
\draw(6.75, -1.7) node[above]{B};
\draw(0, -2.5) node[above]{$\alpha_0^{\textrm{bot}}$};
\draw(3, -2.5) node[above]{$\alpha_1^{\textrm{bot}}$};
\draw(5, -2.5) node[above]{$\alpha_2^{\textrm{bot}}$};
\draw(6, -2.5) node[above]{$\alpha_3^{\textrm{bot}}$};
\draw(7.5, -2.5) node[above]{$\alpha_4^{\textrm{bot}}$};

\end{tikzpicture}

\caption[]{For the data $(\sigma, \boldsymbol{\lambda}, \boldsymbol{k})$, the interval exchange transformation with flips is constructed, the preimage is cut according to the partition $\left(I_i^{\mathrm{top}}\right)$ while the image - according to $\left(I_i^{\mathrm{bot}}\right)$. Here $(\sigma, \boldsymbol{k})= \begin{pmatrix} A & \bar{B} & \bar{C} & D\\
D & \bar{C} & A & \bar{B}
\end{pmatrix}$. }
\label{fig:defintervalexchange}
\end{figure}

\begin{definition}\label{def:irreducible}
A generalized permutation $\sigma$ is called \emph{reducible} if $\exists l, l\in \{1, \ldots, n-1 \}$ such that
\begin{equation*}
\left(\sigma^{\mathrm{top}}\right)^{-1}\left(\left\{1,2, \ldots, l\right\}\right) = \left(\sigma^{\mathrm{bot}}\right)^{-1}\left(\left\{1,2, \ldots, l\right\}\right).
\end{equation*}
A corresponding $F \in \mathrm{IET}^n \cup \IETFn$ is in this case called \emph{reduced}. If $\sigma$ is not reducible, it is called \emph{irreducible.}
\end{definition}

Automatically, reduced IETs and IETFs are not minimal.
\bigskip

Now we are ready to define the modified Rauzy induction. It is an algorithm that either associates to a map $F \in \IETFn$ another map $\mathcal{R} F \in \IETFn$ (defined on a smaller interval), or stops. One can look at the modified Rauzy induction as a map 
\begin{equation*}
\mathcal{R}: \IETFn \rightarrow \IETFn \cup \{\square\},
\end{equation*}
where $\square$ correspond to a stop of Rauzy induction. 

\begin{definition}[\emph{\textbf{Modified Rauzy induction}}]\label{def:MRI}
For $F \in \IETFn [0, \lambda)$, if $\alpha_{n-1}^{\mathrm{top}}=\alpha_{n-1}^{\mathrm{bot}}$ then $\mathcal{R} F:=\square$. Otherwise, the longest interval between
$I^{\mathrm{top}}_{(\sigma^{\mathrm{top}})^{-1}(n)}$ and $I^{\mathrm{bot}}_{(\sigma^{\mathrm{bot}})^{-1}(n)}$ is called the \emph{winner} and the shortest one the \emph{loser}. Then $\mathcal{R} F$ is a first return map of $F$ on the interval $\left[0, \max \left(\alpha_{n-1}^{\mathrm{top}},\alpha_{n-1}^{\mathrm{bot}}\right)\right]$, i.e. $\mathcal{R}F$ is defined on the interval of the length different from the length of the initial interval by the length of the loser. In the case when the winner is on top (on bottom), the step of Rauzy induction is called\emph{ top (bottom) induction.} 
\end{definition}

When $F'=\mathcal{R} F$ exists, it is an interval exchange transformation with flips, and its data 
\begin{equation*}
(\sigma', \boldsymbol{k}', \boldsymbol{\lambda}') \in S_n^{}\mathcal{A} \times \left( \{-1,1\}^n \setminus \one \right) \times \R^n_+
\end{equation*}
can be deduced from the analogical data $(\sigma, \boldsymbol{k}, \boldsymbol{\lambda})$ of $F$. In the case of IETs with flips, one step of Rauzy induction has a bigger number of combinatorial possibilities than one case of IETs without flips. Indeed, not only the lengths of the winner and loser are taken into account but also the values $k_{(\sigma^{\mathrm{top}})^{-1}(n)}$ and $k_{(\sigma^{\mathrm{bot}})^{-1}(n)}$ of the vector $\textbf{k}$ (that take values in the set $\{-1,1\}$). This gives indeed $8$ possibilities (instead of $2$ for standard Rauzy induction for IETs), see Table \ref{table:how_works}.

\begin{remark}
An important remark is that $\sigma'$ is defined on the same set of labels as is $\sigma$, and the labels change in a natural way (analogous to the standard Rauzy induction procedure). The winner is cut into two non-empty intervals by one of the extremities of the loser. The left of this interval is relabeled into the label of the winner as well as its image (if the winner was on top) or pre-image (if the winner was on bottom). The label of the loser is assigned to the part of the initial winner interval which doesn't have a label in one of the lines (on top or on bottom, if the winner was on bottom or on top, correspondingly). In another line the label of the interval which was a loser doesn't change. This remark can be considered as a part of Definition \ref{def:MRI}.
\end{remark}

\begin{table}
\setlength{\tabcolsep}{2mm} % separator between columns
\def\arraystretch{1.25} % vertical stretch factor
\centering

  \begin{tabular}{cc|c}
  %    \hline

   $(k_B, k_D)$ & {top induction, case $\lambda_D>\lambda_B$   }&   bottom induction, case $\lambda_D<\lambda_B$   \\ %\hline

&$\lambda_D'=\lambda_D-\lambda_B$ & $\lambda_B'=\lambda_B-\lambda_D$ 
\\
$(1,1)$
& $\begin{pmatrix}
\bar{A}&B&\bar{C}&D\\
D&\bar{A}&\bar{C}&B
\end{pmatrix} 

\rightarrow
\begin{pmatrix}
\bar{A}&B&\bar{C}&D\\
D&B&\bar{A}&\bar{C}
\end{pmatrix}$     
        
     &    $\begin{pmatrix}
\bar{A}&B&\bar{C}&D\\
D&\bar{A}&\bar{C}&B
\end{pmatrix} \rightarrow 
\begin{pmatrix}
\bar{A}&B&D&\bar{C}\\
D&\bar{A}&\bar{C}&B
\end{pmatrix}$  

    \\ [30pt]

$(1, -1)$
& $\begin{pmatrix}
\bar{A}&B&\bar{C}&\bar{D}\\
\bar{D}&\bar{A}&\bar{C}&B
\end{pmatrix} 

\rightarrow
\begin{pmatrix}
\bar{A}&\bar{B}&\bar{C}&D\\
\bar{B}&\bar{D}&\bar{A}&\bar{C}
\end{pmatrix}$

     &    $\begin{pmatrix}
\bar{A}&B&\bar{C}&\bar{D}\\
\bar{D}&\bar{A}&\bar{C}&B
\end{pmatrix}  \rightarrow 
\begin{pmatrix}
\bar{A}&B&\bar{D}&\bar{C}\\
D&\bar{A}&\bar{C}&B
\end{pmatrix}$  

   \\[30pt]
   % \hline

$(-1, 1)$
& $\begin{pmatrix}
\bar{A}&\bar{B}&\bar{C}&D\\
D&\bar{A}&\bar{C}&\bar{B}
\end{pmatrix} 

\rightarrow
\begin{pmatrix}
\bar{A}&\bar{B}&\bar{C}&D\\
{D}&\bar{B}&\bar{A}&\bar{C}
\end{pmatrix}$

     &    $\begin{pmatrix}
\bar{A}&\bar{B}&\bar{C}&D\\
D&\bar{A}&\bar{C}&\bar{B}
\end{pmatrix} 
 \rightarrow 
\begin{pmatrix}
\bar{A}&\bar{D}&\bar{B}&\bar{C}\\
\bar{D}&\bar{A}&\bar{C}&\bar{B}
\end{pmatrix}$  
 \\ [30pt]

{$(-1, -1)$
}& $\begin{pmatrix}
\bar{A}&\bar{B}&\bar{C}&\bar{D}\\
\bar{D}&\bar{A}&\bar{C}&\bar{B}
\end{pmatrix} 

\rightarrow
\begin{pmatrix}
\bar{A}&B&\bar{C}&\bar{D}\\
B&\bar{D}&\bar{A}&\bar{C}
\end{pmatrix}$

     &  $\begin{pmatrix}
\bar{A}&\bar{B}&\bar{C}&\bar{D}\\
\bar{D}&\bar{A}&\bar{C}&\bar{B}
\end{pmatrix} 
 \rightarrow 
\begin{pmatrix}
\bar{A}&D&\bar{B}&\bar{C}\\
D&\bar{A}&\bar{C}&\bar{B}
\end{pmatrix}$

  \end{tabular}
  \caption{Eight combinatorial cases of Rauzy induction for IETFs depending on the relation between the lengths of two "end intervals" but also on the fact which of these two intervals are flipped. The cases are represented on examples of $F \in \IETFfour$. The data $(\sigma, \boldsymbol{k}', \boldsymbol{\lambda}')$ of $F'=\mathcal{R} F$ is decribed. The lenghts $\lambda_i'=\lambda_i$ for all $i$, except that of the winner. The winner's length is diminished by the loser's length. The permutation and the vector of flips change as shown on the examples of this diagram: if the winner is not flipped, the combinatorics of $\sigma'$ is the same as in the case of standard Rauzy induction. If the winner is flipped, the combinatorics changes. If $\lambda_B=\lambda_D$, $\mathcal{R}F=\square$.}\label{table:how_works}
\end{table}

The Rauzy induction has been used extensively in the last $50$ years as a powerful tool that unites combinatorics and linear algebra, see for example \cite{AR91, AHS16, BL, KZ03, ST18}. Rauzy induction helps to understand the dynamics of IETs, as one can see for example, with this

\begin{lemma}\cite{R79, K75, N89}\label{lemma:minimality}
The map $F \in \IETFn$ is minimal if and only if the modified Rauzy induction never stops, and the vector of the lengths of the intervals $\lambda^{(m)}$ obtained after $m$ iterations of the MRI, tends to zero:
$$
||\lambda^{(m)}||_{\infty}= \max_{i \in \mathcal{A}} \lambda^{(m)}_i \rightarrow_{m \rightarrow \infty} 0.
$$
\end{lemma}

\subsection{Rauzy graphs: constructions, definitions, algorithms.}
Associated to the Rauzy induction on the space of $n$-IETs, one can define a combinatorial object which is called the Rauzy graph. Rauzy graphs were studied in a lot of detail for standard Rauzy induction \cite{KZ03, DM17}. In this section, we define the Rauzy graphs for the modified Rauzy induction in an analogical way.  The proof of our main result (Theorem \ref{thm:intro_main_theorem}) is based on one invariant of modified Rauzy graphs that we stumbled upon, and we believe that many more are to be discovered. The full understanding of modified Rauzy graphs is very far from being achieved.

Fix an alphabet $\A$ with $n$ letters. Denote $Y^X$ the set of maps from $X$ to $Y$, for any two finite sets $X$ and $Y$. Then any vector $\boldsymbol{k} \in \{1,-1\}^n \setminus\boldsymbol{\bar{1}}$ with coordinates labeled by the elements of $\A$, can be seen as an element $\textbf{k} \in {\left(\{-1,1\}^n \setminus \one\right)}^{\A}$.
Also note that $\sigma \in S_n^{\A}$ is nothing else than a pair of bijective maps $\sigma^{\ttop}, \sigma^{\bbot}: \mathcal{A} \rightarrow \{1, \ldots, n\}$.
\begin{definition}
Fix an alphabet $\A$ with $n$ letters. Consider a set of all possibilities of combinatorial data of maps in $\IETFn$:
\begin{equation*}
V:= \left\{(\sigma, \boldsymbol{k}) \left| \right.
\sigma \in S_n^{\mathcal{A}}, \textbf{k} \in {(\{-1,1\}^n \setminus \one)}^{\A}
\right\}.
\end{equation*}
We call this set the \emph{set of Rauzy classes of interval exchange transformations with flips}. Each element of this set is called a \emph{Rauzy class}.
\end{definition}

As we have seen above in Lemma \ref{lemma:minimality}, the minimality of a map $F \in \IETFn$ can be formulated in terms of the modified Rauzy induction. The iterations $\{\mathcal{R}^n F\}$ define a path $(\sigma_n, \boldsymbol{k}_n, \boldsymbol{\lambda}^{(n)})$ in the modified Rauzy graph $G=(E,V)$.  Suppose that this (possibly infinite) path $\gamma$ follows the sequence of edges $e_1, \ldots, e_{m}, \ldots, e_j \in E$. To any edge $e \in E$ one associates a linear \emph{non-negative} matrix $A_e$ corresponding to the inverse transformation of the lengths in a following way. If the edge $e$ corresponds to the induction step where $\lambda^{(m)}_i>\lambda^{(m)}_j, i \neq j$ with $ i,j \in \mathcal{A}$ then $A_e:=E+E_{ij}$. Here $E_{ij}$ is a matrix with all elements equal to zero except one in the $i$-th row and $j$-th column which is equal to $1$. Let $A_{(m)}:=A_{e_1} \ldots A_{e_m}$. Then the lengths of the intervals of continuity for an IET corresponding to the $m$-step of the Rauzy induction are given by $\lambda^{(m)}=A_{(m)}^{-1} \lambda$ since $\lambda^{(m)}= A_{e_m}^{-1} \ldots A_{e_1}^{-1} \lambda$, by definition. We will study these products more attentively in the following Section. 

Now let us give some more combinatorial definitions.

\begin{definition}
The \emph{modified Rauzy graph} or \emph{modified Rauzy diagram} is a finite oriented graph $G=(V,E)$ with the set of vertices being the set of Rauzy classes of IETs. For the edges, $e=(v,w) \in V \times V \in E$ if by one step of Rauzy induction one can pass from the combinatorial data $v$ to the combinatorial data $w$. 
\end{definition}

Since we will be working only with such graphs, we will sometimes omit the world modified, and call these graphs simply Rauzy graphs. Note that such graphs (for different values of $n$) are not necessarily connected. The number of outcoming edges from each vertex belongs to the set $\{0,2\}$ as the number of incoming arrows belongs to the set $\{0,1,2\}$.  Let us define the following

\begin{definition}[{\emph{\textbf{Equivalence relationship on the set of Rauzy classes}}}]\label{def:equivalence_relationship}
Two vertices $v_1,v_2 \in V$ are \emph{equivalent} ($v_1 \sim v_2$) if for $v_1=(\sigma_1, \boldsymbol{k}_1)$ and $v_2=(\sigma_2, \boldsymbol{k}_2)$ there exists a bijective map $h: \A \rightarrow A$ such that
\begin{align*}
\sigma^{\ttop}_1=\sigma^{\ttop}_2 \circ h \\
\sigma^{\bbot}_1=\sigma^{\bbot}_2 \circ h \\
\boldsymbol{k}_1=\boldsymbol{k}_2 \circ h.
\end{align*}
The set of equivalent vertices to a vertex $v_1$ will be denoted $[v_1].$
If one wants to precise the map $h$, one also writes $v_1 \sim_h v_2$.
\end{definition}

Obviously, this is a well defined equivalence relationship. Note that for two vertices $v_1, v_2 \in G$ the bijection $h: \A \rightarrow \A$, if it exists, is uniquely defined. We consider the quotients of the Rauzy graphs with respect to this relationship, defined in a following way.

\begin{definition}
The \emph{quotient Rauzy graph} (or, simply, the \emph{quotient graph}) is a finite oriented graph $G=(V,E)$ with the vertices being the equivalence classes of Rauzy classes of IETF with respect to the equivalence relationship $\sim$. Two vertices $[v_1], [v_2]$ are connected by an edge $e \in G$ if there exist representatives of Rauzy classes $v_1 \in [v_1], v_2 \in [v_2]$ as well as a map $h: \A \rightarrow \A$ such that $v_1 \sim_h v_2$. Moreover, the quotient Rauzy graph comes with a labeling map $\mathcal{L}: E \rightarrow H$, where $H$ is a set of bijections $h: \A \rightarrow \A$. This labeling map is defined as follows: $e \mapsto h$, i.e. the edges of the quotient Rauzy graph are labeled by the permutations of labels.
\end{definition}

The quotient Rauzy graphs have a much smaller number of vertices than Rauzy graphs (indeed, the class $[v_1]$ contains $n!$ Rauzy classes). Although, they carry all the additional information contained in the Rauzy graph in the labelling of the edges between the vertices.

\begin{center}
\textbf{How the pictures of Rauzy graphs are drawn in this article.}
\end{center}

For this project, we collaborated with Paul Mercat, who has written a code in Sage that draws Rauzy graphs and quotient Rauzy graphs. For all the pictures of these graphs, here are our assumptions on the graphical representation that we have chosen. 

\begin{itemize}
\item[A1.] For the sake of the economy of place, we draw quotient Rauzy graphs instead of the Rauzy graphs.
\item[A2.] Sometimes only the connected components of the quotient Rauzy graphs are drawn.
\item[A3.] For the vertices of quotient graphs, each class of equivalence relationship $\sim$ is represented by one of the vertices in this class, i.e by some generalized permutation $\sigma \in S_n^{\A}$. Of course, such a graphic representation of the quotient Rauzy graph is not unique because it depends on the choice of the representatives for each class.
\item[A4.] The bars are put on the letters of $\sigma$ that correspond to flipped intervals. The letters that correspond to the intervals which are not flipped, are represented by a green color. 
\item[A5.] Since we are mostly interested in the cycles in the Rauzy graphs, we do not draw the vertices $[\sigma], \sigma \in S_n^{\A}$ for which  $(\sigma^{\ttop})^{-1}(n)=(\sigma^{\bbot})^{-1}(n)$. Indeed, for any $F \in \IETFn$ represented by such a vertex, obviously, $\mathcal{R} F = \square$.
\item[A6.] On each edge in the quotient Rauzy graph, we write the labels of the winner and the loser for the Rauzy induction on the corresponding element of the equivalence class. The winner and the loser are marked (e.g. $C>D$ for $I_C$ being the winner and $I_D$ being the loser). 
\item[A7.] The arrows in the quotient Rauzy graph are marked by a labeling map $\mathcal{L}$. In our pictures, the arrows marked by the identity map, are drawn in a standard way. Meanwhile, any arrow $e$ marked by  $h=\mathcal{L}(e) \neq \mathrm{id}$ is dotted. Of course, the map $h$ can be reconstructed explicitely if one knows the connected vertices $[v_1], [v_2]$ (and their representatives), as well as the labels of the winner and the loser.
\end{itemize}

\begin{example}
On Figure \ref{pic:graph_for_four} the reader can see the connected component of the permutation 
\begin{equation}\label{eq:basic_combinatorial_type_CET_3}
\sigma=\begin{pmatrix}
\bar{A}&\bar{B}&\bar{C}&\bar{D}\\
\bar{B}&\bar{D}&\bar{A}&\bar{C}
\end{pmatrix}
\end{equation}
in the quotient Rauzy graph for $\IETFfour$. 
After applying the Rauzy induction to $\sigma$, one gets:
\begin{equation*}
\sigma'=\mathcal{R}_{C<D} \sigma=\begin{pmatrix}
\bar{A}&\bar{B}&{C}&\bar{D}\\
\bar{B}&C&\bar{D}&\bar{A}
\end{pmatrix}, \; \; \;  \; \; \; \sigma''=\mathcal{R}_{C>D} \sigma=\begin{pmatrix}
\bar{A}&\bar{B}&D&\bar{C}\\
\bar{B}&{D}&\bar{A}&\bar{C}
\end{pmatrix}.
\end{equation*}
 Only on one of these two permutations, $\sigma'$, the Rauzy induction can be continued. The class $[\sigma'']$ is hence not included into the graph on the picture, as in A5. The presented strongly connected component has two dotted arrows. For example,
for the permutation 
\begin{equation*}
\delta=\begin{pmatrix}
D&\bar{A}&\bar{B}&\bar{C}\\
\bar{A}&\bar{C}&D&\bar{B}
\end{pmatrix},
\end{equation*}
the combinatorial Rauzy induction leads to
\begin{equation*}
\delta'=\mathcal{R}_{B<C} \delta=\begin{pmatrix}
D&\bar{A}&{B}&\bar{C}\\
\bar{A}&B&\bar{C}&D
\end{pmatrix}.
\end{equation*}
One can see that $\delta'\sim_h \tilde{\delta'}$, where $\tilde{\delta'}=
\begin{pmatrix}
D&\bar{A}&{B}&\bar{C}\\
\bar{A}&B&\bar{C}&D
\end{pmatrix}
$ via the map $h: \A \rightarrow \A$ such that $h(A)=B, h(B)=C, h(C)=A, h(D)=D$. Hence the edge $e=(\delta,\tilde{\delta'})$ is dotted, see A7.

%For the information on the full connected component of $\sigma,$ see the Appendix.
\begin{figure}
\includegraphics[scale=0.7]{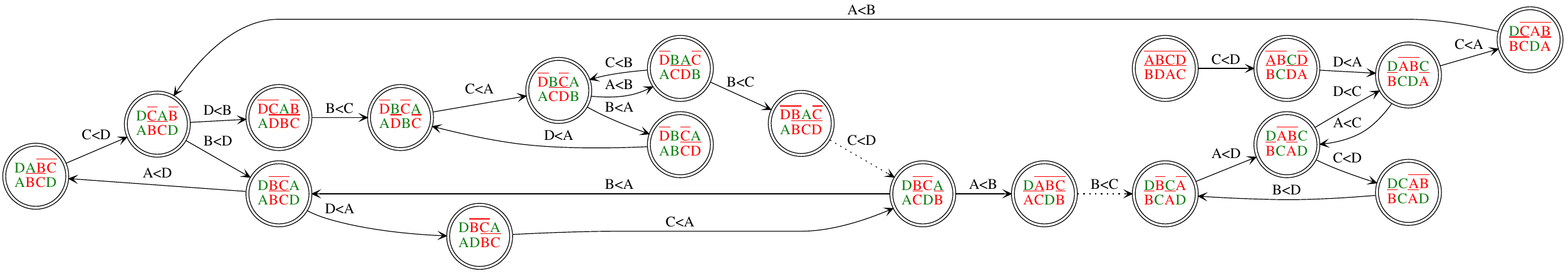}
\caption{A graphic representation of a connected component of the permutation $\sigma$ defined by \eqref{eq:basic_combinatorial_type_CET_3} in the quotient Rauzy graph.}
\label{pic:graph_for_four}
\end{figure}
\end{example}

\section{Necessary condition for minimality in $\CETthree$.}\label{sec:one-half}
In this Section we will prove the following

\begin{theorem}\label{thm:first_ingredient}
If $F \in \CETthree$ is minimal then $\tau=\frac{1}{2}$.
\end{theorem}

We will see in the following that this theorem is an important ingredient in order to qualify the non-linearly escaping behavior in triangle tiling billiards (Theorem \ref{thm:when_escape}).

\subsection{The existence of the invariant hyperplane $\tau=\frac{1}{2}$.}

In this paragraph we explain why the hyperspace $\tau=\frac{1}{2}$ is invariant under Rauzy induction in the families of $\CETn$.

Let $F \in \CETn$ acting on $\mathbb{S}^1$ and $X_F$ be the cylinder $\mathbb{S}^1 \times (-1,1)$ with the following identifications on the horizontal boundaries:
$(x,1)$ is identified with $(F(x),-1)$. In plain terms, on each interval of continuity of $F$, an interval and its image are identified by flip.
$X_F$ is a non orientable compact surface.  The vertical flow is well-defined on $X_F$. It means that we consider an orientable foliation on a non orientable surface. By definition, the first return time on  $\mathbb{S}^1 \times \{ 1\}$ of this flow is the map $F$.

When $\tau = 1/2$, the surface $X_F$ possesses an additional symmetry. It is invariant by the anti-holomorphic map 
$\phi: (x,y) \to (x+1/2[1], -y)$. An elementary calculation shows that the quotient is a projective plane (see Figure \ref{surfaces}).

\begin{figure}[h]
\centering
\begin{tikzpicture}
\draw[magenta, very thick, ->](0,0)--(2,0);
\draw[yellow, very thick, ->](2,0)--(6,0);
\draw[cyan, very thick, ->](6,0)--(7,0);

\draw[magenta, very thick, <-](3.5,-2)--(5.5,-2);
\draw[yellow, very thick, <-](5.5,-2)--(7,-2);
\draw[yellow, very thick](0,-2)--(2.5,-2);
\draw[cyan, very thick, <-](2.5,-2)--(3.5,-2);

\draw[ very thick](0,0)--(0,-2);
\draw[ very thick](7,0)--(7,-2);

\draw(1,0.5) node{A};
\draw(4,0.5) node{B};
\draw(6.5,0.5) node{C};

\draw(4.5,-2.5) node{A};
\draw(0.5,-2.5) node{B};
\draw(3,-2.5) node{C};

\draw[magenta, very thick, ->](0,-4)--(1,-4);
\draw[magenta, very thick, dashed, <-](1,-4)--(2,-4);
\draw[yellow, very thick, ->](2,-4)--(4,-4);
\draw[yellow, very thick, dashed, <-](4,-4)--(6,-4);
\draw[cyan, very thick, ->](6,-4)--(6.5,-4);
\draw[cyan, very thick, dashed, <-](6.5,-4)--(7,-4);

\draw[ very thick](0,-4)--(0,-5);
\draw[ very thick](7,-4)--(7,-5);

\draw[green, very thick, ->](0,-5)--(3.5,-5);
\draw[green, very thick, dashed,->](3.5,-5)--(7,-5);
\end{tikzpicture}
\caption{The surface $X_F$ and its quotient by the involution $\phi$.}
\label{surfaces}
\end{figure}

The vertical flow descends to the quotient as a foliation. Since the action of Rauzy induction on $X_F$ is realized as a cut and paste of rectangles with horizontal and vertical sides, it preserves the symmetry $\phi$.  Therefore, Rauzy induction preserves the hypersurface $\tau = 1/2$.

Unfortunately, this geometric remark doesn't suffice to prove Theorem 
\ref{thm:first_ingredient}. We need one extra combinatorial remark that is based on the study of the invariants in the modified Rauzy graphs of IETFs. We hope that in the future we will be able to understand the geometric meaning of the invariant that we describe in the next paragraph.

\subsection{The combinatorial proof.}\label{subs:combinatorial_proof}

\begin{lemma}\label{lemma:simple_combinatorics_lemma}
If $F \in \CETthree$ is minimal hence its combinatorics as a map in $\mathrm{FET}^{4}[0,1)$ where $0$ corresponds to the end of one of the intervals of continuity, is encoded with a generalized permutation $\sigma$, defined by \eqref{eq:basic_combinatorial_type_CET_3}.
\end{lemma}

\begin{proof}
If $I_j'=\mathring{F(I_j)} \cap \mathring{I_j} \neq \emptyset$ for some $j \in \left\{a,b,c\right\}$, $F$ is obviously not minimal since it has an open interval of $2$-periodic points. Hence, if $0$ is a left singularity of the interval $I_a$ (as in Definition \ref{def:fully_flpped_3_interval_exchange_transformations_on_the_circle:4}) and $F$ is minimal then $0 \in F(I_b)$, see Figure \ref{fig:combinatoric_type}.
\end{proof}

\begin{remark}
One can explicitely write out all the possible combinatorial possibilities for the generalized permutations defining the maps in $\CETthree$, only the first one of them being irreducible:
\begin{equation}\label{eq:choice}
\begin{pmatrix}
\bar{A}&\bar{B}&\bar{C}&\bar{D}\\
\bar{B}&\bar{D}&\bar{A}&\bar{C}
\end{pmatrix};
\begin{pmatrix}
\bar{A}&\bar{B}&\bar{C}&\bar{D}\\
\bar{A}&\bar{C}&\bar{D}&\bar{B}
\end{pmatrix};
\begin{pmatrix}
\bar{A}&\bar{B}&\bar{C}&\bar{D}\\
\bar{C}&\bar{A}&\bar{B}&\bar{D}
\end{pmatrix}.
\end{equation} 
\end{remark}

From Lemma \ref{lemma:minimality} and Lemma \ref{lemma:simple_combinatorics_lemma} we see that the understanding of the minimal maps in $\CETthree$ boils down to the understanding of the cycles of Rauzy induction in the connected component of the Rauzy graph of $\sigma$. This connected component is the main hero of this Section, its full representation is very big although can be drawn explicitely which was done by Paul Mercat.\footnote{Here is the address : \url{https://drive.google.com/file/d/1JlUxXyWcO0izTZtu9CiJl13eannz_W5I/view?usp=sharing}} Although, a representation of this component in the quotient graph has only $19$ (irreducible) vertices and is sufficient for our needs. It is given on the Figure \ref{pic:graph_for_four}.
%\footnote{
%\href{https://drive.google.com/file/d/1JlUxXyWcO0izTZtu9CiJl13eannz_W5I/view?usp=sharing}{\underline{\textit{https://drive.google.com/file/d/1JlUxXyWcO0izTZtu9CiJl13eannz_W5I/view?usp=sharing}}} }

%We will prove that the values of parameters $(l_j, \tau) \in C \times \mathbb{S}^1$ corresponding to infinite paths in this graph all belong to the hyperplane $\{\tau=\frac{1}{2}\}$. 

\bigskip

Note that the Rauzy induction can be defined as a map on $\IETFn[0,1)$ by renormalizing the interval of the definition of $\mathcal{R}F$ in order for it to have length $1$. Then, the lengths of the intervals of continuity can be parametrized by homogeneous coordinates
$[\lambda_A:\lambda_B:\lambda_C:\lambda_D] \in \PP^3$. In this case, for $F_{\tau, l_1, l_2, l_3} \in \CETthree$ with the combinatorics $\sigma$ defined by \ref{eq:basic_combinatorial_type_CET_3}, the lengths of the intervals $\{\lambda_i\}_{i \in \A}$ can be described in terms of the parameters $l_1, l_2, l_3, \tau$:
\begin{equation*}
\begin{pmatrix} 
\lambda_A\\
\lambda_B\\
\lambda_C\\
\lambda_D
\end{pmatrix}=
\begin{pmatrix}
1&0&0&0\\
0&0&-1&1\\
0&1&1&-1\\
0&0&1&0
\end{pmatrix}
\begin{pmatrix}
l_1\\
l_2\\
l_3\\
\tau
\end{pmatrix}.
\end{equation*}

The equation $\tau=\frac{1}{2}$ (or, $2\tau=l_1+l_2+l_3$), in the homogeneous coordinates $\{\lambda_i\}_{i \in \A}$ can be written as 
\begin{equation*}
\lambda_A+\lambda_C=\lambda_B+\lambda_D.
\end{equation*}

This equation defines a hyperplane $H \subset \PP^3$ with an orthogonal vector $v^{\perp}:=(1,-1,1,-1)^T$. It happens that the vector $v^{\perp}$ is invariant under the Rauzy transformations.

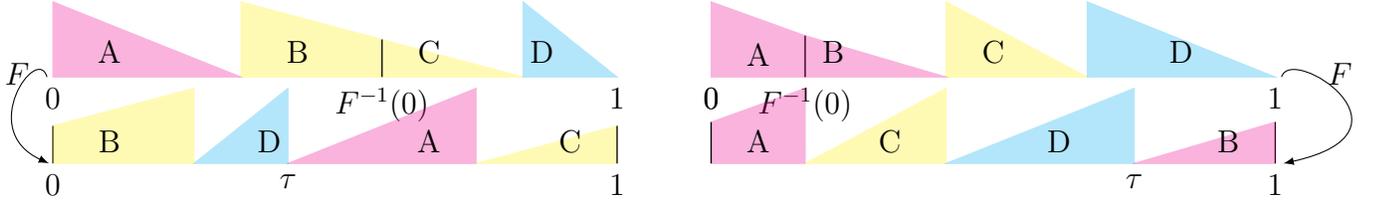
\begin{figure}
\centering
\begin{tikzpicture}[xscale=12.5]
\path[draw] (0,0) node[below] (1) {$0$} -- cycle;

\path[draw] (0.6,0) node[below]{$1$} -- cycle;
\path[draw, fill=magenta,magenta, opacity=0.3] (0,0)--(0,1)--(0.2,0)--cycle;
\path[draw] (0.06,0.05) node[above]{A} -- cycle;

\path[draw, fill=yellow,yellow, opacity=0.3] (0.2,0)--(.2,1)--(0.5,0)--cycle;
\path[draw] (0.26,0.05) node[above]{B} -- cycle;
\path[draw] (0.4,0.05) node[above]{C} -- cycle;

\path[draw=black, line width=0.5] (0.35,0) node[below]{$F^{-1}(0)$}--(0.35,.5) -- cycle;

\path[draw, fill=cyan,cyan, opacity=0.3] (0.5,0)--(.5,1)--(0.6, 0)--cycle;
\path[draw] (0.52,0.05) node[above]{D} -- cycle;
%
%\path[draw, fill=green,green, opacity=0.3] (0.6,0)--(.6,1)--(1, 0) --cycle;
%\path[draw] (0.7,0.05) node[above]{E} -- cycle;

%SECOND DESSIN
\path[draw] (0.7,0) node[below]{$0$} -- cycle;

\path[draw, fill=magenta,magenta, opacity=0.3] (0.7,0)--(0.8,0)--(0.8,0.55)--(0.7,1)--cycle;
\path[draw] (0.75,0) node[above]{A} -- cycle;

\path[draw=black, line width=0.5] (0.8,0) node[below]{$F^{-1}(0)$}--(0.8,.55) -- cycle;

\path[draw, fill=magenta,magenta, opacity=0.3] (0.8,0)--(.8,.55)--(0.95,0)--cycle;
\path[draw] (0.83,0.05) node[above]{B} -- cycle;

\path[draw, fill=yellow,yellow, opacity=0.3] (0.95,0)--(.95,1)--(1.1,0)--cycle;
\path[draw] (1,0.05) node[above]{C} -- cycle;

\path[draw, fill=cyan,cyan, opacity=0.3] (1.1,0)--(1.1,1)--(1.3, 0)--cycle;
\path[draw] (1.2,0.05) node[above]{D} -- cycle;

\path[draw] (1.3,0) node[below] (2) {$1$} -- cycle;

%DOWN PART

\path[draw=black, line width=0.5] (0,-1.15) node[below] (3){$0$}--(0,-0.65) -- cycle;

\path[draw, fill=yellow,yellow, opacity=0.3] (0,-1.15)--(0,-0.65)--(.15,-0.15)--(0.15,-1.15)--cycle;
\path[draw] (0.06,-1.15) node[above]{B} -- cycle;

\path[draw, fill=cyan,cyan, opacity=0.3] (0.15,-1.15)--(.25,-1.15)--(0.25, -0.15)--cycle;
\path[draw] (0.23,-1.15) node[above]{D} -- cycle;
%
%\path[draw, fill=green,green, opacity=0.3] (0.25,0)--(.65,0)--(0.65, 1)--cycle;
%\path[draw] (0.5,0.05) node[above]{E} -- cycle;
\path[draw] (0.25,-1.15) node[below]{$\tau$} -- cycle;

\path[draw, fill=magenta,magenta, opacity=0.3] (0.25,-1.15)--(0.45,-1.15)--(0.45,-0.15)--cycle;
\path[draw] (0.4,-1.15) node[above]{A} -- cycle;

\path[draw, fill=yellow,yellow, opacity=0.3] (0.45,-1.15)--(0.6,-1.15)--(0.6,-0.65)--cycle;
\path[draw] (0.55,-1.15) node[above]{C} -- cycle;
\path[draw=black, line width=0.5] (0.6,-1.15)--(0.6,-0.65) -- cycle;

\path[draw=black, line width=0.5] (0.6,-1.15) node[below]{$1$}  -- cycle;

%SECOND DESSIN
\path[draw] (0.7,0) node[below]{$0$} -- cycle;

\path[draw, fill=magenta,magenta, opacity=0.3] (0.7,-1.15)--(0.7,-0.6)--(0.8,-0.15)--(0.8,-1.15)--cycle;
\path[draw] (0.75,-1.15) node[above]{A} -- cycle;

\path[draw=black, line width=0.5] (0.7,-1.15)--(0.7,-0.6) -- cycle;
\path[draw=black, line width=0.5] (1.3,-1.15)--(1.3,-.6) -- cycle;

\path[draw, fill=yellow,yellow, opacity=0.3] (0.8,-1.15)--(.95,-1.15)--(0.95,-0.15)--cycle;
\path[draw] (0.89,-1.15) node[above]{C} -- cycle;

\path[draw, fill=cyan,cyan, opacity=0.3] (.95,-1.15)--(1.15,-.15)--(1.15, -1.15)--cycle;
\path[draw] (1.07,-1.15) node[above]{D} -- cycle;

\path[draw, fill=magenta,magenta, opacity=0.3] (1.15,-1.15) --(1.3,-1.15)--(1.3,-.6)--cycle;
\path[draw] (1.15,-1.15) node[below]{$\tau$} -- cycle;
\path[draw] (1.25,-1.15) node[above]{B} -- cycle;

\path[draw] (1.3,-1.15) node[below] (4) {$1$} -- cycle;

 \tikzset{
        arrow/.style={
            color=black,
            draw=black,
            -latex,
                font=\fontsize{13}{13}\selectfont},
        }

\draw[arrow](1) to [out=92,in=100] node[above]{$F$} (3);
\draw[arrow](2) to [out=88,in=70] node[above]{$F$} (4);
\end{tikzpicture}

\caption[]{The representation of the map $F \in \CETthree$ as a map of the interval. On the \emph{left} : the case when $0 \in F(I_b)$, and with the combinatorics   \eqref{eq:basic_combinatorial_type_CET_3}. Here $I_b=I_B^{\ttop} \cup I_C^{\ttop}$. On the \emph{right} : the case when $0 \in F(I_a)$. Here $I_a=I_A^{\ttop} \cup I_B^{\ttop}$, and $F$ is not minimal - $I_A^{\ttop}$ is an interval of $2$-periodic points.}\label{fig:combinatoric_type}
\end{figure}

\begin{lemma}\label{lemma:main_vector_preserved_lemma}
Fix any vertex $\sigma'$ in the connected component of the Rauzy graph of the permutation $\sigma$ defined by \eqref{eq:basic_combinatorial_type_CET_3}. Take a finite path $\gamma=e_1 \ldots e_m$ in the graph connecting $\sigma$ to $\sigma'$. Then the vector $A_{\gamma}^T v^{\perp}:=A_{(m)}^{T} v^{\perp}$ doesn't depend on the path taken. 
\end{lemma}
\begin{proof}
The proof is an explicit verification by computer, see Figure \ref{fig:part_graph_4_connected_component_equivalence} of the Appendix for illustration.
 \end{proof}

The vector $A_{\gamma}^T v^{\perp}:=A_{\sigma'} v^{\perp}$ is then the invariant of this vertex.

%\begin{figure}
%\centering
%%\includegraphics*[scale=0.3]{part_graph_3_small.eps}
%\includegraphics*[scale=0.3]{part_graph_3_small.png}
%\caption[]{
%%\textbf{OLGA ! CHANGE THIS FIGURE!! and the text below!!! with new notations}
%A part of the connected component of $\sigma$ in the Rauzy graph. To each vertex $(\sigma, \boldsymbol{k})$ corresponds an invariant which is an image of vector $v^{\perp}$ under the compositions of matrices $A_e^T$ leading to this vertex. The vector $\boldsymbol{k}$ is represented by the colors : red colored letters correspond to the components equal to $-1$ and green colored letters to the components equal to $1$.}\label{fig:part_graph_3_small}
%\end{figure}

We think that the invariance of the vector $v^{\perp}$ in the sense of Lemma \ref{lemma:main_vector_preserved_lemma} has a geometric interpretation that we didn't manage to find. The following is based on experimental data: for the vertices in the connected component of $\sigma$ that belong to the cycles in the Rauzy graph and their assigned vectors $A^T v^{\perp}$, the non-flipped letters correspond to the coordinates equal to $0$, and the flipped letters to coordinates equal to $\pm 1$. We hope to explore this in the future work. 

On the Figure \ref{fig:part_graph_4_connected_component_equivalence} of the Appendix we represent 
the graph from Figure \ref{pic:graph_for_four} with the additional information on the values of vectors $A_{[\sigma]} v^{\perp}$ for each of its vertices $[\sigma]$. Now we are ready to prove Theorem \ref{thm:first_ingredient}.

\smallskip

\begin{proof}
The vector of lengths $\lambda=(\lambda_i)_{i \in \mathcal{A}}$ with $\mathcal{A}=\{A,B,C,D\}$  and the permutation $\sigma$ as in \ref{eq:basic_combinatorial_type_CET_3} define a map $F \in \CETthree$. The goal is to prove that if the set $\{\mathcal{R}^n F\}$ is infinite (i.e., the Rauzy induction doesn't stop), then $\lambda \in H$.

Consider the infinite path $\gamma$ in the Rauzy graph $G$ defined by $\{\mathcal{R}^n F\}_{n=1}^{\infty}$: $\gamma=e_1 \ldots e_m \ldots $. This path is obviously contained in the connected component of $\sigma$ in $G$. 

Since $G$ is a finite graph, there exists a vertex $\tilde{\sigma} \in G$ such that the path $\gamma$ passes by $\tilde{\sigma}$ an infinite number of times: there exists a sequence of oriented edges $e_{k_i} \in G,  (k_i)_{i=1}^{\infty}$ such that each of them ends up at $\tilde{\sigma}$. 

Then $\lambda^{(k_m)}=A_{(k_m)}^{-1} \lambda$ for $A_{k_m}=A_{e_1} \ldots A_{e_{k_m}}$. By Lemma \ref{lemma:minimality}, $||\lambda^{(k_m)}||_{\infty} \rightarrow 0$ when $k_m \rightarrow \infty$. Let 
\begin{equation*}
\tilde{v}^{\perp}_j:=A_{(k_j)}^T v^{\perp}=A_{e_{k_1}}^T \ldots A_{e_{k_j}}^Tv^{\perp}
\end{equation*}
be the image of $v^{\perp}$ corresponding to the pre-cycle $e_{k_1} \ldots e_{k_j}$. By Lemma \ref{lemma:main_vector_preserved_lemma}, the vector $\tilde{v}^{\perp}_j$ doesn't depend on $j \in \mathbb{N}^*$ and is an invariant of the vertex $\tilde{\sigma}$ itself: $\tilde{v}^{\perp}_j=A_{\tilde{\sigma}} {v}^{\perp}$.

Denote $\tilde{H}:=A^{-1}_{(k_1)} H$. The orthogonality of $v^{\perp}$ and $H$ is equivalent to the orthogonality of $A_{\tilde{\sigma}} {v}^{\perp}$ and $\tilde{H}$. Now define the products of matrices that correspond to the loops in $G$ created by the path $\gamma$ based at $\tilde{\sigma}$. Let for any $m \in \mathbb{N}^*, m \geq 2$

\begin{equation}\label{eq:cycle_matrix_A}
\tilde{A}_{(m)}:=A_{e_{k_{m-1}+1}} \ldots A_{e_{k_m}}.
\end{equation}

Then, the calculation of the scalar products gives the result:
\begin{align*}
<A_{\tilde{\sigma}} {v}^{\perp}, \lambda^{(k_1)}> = <(\tilde{A}_{(m)}^{-1})^TA_{\tilde{\sigma}} {v}^{\perp}, \lambda^{(k_1)}> = <A_{\tilde{\sigma}} {v}^{\perp}, \tilde{A}_{(m)}^{-1} \lambda^{(k_1)}>= <A_{\tilde{\sigma}} {v}^{\perp},\lambda^{(k_m)} > \rightarrow 0, m \rightarrow \infty.
\end{align*}

Indeed, we see that $\lambda^{(k_1)} \perp A_{\tilde{\sigma}} {v}^{\perp}$ which is equivalent to $\lambda \in H$.

\end{proof}

\section{Structure of non-minimal maps in  $\CETn$ and integrability.}\label{sec:integrability_section}

In this Section we study a class of\emph{ integrable} interval exchange transformations with flips. For any map in this class, its suspension on a non-orientable surface has invariant cylinders and tori, foliated by linear foliations. This means that the dynamics of such IETs is very simple. We prove that a map in $\IETFn$ is integrable: always for $n=3$ and almost always for $n=4$. For $n=5$, we find an open set of non-integrable dynamics.

\subsection{Simple and integrable interval exchange transformations with flips: definitions.}
%: yes for $n=3,4$ and no for $n \geq 5$

Any interval exchange transformation with flips $F$ can be defined by a triple $(\sigma, \boldsymbol{k}, \boldsymbol{ \lambda})$  that contains combinatorial data (generalized permutation and the vector of flips) and the vector of lengths of the intervals of continuity, see paragraph \ref{subs:MRI}. The generalized permutation $\sigma$ encodes in itself two orders on alphabet $\mathcal{A}$ that correspond to two words $\omega^{\ttop}, \omega^{\bbot} \in \mathcal{A}^n$ which correspond to the first and the second row of the matrix representation \eqref{eq:matrix_representation_gen_permutation} of $\sigma$.

%Indeed, $\omega^{\ttop}$ is an $n$-letter word $\omega^{\ttop}=$with a letter =\left(\sigma^{\ttop}\right)^{-1}(1)\ldots $ 

\begin{definition}\label{def:simple}
The interval exchange transformation with (or without) flips $F \in \IETFn \cup \mathrm{IET}^n$ defined by the data $(\sigma, \boldsymbol{k}, \boldsymbol{ \lambda})$  is called \emph{simple} if $\exists p \in \{1, \ldots, n \}$ such that $\omega^{\ttop}=\omega^{\ttop}_1 \ldots \omega^{\ttop}_p$ and $\omega^{\bbot}=\omega^{\bbot}_1 \ldots \omega^{\bbot}_p$ with $\omega^{\bbot}_j, \omega^{\ttop}_j$ being nonempty words in the alphabet $\mathcal{A}$. And moreover, for any $j \in \{1, \ldots, p \}$ exactly one of these possibilities holds:

\begin{itemize}
\item[1.] (\emph{periodic cylinders}) $\omega^{\ttop}_j=\omega^{\bbot}_j$,
\item[2.] (\emph{cylinders of rotation}) there exist two words $x \neq y$ such that $\omega^{\ttop}_j=xy$ and $\omega^{\bbot}_j=yx$, and the coordinates of $\boldsymbol{k}$ corresponding to all the letters of $\A$ in the word $\omega^{\ttop}_j$ are equal to $1$,
\item[3.] (\emph{cylinders of rotation with a marked singularity})  there exist three different words $x, y$ and $z$ such that $\omega^{\ttop}_j=xyz, \omega^{\bbot}_j=zyx$; $\omega^{\ttop}_j=xyz, \omega^{\bbot}_j=zxy$ or $\omega^{\ttop}_j=xyz, \omega^{\bbot}_j=yzx$, and the corresponding to $\omega^{\ttop}$ coordinates of $\boldsymbol{k}$ are equal to $1$.
\end{itemize}

Obviously, in any of three cases, the lengths of $\omega^{\bbot}_j$ and $\omega^{\ttop}_j$ coincide and these words consist of the same sets of letters. The simplicity of $F$ doesn't depend on $\boldsymbol{\lambda}$ but only on its combinatorial data $(\sigma, \boldsymbol{k})$. 

\end{definition}

\begin{remark}
Simplicity can be defined in purely geometric terms. For example, for a map $F \in \FETn$, its square $T=F^2$ is an IET without flips. If $F$ is simple then, a translation surface corresponding to $T$ can be cut along the lines of the flow into  the union of the invariant cylinders and tori. The flow preserves a linear foliation on the invariant tori and a trivial foliation by periodic leaves on the cylinders.
\end{remark}

\begin{example}
The following combinatorial data gives three examples of simple maps.
\begin{itemize}
%A $4$-interval exchange transformation with flips with combinatorial data 
\item[1.]$\begin{pmatrix}
A&B&\bar{C}&\bar{D}\\
B&A&\bar{C}&\bar{D}
\end{pmatrix}
: \; \; \; \; \omega_1^{\ttop}=AB, \omega_1^{\bbot}=BA, \omega_2^{\ttop}=\omega_2^{\bbot}=CD$, $\boldsymbol{k}=(1,1,-1,-1)$;

\item[2.]$\begin{pmatrix}
B&A&C&{D}&\bar{E}\\
C&B&A&{D}&\bar{E}
\end{pmatrix}: \; \; \; \omega_1^{\ttop}=BAC, \omega_1^{\bbot}=CBA,  \omega_2^{\ttop}=\omega_2^{\bbot}=DE$, $\boldsymbol{k}=(1,1,1,1,-1)$;

\item[3.]$\begin{pmatrix}
A&D&E&B\\
B&A&D&E
\end{pmatrix}: \;\;\;\;\; \omega_1^{\ttop}=xy,\omega_1^{\bbot}=yx, x=ADE, y=B$, $\boldsymbol{k}=(1,1,1,1)$.
\end{itemize}
\end{example}

\begin{definition}
A map $F \in \IETFn \cup \mathrm{IET}^n$ is \emph{integrable} if there exists a Poincaré section for $F$ such that the dynamics of the first return map on this section is a simple IET (with or without) flips. 
A map $F \in \mathrm{CET}^n \cup \mathrm{CETF}^n$ is \emph{integrable} if a corresponding map on the interval is integrable.
\end{definition}

One easily checks that this definition for circle maps doesn't depend on the marked point. 
%For all the maps in the article, the Poincaré section can be represented as a union of segments.

\begin{remark}
Integrability of an interval exchange transformation with flips has a simple geometric interpretation. Indeed, the integrability condition gives a strong topological restriction on a corresponding non-orientable flat surface obtained as a suspension. This surface can be cut along the lines of the suspension flow into a union of cylinders (on which a first-return map is a flip) and tori (on which a first-return map is a rotation, possibly identical). Of course, integrability implies the absence of minimality.
\end{remark}

\subsection{Integrability of maps in $\IETFthree$.}

In this paragraph we show that all maps from $\IETFthree$ are integrable. Let us first remark that for $n=2$ it follows obviously from the following result, proven by Keane.

\begin{proposition}[\cite{K75}]\label{prop:keane}
All $\IETFtwo$ are integrable, and even more, completely periodic. 
\end{proposition}

\begin{proof} 
The study of the Rauzy graph gives a proof of this Theorem, see Figure \ref{fig:IET2}. A Rauzy graph for $\IETFtwo$ doesn't permit infinite loops for Rauzy induction and moreover, all the stop points correspond to simple maps.
%If all the intervals are flipped, $F$ is $2$-periodic and the Rauzy induction stops after one iteration. If only one of the intervals is flipped then the number of steps for the induction depends on the ratio of lengths $\frac{\lambda_A}{\lambda_B}$.
%\textbf{write down possible periods !}
\end{proof}

The proof we give here is modern. Keane's proof was done in $1975$ by other methods, four years before the invention of the standard Rauzy induction. We prove now an analogous statement for the family $\IETFthree$. 

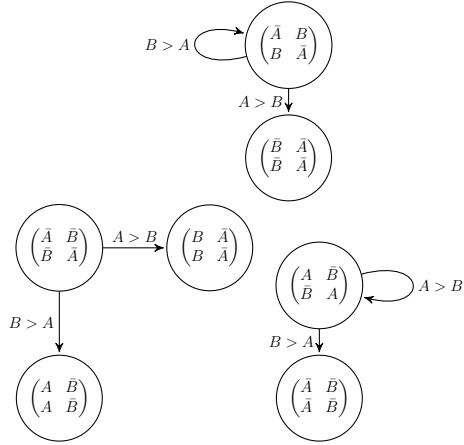
\begin{figure}
\centering

\begin{tikzpicture}[>=stealth',shorten >=1pt,auto,node distance=3 cm, scale = 0.5, transform shape]

\node[state] (A)                                   
 {$\begin{pmatrix}
     \bar{A}&B\\
     B&\bar{A}
\end{pmatrix} $};
\node[state]         (B) [below of=A]              
         {$\begin{pmatrix}
    \bar{B}&\bar{A}\\
     \bar{B}&\bar{A}
\end{pmatrix}$};

\path[->] (A) edge [left]       node [align=center]  {$ A>B $} (B)
        (A) edge [loop left] node [align=center]  {$ B >A $} (A);

\end{tikzpicture}

\begin{tikzpicture}[>=stealth',shorten >=1pt,auto,node distance=4 cm, scale = 0.5, transform shape]
\node[state] (A)                                   
 {$\begin{pmatrix}
     \bar{A}&\bar{B}\\
     \bar{B}&\bar{A} 
     \end{pmatrix}$};
\node[state]   (B) [below of=A]                      
 {$\begin{pmatrix}
    A&\bar{B}\\
     A&\bar{B}
\end{pmatrix}$};
\node[state]         (C) [right of=A]                       
{$\begin{pmatrix}
    B&\bar{A}\\
     B&\bar{A}
\end{pmatrix}$};

\path[->] (A) edge [left]       node [align=center]  {$ B>A $} (B)
        (A) edge [above] node [align=center]  {$ A >B $} (C);

\end{tikzpicture}
\begin{tikzpicture}[>=stealth',shorten >=1pt,auto,node distance=3 cm, scale = 0.5, transform shape]

\node[state] (A)                                   
 {$\begin{pmatrix}
     A&\bar{B}\\
     \bar{B}&A
\end{pmatrix} $};
\node[state]         (B) [below of=A]              
         {$\begin{pmatrix}
    \bar{A}&\bar{B}\\
     \bar{A}&\bar{B}
\end{pmatrix}$};

\path[->] (A) edge [left]       node [align=center]  {$ B>A $} (B)
        (A) edge [loop right] node [align=center]  {$ A >B $} (A);

\end{tikzpicture}

\caption[]{A union of components of the Rauzy graph for the class of maps from $\mathrm{IETF}^2$ for all generalized permutations $\sigma$ such that ${\omega}^{\ttop}=AB$.}\label{fig:IET2}
\end{figure}

\begin{proposition}\label{prop:ietfthree_integrability}
Any $F \in \IETFthree[0,1)$ is integrable, the corresponding Poincaré section can be chosen as a segment with one of its ends equal to $0$, and the set $\{{\mathcal{R}}^n F\}$ is finite (the modified Rauzy induction eventually stops).
\end{proposition}
\begin{proof}
The proof follows from the explicit study of the Rauzy graph for $\IETFthree$. One can easily see that for any cycle $\gamma$ in this graph there exists a letter from the alphabet $\A=\{A,B,C\}$  such that the corresponding interval never wins along $\gamma$. This means that the Rauzy induction stops for any $F \in \IETFthree$. See Figure \ref{fig:little_component} for one of such cycles in the Rauzy graph. 

In general, the stop of the Rauzy induction reduces the study of the integrability of maps in $\IETFthree$  to the case of the integrability of IETs on a smaller number of intervals. Indeed, finding a Poincaré section with an integrable map gives a (possibly, finer) Poincaré section with a simple map. The proof of integrability finishes by recurrence: in the set $\IETFtwo \cup \mathrm{IET}^2$ all maps are integrable, from Proposition \ref{prop:keane} and the integrability of rotations. Here a more precise study of the Rauzy graph shows that the vertices on which the combinatorial Rauzy induction can't be continued, are simple.

%Proposition \ref{prop:keane}
%We prove that Rauzy induction stops after a final number of steps. This and the Proposition \ref{prop:keane} suffice to prove integrability. An important remark is that the only cycles  that can survive infinitely in the MRI are those in which each one of the labels wins at least one time. On the Rauzy diagram for $3$-IETs with flips these cycles simply do not exist. Indeed, the only possible cycles in this graph are one of two types (modulo the renaming of the letters):
%
%\begin{itemize}
%\item[1.] The cycle of edges $(C>A), (C>B)$ (up to relabelling)
%\item[2.] Two cycles of type 1. $(C>A), (C>B)$  and $(B>A), (B>C)$ that meet up in one vertex $\sigma$ : one can get in the vertex by following $B>A$ or $C>A$ and get out of the vertex by following $B>C$ or $C>B$. 
%\end{itemize}
%
%These two cycles can't be repeated infinitely, and there are no other cycles in this graph. Hence the MRI stops, and one easily verifies that it stops at a simple permutation.
\end{proof}

%For the information on the full Rauzy graph of $\IETFthree$, see the Appendix.

\begin{figure}
\centering
\includegraphics*[scale=0.5]{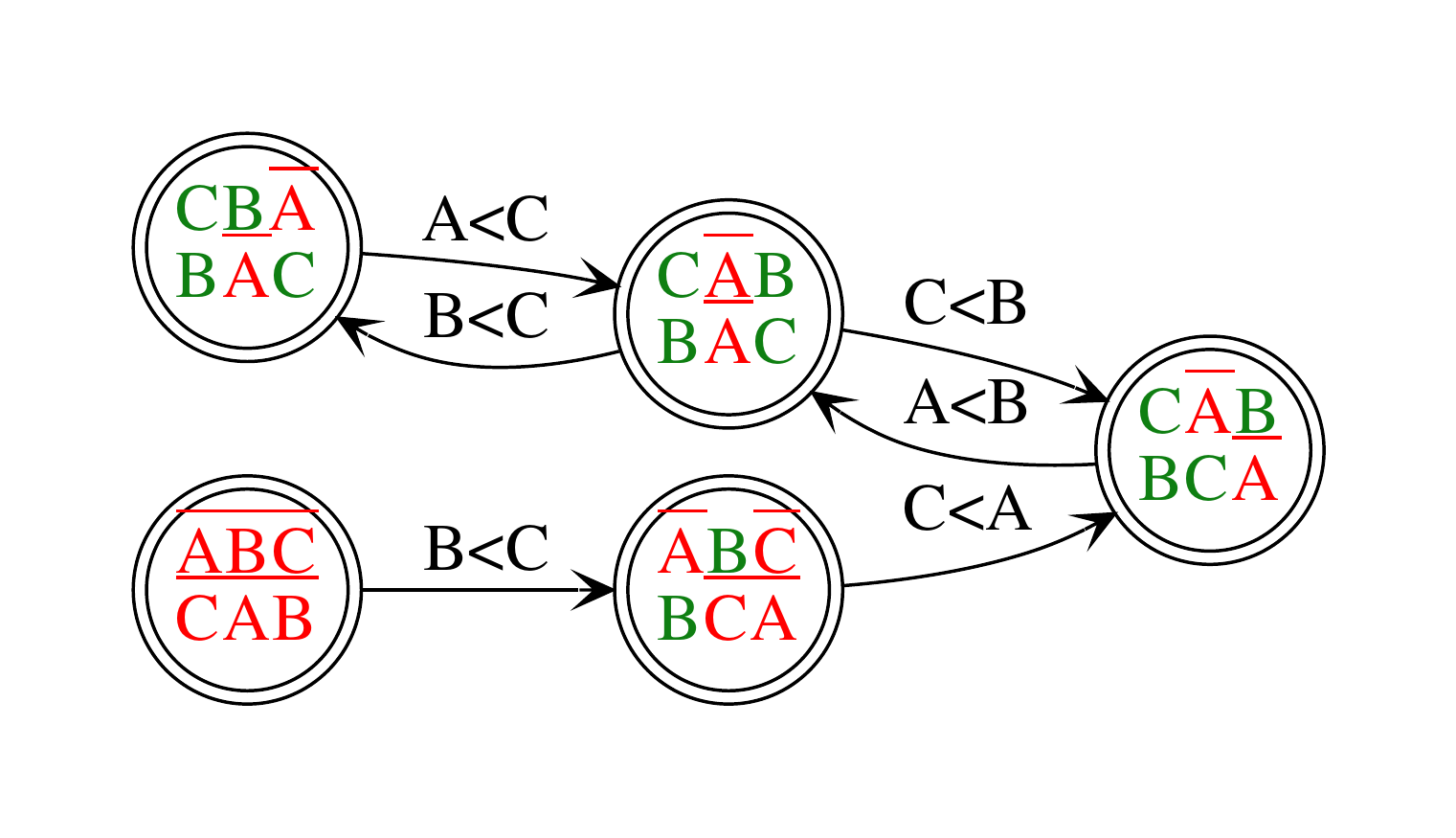}
\caption[]{
The connected component of the permutation $\begin{pmatrix}
\bar{A}&\bar{B}&\bar{C}\\
\bar{C}&\bar{A}&\bar{B}
\end{pmatrix}$ in the Rauzy graph of $\IETFthree$. One can see that in the cycles of this component, the letter $A$ never wins, hence the Rauzy induction stops on any of the maps $F\in \IETFthree$ with the combinatorics of this component.}
\label{fig:little_component}
\end{figure}

\subsection{Generic integrability of maps in $\IETFfour$.}

\begin{proposition}\label{prop:just_before}
Let $F \in \IETFfour[0,1)$ be such that the set $\{{\mathcal{R}}^n F\}$ is finite and the lengths $\lambda_i$ are rationnally independent. Then $F$ is integrable. 
%and the corresponding Poincaré section can be chosen as an interval with a left extremity equal to $0$.
\end{proposition} 
\begin{proof}

Suppose that the Rauzy induction stopped after $n$ steps for $F \in \IETFfour[0,1)$. This means that the combinatorial Rauzy induction has stopped as well. Then $\exists \alpha, \beta \in [0,1], \beta<\alpha$ and the segments $J=[\beta, \alpha), I=[0, \alpha)$ such that $I$ is a Poincaré section for $F$ and the restriction of the first-return map on $J$ is either an identity map or a flip. Denote by $G$ the restriction of the first-return map on $[0, \beta)$.

First, from the combinatorics of Rauzy induction follows that if $J$ is not flipped, then $F$ is reduced, $n=1$ and $\alpha=1$ (since a non-flipped winner can't come to the last place in a row). Moreover, $G \in \IETFthree [0,\beta)$. By Proposition \ref{prop:ietfthree_integrability}, $G$ is integrable with a corresponding Poincaré section $[0, \gamma)$ for some $\gamma\leq\beta$. Hence $F$ is integrable with a Poincaré section equal to the union $[0, \gamma) \cup [\beta, \alpha)$.
 
Second, in the case when $J$ is flipped, then either $ G \in \IETFthree[0,\beta)$ or $G \in \mathrm{IET}^3[0,\beta)$. The first case is treated as before. In the second case the proof is also finished since any map in $ \mathrm{IET}^3$ corresponds to a rotation with a marked singularity.
\end{proof} 

In the proof of this Proposition, we construct explicitely a Poincaré section that provides a simple first-return map on it, with the help of the Rauzy induction (on the right) for the maps in $\IETFthree$. Let us note that this construction can be generalized for the maps in $ \IETFn[0,1)$ for any $n$.

\begin{definition}
Take any $F \in \IETFn[0,1)$ such that the lengths $\lambda_i$ are rationally independent. If $\{ \mathcal{R}^n F\}$ is infinite, we define its standard Poincaré section to be $[0,1)$ and its standard Poincaré map to be itself. Suppose now $\{ \mathcal{R}^n F\}$ finite and $F$ not reduced. Then, analogically to the proof of Proposition \ref{prop:just_before}, one defines $\alpha_1, \beta_1 \in [0,1], J_1 \subset I_1=[0, \alpha_1)$ and the first return map $G_1$. If $G_1$ is simple, then we stop the procedure. If it is not simple, we reiterate the process with $F:=G_1$. Then, the union of corresponding segments gives a so-called \emph{Poincaré-Rauzy section}, and a \emph{Poincaré-Rauzy map} (as a first-return map on it). For the reduced $F$, we proceed with the same construction for each of the reduced components, and then unite the resulting Poincaré sections.
%
%
%
%Then, define $G_1:=G \in \IETFn [0,\beta)$ a map of the first return on the interval  $[0, \beta)$.
%(still flips at least one interval). If at some step, 
%
%
%followed as long as the map $G$ of the restriction of the first-return map defined by the sequence of Rauzy inductions on the interval $[0, \beta_j]$
% is still an interval exchange transformation with flips.
%
%in the proof of Proposition \ref{prop:just_before}, by following the Rauzy induction on the right. We will call such a Poincaré section a \textbf{standard} section, and denote a restriction of the first-return map on it by $G$. The standard Poincaré section is defined inductively as the union of the invariant interval coming from the stop of the Rauzy induction and the Poincaré section of a map on a smaller interval.
 \end{definition}

 As a direct corollary of Proposition \ref{prop:just_before} and Theorem \ref{thm:Nogueira}, we obtain
\begin{proposition}\label{thm:IETfour}
Almost any $F \in \IETFfour$ is integrable.
\end{proposition}

\subsection{Integrability of $\CETn$.}

Suppose that $F=F_{\tau, l_1, l_2, l_2}  \in \CETthree$ is such that the ratios $\frac{l_i}{l_j} \notin \Q$ for $i \neq j$ and the Rauzy induction stops for $F$. Then, by Proposition \ref{prop:just_before}, $F$ is integrable, with a Poincaré-Rauzy map being simple.

\begin{proposition}\label{prop:integrable_3}
Suppose that $F=F_{\tau, l_1, l_2, l_2}  \in \CETthree$ is such as above. Then $G$ is an IET with flips and for any interval $I_i$ of continuity of $G$ such that $G(I_i)=I_i$, $k_i=-1$ (this interval is flipped).
\end{proposition}

\begin{proof}
Let $\sigma$ and $\delta$ be generalized permutations corresponding to $F$ and $G$. In a Rauzy graph, there exists a path connecting $\sigma$ to $\delta$, by construction.
For some letter $X \in \mathcal{A}=\{A,B,C,D\}$ we have $\left({\delta}^{\ttop}\right)^{-1}(4)=\left({\delta}^{\bbot}\right)^{-1}(4)=X \in \mathcal{A}$. As in the arguments of Proposition \ref{prop:just_before}, $k_X \neq 1$.
% since a non-flipped winner can't come to t
%and hence $k_X=-1$. Indeed, on an arrow going into $\sigma$ the interval with label $X$ was the longest hence if it was not flipped, it couldn't be placed last after the cutting.

Suppose now that there exists an invariant interval $I_Y$ for $G$ different from $I_X$. This can happen only if  $\left(\delta^{\ttop}\right)^{-1}(j)=\left(\delta^{\bbot}\right)^{-1}(j)=Y \in \mathcal{A}$ for $j=1,3$ (since the lengths are rationally independent). We will now obtain a combinatorial contradiction: there is no inverse Rauzy path connecting $\delta$ to $\sigma$. 

Indeed, first, if $j=3$ then by following backward Rauzy induction, we obtain:
\begin{equation*}
\delta=\begin{pmatrix}
\ast&\ast&Y&\bar{X}\\
\ast&\ast&Y&\bar{X}
\end{pmatrix}\xleftarrow[]{X>Y}
\begin{pmatrix}
\ast&\ast&\bar{X}&\bar{Y}\\
\ast&\ast&\bar{Y}&\bar{X}
\end{pmatrix}=: \sigma_1 \xleftarrow[]{?} \emptyset.
\end{equation*} 

Inverse Rauzy induction can't be continued, hence $\sigma_1=\sigma$ but this is inconsistent with the possible combinatorics for the map $F$, see \eqref{eq:choice}. 

Second, if $j=1$, $\delta \neq \sigma$ (since $\sigma$ is defined by one of three permutation \eqref{eq:choice} and $\delta$ has an invariant non-flipped cylinder on the left of the interval). Then, since any letter moves to the left in a row only if it is a lose in a step of the induction, we have a following chain in the connected component of $\sigma$:

\begin{equation*}
\delta=\begin{pmatrix}
Y&\ast&\ast&\bar{X}\\
Y&\ast&\ast&\bar{X}
\end{pmatrix}\xleftarrow[] \; \; \; \ldots \; \; \; \xleftarrow[] \;\;
\begin{pmatrix}
Y&\ast&\ast&\bar{Z}\\
Y&\bar{Z}&\ast&\ast
\end{pmatrix}
\xleftarrow[]{Z>Y}
\begin{pmatrix}
\bar{Y}&\ast&\ast&\bar{Z}\\
\bar{Z}&\ast&\ast&\bar{Y}
\end{pmatrix}=: \sigma_2 \xleftarrow[]{?} \emptyset.
\end{equation*} 

After $\sigma_2$, the inverse path can't be continued and none of the permutations has one of the combinatorial types in \eqref{eq:choice}.
%
%The case $j=1$ is treated analogously. Indeed, any letter moves to the left in the line compared to its position only if it is a loser in a step of Rauzy induction.
%%Otherwise it moves right or stays in the same place. 
%Hence the situation with $j=1$ can occur only if the size of the permutation is $2$ or if the permutation has been already reduced with a non-flipped cylinder on the left to start with which is not the case:
%

%
\end{proof}

As a direct corollary of Theorem \ref{thm:first_ingredient} and Proposition \ref{prop:integrable_3}, we have
\begin{proposition}\label{thm:CETthree}
If $\tau \neq \frac{1}{2}$ then any $F_{\tau, l_1, l_2, l_3} \in \CETthree$ is integrable (and not minimal). Moreover, if the ratios $\frac{l_i}{l_j} \notin \Q$ for $i \neq j$ then for any periodic point of $F$ there exist an interval $I_i$ on the Poincaré-Rauzy section, containing it and flipped on itself by the Poincaré-Rauzy first-return map.
\end{proposition}

We have seen that the families $\IETFtwo, \IETFthree$ consist only of integrable maps, and the families $\IETFfour$ and $\CETthree$ have almost all of their maps integrable. Although, for a bigger number of intervals, in the families $\IETFn$ for $n \geq 5$, the stop of the Rauzy induction doesn't necessarily imply integrability.

\begin{proposition}\label{prop:non-integrability-for-big-n}
For any $n \in \N \geq 5$ there exists an open set of non-integrable maps in $\IETFn$.
\end{proposition}

\begin{proof}
For an alphabet $A=\{A,B,C,D,E\}$, take a generalized permutation
$$
\delta:=\begin{pmatrix} \bar{A} &\bar{B} & \bar{C} & \bar{D} & \bar{E} \\
\bar{B} &\bar{C} & \bar{D} & \bar{E} & \bar{A}
\end{pmatrix}
$$
and a subset $\mathcal{F}$ of maps in $\mathrm{FET}^5$ with such combinatorics $\delta$ and the following restrictions on the lengths : $\lambda_A \geq \max \{\lambda_B, \lambda_C, \lambda_D, \lambda_E\}$ and $\lambda_i$ being rationally independent, $i \in \A$. This is obviously an open set.

For any map $F \in \mathcal{F}$, the Rauzy induction will make the following four steps and then stop:

\begin{equation*}
\begin{pmatrix} \bar{A} &\bar{B} & \bar{C} & \bar{D} & \bar{E} \\
\bar{B} &\bar{C} & \bar{D} & \bar{E} & \bar{A}
\end{pmatrix}
\xrightarrow[]{A>E} 
\begin{pmatrix} E & \bar{A} &\bar{B} & \bar{C} & \bar{D}  \\
\bar{B} &\bar{C} & \bar{D} & {E} & \bar{A}
\end{pmatrix}
\xrightarrow[]{A>D}
\ldots 
\xrightarrow[]{A>B}
\begin{pmatrix} E & D &C & B & \bar{A}  \\
{B} &{C} & {D} & {E} & \bar{A}
\end{pmatrix}.
\end{equation*}
One can see that the map 
\begin{equation}\label{eq:not_simple_IET}
\begin{pmatrix} E & D &C & B  \\
{B} &{C} & {D} & {E} 
\end{pmatrix}.
\end{equation}
is not simple by looking at a corresponding foliation. This argument can be generalized for any value of $n \geq 5$. 
\end{proof}

We see that the class $\IETFfive$ has open sets of non-integrable maps in it. But by restricting to a smaller subset of specific combinatorics, one observes integrable behavior.

\begin{proposition}\label{thm:CETfour}
Consider the set $\mathcal{F}$ of fully flipped interval exchange transformations obtained as the image of $\CETfour$ under its natural inclusion in $\mathrm{FET}^5$. 
Take any $F \in \mathcal{F}$ such that ${\lambda}_j$ are independent over $\Q$ and the Rauzy induction stops. Then $F$ is integrable, and any periodic interval on the Poincaré-Rauzy section is flipped.
\end{proposition}

\begin{proof}
Here is the list of all  possible combinatorics of the maps in $\mathcal{F}$ depending on which of four intervals contains $0$ in its image:
\begin{equation*}
1. \begin{pmatrix} \bar{A} &\bar{B} & \bar{C} & \bar{D} & \bar{E} \\
\bar{A} &\bar{C} & \bar{D} & \bar{E} & \bar{B}
\end{pmatrix}; 2. \begin{pmatrix} \bar{A} &\bar{B} & \bar{C} & \bar{D} & \bar{E} \\
\bar{B} &\bar{D} & \bar{E} & \bar{A} & \bar{C}
\end{pmatrix}; 
3. \begin{pmatrix} \bar{A} &\bar{B} & \bar{C} & \bar{D} & \bar{E} \\
\bar{C} &\bar{E} & \bar{A} & \bar{B} & \bar{D}
\end{pmatrix}; 
4. \begin{pmatrix} \bar{A} &\bar{B} & \bar{C} & \bar{D} & \bar{E} \\
\bar{D} &\bar{A} & \bar{B} & \bar{C} & \bar{E}
\end{pmatrix}.
\end{equation*}
Note that dynamically the cases $1.$ and $4.$ (as well as $2.$ and $3.$) are equivalent - it suffices to change the orientation of the initial interval.

Note that in the case $1.$ (and $4.$), the map $F$ has an invariant interval $I_A$($I_E$), and the restriction of $F$ on its complement belongs to $\mathrm{FET}^4$. Hence, $F$ is integrable by Proposition \ref{prop:just_before}.

Hence, th study of integrability of the maps in the family $\CETfour$ is reduced to the study of integrability of the maps with the combinatorics 

\begin{equation}\label{eq:combinatorics_CETfour}
\begin{pmatrix} \bar{A} &\bar{B} & \bar{C} & \bar{D} & \bar{E} \\
\bar{B} &\bar{D} & \bar{E} & \bar{A} & \bar{C}
\end{pmatrix}.
\end{equation}

See Figure \ref{fig:combinatorics_for_quadrilaterals} for the illustration of the map with such combinatorics, when $0 \in F(I_B  \cup I_C) $. Here, some additional work is needed to prove that the Rauzy induction can't stop in such a way that the 
restriction $G$ is a non-integrable IET (as in Proposition \ref{prop:non-integrability-for-big-n}). The only possibility for an IET on $4$ (or less) intervals to be non-integrable is exactly to have the combinatorics \eqref{eq:not_simple_IET}. The Rauzy induction can always be continued on an IET with flips (except for the case of the invariant flipped cylinder on the right). Hence, one can indeed suppose that $G \in \mathrm{IET}$. 

The end of the proof is computer assisted. We calculate explicitely the component of the Rauzy graph corresponding to the combinatorics 2. and we verify that the Rauzy induction never stops at the vertices of  Rauzy graph of combinatorial type \eqref{eq:not_simple_IET}.
% For more on this component, see Appendix.

One can show that $F$ has all of its periodic intervals flipped, by repeating the argument of Proposition \ref{prop:integrable_3}. 
\end{proof}

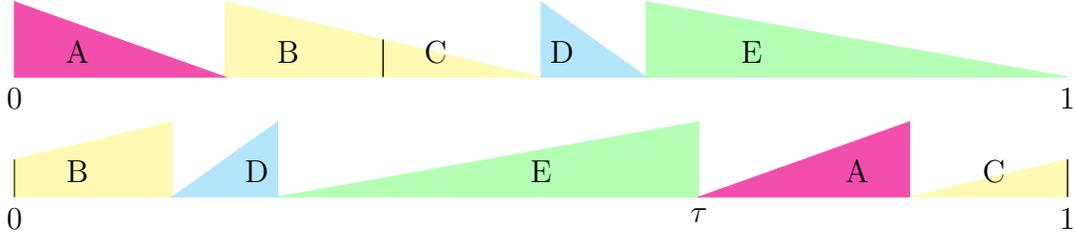
\begin{figure}
\centering
\begin{tikzpicture}[xscale=14]
\path[draw] (0,0) node[below]{$0$} -- cycle;

\path[draw] (1,0) node[below]{$1$} -- cycle;
\path[draw, fill=magenta,magenta, opacity=0.7] (0,0)--(0,1)--(0.2,0)--cycle;
\path[draw] (0.06,0.05) node[above]{A} -- cycle;

\path[draw, fill=yellow,yellow, opacity=0.3] (0.2,0)--(.2,1)--(0.5,0)--cycle;
\path[draw] (0.26,0.05) node[above]{B} -- cycle;
\path[draw] (0.4,0.05) node[above]{C} -- cycle;

\path[draw=black, line width=0.5] (0.35,0)--(0.35,.5) -- cycle;

\path[draw, fill=cyan,cyan, opacity=0.3] (0.5,0)--(.5,1)--(0.6, 0)--cycle;
\path[draw] (0.52,0.05) node[above]{D} -- cycle;

\path[draw, fill=green,green, opacity=0.3] (0.6,0)--(.6,1)--(1, 0) --cycle;
\path[draw] (0.7,0.05) node[above]{E} -- cycle;

\end{tikzpicture}

\begin{tikzpicture}[xscale=14]

\path[draw=black, line width=0.5] (0,0) node[below]{$0$}--(0,.5) -- cycle;

\path[draw, fill=yellow,yellow, opacity=0.3] (0,0)--(0,0.5)--(.15,1)--(0.15,0)--cycle;
\path[draw] (0.06,0.05) node[above]{B} -- cycle;

\path[draw, fill=cyan,cyan, opacity=0.3] (0.15,0)--(.25,0)--(0.25, 1)--cycle;
\path[draw] (0.23,0.05) node[above]{D} -- cycle;

\path[draw, fill=green,green, opacity=0.3] (0.25,0)--(.65,0)--(0.65, 1)--cycle;
\path[draw] (0.5,0.05) node[above]{E} -- cycle;
\path[draw] (0.65,0) node[below]{$\tau$} -- cycle;

\path[draw, fill=magenta,magenta, opacity=0.7] (0.65,0)--(0.85,0)--(0.85,1)--cycle;
\path[draw] (0.8,0.05) node[above]{A} -- cycle;

\path[draw, fill=yellow,yellow, opacity=0.3] (0.85,0)--(1,0)--(1,0.5)--cycle;
\path[draw] (0.93,0.05) node[above]{C} -- cycle;

\path[draw=black, line width=0.5] (1,0) node[below]{$1$} --(1,.5) -- cycle;

\end{tikzpicture}
\caption[]{An element of $\CETfour$ with the combinatorial type 2. Here the lengths $l_j$ and the parameter $\tau$ are such that the image of the second interval contains $0$. The second interval is hence cut into two intervals with labels $B$ and $C$.}\label{fig:combinatorics_for_quadrilaterals}
\end{figure}

We have seen that almost all of the maps in $\CETthree$ as well as in $\CETfour$ are integrable. It turns out that starting from $n=5$, the dynamics of $\CETn$ starts to become more complicated.

\begin{proposition}
There exists an open subset $U$ in the set of parameters $$\mathcal{P}=\left\{(l_1, \ldots, l_n) \in [0,1]^n| l_1 + \ldots + l_n=1\right\} \times \left\{ \tau \left| \right. \tau \in \mathbb{S}^1 \right\}$$ for 
the maps $\CETn, n \geq 5$ such that the dynamics of corresponding maps is non-integrable.
\end{proposition}
\begin{proof}
Indeed, take $U_0:=\left\{(l_1, \ldots, l_n, \tau) \in \mathcal{P} \left|\right. l_1 > 1 - \tau > \max (l_2, \ldots, l_n) \right\}$. In this case the image of the first interval contains $0$ and the combinatorics of a correponding map is the following:
\begin{equation*}
\begin{pmatrix}
\bar{I}_{1}^{-}&\bar{I}_1^{+}&\bar{I}_2&\ldots&\bar{I}_{n-1}&\bar{I}_{n} \\
\bar{I}_{1}^{-}&\bar{I}_2&\bar{I}_3&\ldots&\bar{I}_{n}&\bar{I}_1^{+}
\end{pmatrix}.
\end{equation*}
Let $U$ be the intersection of $U_0$ with the set of the points in $\mathcal{P}$ such that the lengths $l_1-1+\tau, 1-\tau, l_2, \ldots, l_n$ don't have any rational relationships except for trivial ones.

For $F \in \CETn$ such that $(l_1, \ldots, l_m, \tau) \in U$, the corresponding Rauzy induction follows the path
\begin{equation*}
\begin{pmatrix}
\bar{I}_{1}^{-}&\bar{I}_1^{+}&\bar{I}_2&\ldots&\bar{I}_{n-1}&\bar{I}_{n} \\
\bar{I}_{1}^{-}&\bar{I}_2&\bar{I}_3&\ldots&\bar{I}_{n}&\bar{I}_1^{+}
\end{pmatrix}
\xrightarrow[]{I_1^+>I_n, \ldots, I_1^+ > I_2} 
\begin{pmatrix}
\bar{I}_{1}^{-}&I_n&{I}_{n-1}&\ldots&I_2&\bar{I}_{1}^{+} \\
\bar{I}_{1}^{-}&{I}_2&{I}_3&\ldots&{I}_{n}&\bar{I}_1^{+}
\end{pmatrix}.
\end{equation*}
Thus, for $n \geq 5$, the obtained first return map has two invariant flipped intervals and one interval on which the dynamics is given by a non-simple $(n-1)$-IET. Hence, the dynamics of $F$ is not integrable.
\end{proof}

\section{The properties of the orbits in triangular tiling billiards.}\label{sec:properties_tiling_billiards}

Now we get back to the properties of triangle tiling billiards and 
prove the results on their dynamics. For that, we will use the machinery prepared in previous sections as well as some additional arguments.

\subsection{Symbolic dynamics and combinatorial orbits.}\label{subs:definition_HW}

Fix a triangular tiling by the tiles congruent to the triangle $\Delta$ with vertices $A,B,C$ and sides $a,b,c$  on the plane . To any trajectory $\delta$ on this tiling one can associate a bi-infinite word $\omega(\delta)= \ldots w_{-2} w_{-1} w_0 w_1 w_2 \ldots $ in the alphabet $\mathcal{A}=\{a,b,c\}$ by writing the labels of sides in the tiling intersected along $\delta$. By fixing a starting triangle, one fixes a special slot in the symbolic word. Note that singular trajectories (passing by the vertices) may give  several codings with common past or future since they branch out.

One can say that the symbolic sequence $\omega(\delta)$ characterizes how the trajectory is seen with a naked eye (that can't precisely measure the angles but sees the sequence of triangles that are crossed). 

\begin{example}
On Figure \ref{figure: different behaviors} the first depicted trajectory is a closed trajectory corresponding to the sequence $\bar{\omega} \in \mathcal{A}^\Z$ where $\omega=bcababcaba$ and the names $a,b,c$ are given to the sides in the decreasing order of lengts. One can see that by fixing the starting triangle (grey triangle on the Figure \ref{figure: different behaviors}), one fixes the zero position symbol of the word $\omega$ and we have $w_0=b$.
\end{example}

\begin{definition}
Define $\mathcal{S} \subset \A^{\Z}$ to be the set of all bi-infinite words in the alphabet $\A$ that correspond to symbolic dynamics of (possibly, singular) trajectories in triangle tiling billiards. Define ${\mathcal{S}}_{\Delta}$ the analogous set for the triangle tiling billiard in a tiling defined by the fixed triangle $\Delta$. We call the elements in these sets \emph{symbolic orbits} of triangle tiling billiards.
\end{definition}

It is easy to see that $\mathcal{S}$ is strictly contained in $\mathcal{A}^{\Z}$. Indeed, the words in $\mathcal{S}$ don't contain $aa, bb$ or $cc$ as a subword.  for since a trajectory can't pass by the same side consequently. One can also prove that for $c$ being the shortest side of $\Delta$, the words in $\mathcal{S}_{\Delta}$ don't contain the subwords $cbc$ and $cac$. In general, the set $\mathcal{S}_{\Delta}$ is a "small" set: the corresponding dynamics has zero topological entropy and linear complexity. The precise understanding of the structure of $\mathcal{S}_{\Delta}$ is not an easy question, see for example the discussion about the Conjecture \ref{conj:tree}. 
%paragraph \ref{subs:tree} on the Tree Conjecture.

\smallskip

For a fixed triangle $\Delta$, the study of the triangle tiling billiard with the tiles congruent to $\Delta$, is equivalent to that of the map $F_{\tau, l_1, l_2, l_3} \in \CETthree$, restricted onto some subset of the circle (as explained in paragraph \ref{subs:connections}), with $l_j$ being the normalized angles of $\Delta$ and $\tau$ a parameter. Since $F$ is fully flipped, then its square $T_{\tau, l_1, l_2, l_3}=F^2: \Sph^1 \rightarrow \Sph^1$ is an interval exchange transformation. An accelerated (two steps in one) triangle tiling trajectory (corresponding to $T$) goes through the triangles of the same orientation on the circumcircle. Define  for any pair $ w_0, w_1 \in \mathcal{A}$ the following subset of $\Sph^1$:

\begin{equation}\label{eq:sets_symbolic}
I_{\omega_0 \omega_1}:=\left\{
x \in \Sph^1 \left| \right. x \in I_{\omega_0}, F(x) \in I_{\omega_1} 
\right\}.
\end{equation}

Such a subset is a segment of continuity for $T$. Since $F$ has $3$ singular points on $ \Sph^1 $, then $T$ has at most $6$ singular points on $ \Sph^1 $. From this we see that at least $3$ of these subsets are empty. 

We associate to any symbolic orbit $\omega \in \mathcal{S}_{\Delta}$, a word $\Omega$ in the alphabet of vertices $\{A,B,C\}$ simply by replacing small letters by capital ones. Then for those $x \in \Sph^1 $ that correspond to the trajectories of triangle tiling billiard, $\overrightarrow{\Omega_0 \Omega_1}$ considered as a vector on the plane, lies in a following set: 
\begin{equation}\label{eq:possible directions}
\mathcal{V}=\left\{
\pm \overrightarrow{AB}, \pm \overrightarrow{BC}, \pm \overrightarrow{AC}
\right\}
\end{equation}

\begin{definition}
Fix a triangle $\Delta$ on the plane with positive orientation (such that the curve $ABC$ has the triangle $\Delta$ on its left). Let the origin $(0,0) \in \mathbb{R}^2$ be at the barycenter of $\Delta$. Consider a trajectory $\delta$ of a triangle tiling billiard starting in $\Delta$ and its symbolic orbit $\omega$ with an associated word $\Omega$ in the alphabet of vertices. Then the \emph{combinatorial orbit }of $\delta$ is a piecewise linear curve that connects by straight segments the following points in the following order:
\begin{equation*}
\ldots, -\sum_{j=0}^{k} \overrightarrow{{\Omega}_{-2(j+1)} {\Omega}_{-2(j+1)+1}}
 ,\ldots,-\overrightarrow{{\Omega}_{-2} {\Omega}_{-1}},0,\overrightarrow{{\Omega}_0 {\Omega}_1},\overrightarrow{{\Omega}_0 {\Omega}_1}+\overrightarrow{{\Omega}_2 {\Omega}_3}, \ldots, 
\sum_{j=0}^{k} \overrightarrow{{\Omega}_{2j} {\Omega}_{2j+1}}, \ldots
\end{equation*}
\end{definition}
One can easily see that the points where the combinatorial orbit  breaks (is not smooth) are exactly the barycenters of positively oriented triangles through which $\delta$ passes, and that the segments of the combinatorial orbit are parallel to the sides of the triangles in a tiling.

This curve was already studied by Hooper and Weiss in \cite{HW18}, where they were interested in rel leaves of some special class of translation surfaces, being one-parameter deformations of Arnoux-Yoccoz translation surface. By finding a triangle tiling billiard associated to this case (which was done in \cite{BDFI18}), one can see that its combinatorial orbits are abelianizations of other interesting curves, studied by McMullen \cite{M12} and Arnoux \cite{ABB11, A88}. See more on this case in paragraph \ref{subs:HW_subsec}.

%Each trajectory defines its own Hooper-Weiss curve. Note that the segments in a Hooper-Weiss curve are parallel to the sides of the triangle $\Delta$ in the plane and the symbolic displacement $\omega_0 \omega_1$ corresponds to the displacement along the vector $\overrightarrow{\Omega_0 \Omega_1}$, where $\Omega_j$ is a corresponding vertex, $\Omega_j \in \{A,B,C\}$ (vertex $A$ is a vertex in front of the side $a$ etc.).

\begin{example}
A combinatorial orbit for the first trajectory of Figure \ref{figure: different behaviors} is a union of segments connecting the consecutive results of the sums of vectors $\overrightarrow{BC}+\overrightarrow{AB}+\overrightarrow{AB}+\overrightarrow{CA}+\overrightarrow{BA}$.
\end{example}

Obviously, from Proposition \ref{prop:Dianacomeback} follows
\begin{lemma}\label{prop:HW}
A combinatorial orbit of $\delta$ is a closed (drift-periodic, linearly escaping) curve if and only if the $\delta$ is closed (drift-periodic, linearly escaping).  
\end{lemma}

\subsection{Properties of periodic orbits.}\label{subs:properties_periodic}
Closed and drift periodic trajectories both have their symbolic orbits periodic. Although, the balance properties of each periodic word $\omega(\delta) \in \mathcal{S}$ show if $\delta$ is a periodic or a drift periodic trajectory. Indeed, take a finite word $s \in \mathcal{A}^L$ of length $L$ corresponding to the geometric period of $\delta$. Then $\omega=\bar{s}$ and $L \in 2 \Z$. The corresponding portion of a combinatorial orbit of length $\frac{L}{2}$ is a union of segments connecting the consecutive results in the sum $\overrightarrow{\Omega_0 \Omega_1}+ \overrightarrow{\Omega_2 \Omega_3}+ \ldots+ \overrightarrow{\Omega_{L-1} \Omega_{L}}$. Obviously,

\begin{proposition}
A trajectory $\delta$ in a triangle tiling billiard is closed (drift-periodic) if and only if a corresponding symbolic word $w(\delta)$ is periodic $w=\bar{s}$ and the sum $\overrightarrow{\Omega_0 \Omega_1}+ \overrightarrow{\Omega_2 \Omega_3}+ \ldots+ \overrightarrow{\Omega_{L-1} \Omega_{L}}$ is equal (not equal) to zero.
\end{proposition}

\begin{example}
For the third (drift-periodic) trajectory on Figure \ref{figure: different behaviors}, the coding $\omega=\bar{s}$ with $s=bababc$, and $\overrightarrow{BA}+\overrightarrow{BA}+\overrightarrow{BC}\neq 0$.
\end{example}

\begin{proposition}
Suppose that the symbolic coding of some trajectory $\delta$ in a triangle tiling billiard is periodic, i.e. there exists a word $s$ such that $\omega=\bar{s}$. Suppose additionally that $s$ has odd length. Then the trajectory $\delta$ is closed. 
\end{proposition}
\begin{proof}
Indeed, $\omega=\overline{s}=\overline{s^2}$ where $s=s_0 \ldots s_{l-1}, |s|=l, l\in 2\Z+1$. Define $s_j$ for all $j=l, \ldots, 2l-1$ by $s_{j+l}:=s_{j+l}]$. In order to prove the statement, it suffices to show, by Lemma \ref{prop:HW}, that the sum of the vectors $\overline{S}_{l}:=\sum_{k=0}^{l-1} \overrightarrow{s_{2k} s_{2k+1}}$ is equal to $0$. 

Proof is by induction on $l$ and uses only simple linear algebra. For $l=1$, we have $\overline{S}_{1}=\overline{s_0 s_0}=0$. Suppose the statement is proven for some odd $l$ and $\overline{S}_{l}=0$ for any word $\omega$. Let us prove it for $l+2$. 
%Now we have the vector $\overline{S}_{l+2}$ and we want to prove that it is equal to zero. This is easy to prove since the difference $\overline{p}_{l+2}=\overline{p}_{l} $is a sum of a finite, non-depending on $l$ number of vectors.

Indeed, now we have $s^2=s_0 \ldots s_{l+1} s_{l+2} \ldots s_{2l+1}$ with $s_j=s_{j+l+2} \; \; \forall j \in [[0, l+1]]$. Then we have

\begin{multline*}
\overline{S}_{l+1}=\sum_{k=0}^{l+1} \overrightarrow{s_{2k} s_{2k+1}}=\sum_{k=0}^{\frac{l-3}{2}} \overrightarrow{s_{2k} s_{2k+1}} + \overrightarrow{s_{l-1} s_l}+\overrightarrow{s_{l+1} s_{0}}+ \sum_{k=\frac{l+3}{2}}^{l} \overrightarrow{s_{2k} s_{2k+1}} + \overrightarrow{s_{2l+2} s_{2l+3}} =\\
=  \left(\sum_{k=0}^{\frac{l-3}{2}} \overrightarrow{s_{2k} s_{2k+1}} + \overrightarrow{s_{l-1} s_0} +
\sum_{k=\frac{l+3}{2}}^{l} \overrightarrow{s_{2k} s_{2k+1}}
\right)-\overrightarrow{s_{l-1} s_0} + \overrightarrow{s_{l-1} s_l}+\overrightarrow{s_{l+1} s_{0}} + \overrightarrow{s_{2l+2} s_{2l+3}}=\\
= \overrightarrow{s_{l+1} s_{0}}+ \overrightarrow{s_0 s_{l-1}}+\overrightarrow{s_{l-1} s_l}+ \overrightarrow{s_{l} s_{l+1}}=0.
\end{multline*}
\end{proof}

\subsection{Generic behavior of trajectories in triangle tiling billiards.}
In this Section we use the results of previous Section \ref{sec:integrability_section} in order to characterize the qualitative behavior of triangle tiling billiards. The following Proposition describes the generic behavior of triangle tiling billiards and answers the following question. If one picks a random triangle and a random trajectory in a corresponding tiling, how does this trajectory look like ? The answer is: it is either closed, or linearly escaping with an irrational slope.

\begin{proposition}\label{thm:generic}
Let $C=\left\{
(l_1, l_2, l_3) \in {\R}^3_+ \left|\right. l_1+l_2+l_3=1
\right\}$ be a simplex of normalized angles of triangles $\Delta$, $l_j=\frac{\alpha_j}{\pi}$. Consider a subset of $\mathcal{C} \subset C$ of lengths independent over $\Q$. Consider the set of all pairs $(\delta, \Delta)$ such that the lengths of $\Delta$ belong to $\mathcal{C}$ and the trajectory $\delta$ does not pass through the circumcenters of the crossed triangles. Then the following holds:
\begin{itemize}
\item[1.]The orbit $\delta$ is not drift-periodic orbits.
\item[2.] If $\delta$ is closed then it has a symboling coding $\overline{s^2}$ for some word $s$ in the alphabet $a,b,c$ of odd length. Consequently, its period is equal to $4n+2$ for some $n \in \N^*$, and its combinatorial orbit has central symmetry.
\item[3.] If $\delta$ is not closed, it is linearly escaping. Its symbolic orbit is described by an infinite word $\omega$ that can be represented as an infinite concatenation of two finite words $\omega_1$ and $\omega_2$ in the alphabet $\mathcal{A}$ in some order corresponding to a sturmian sequence. 
\end{itemize}

\end{proposition}

\begin{proof}
Assertion $1.$ has been proven in \cite{BDFI18}, see Proposition 2.15. 

Then, for any point in the set $\mathcal{C}$, and any trajectory $\delta$ not passing through the circumcenters, the Rauzy induction for a corresponding  $F=F_{\tau, l_1, l_2, l_3} \in \CETthree$ stops, by Theorem \ref{thm:first_ingredient} since $\tau \neq \frac{1}{2}$. By Theorem \ref{thm:CETthree}, $F$ is integrable with a simple first-return map on the Poincaré-Rauzy section. A point $x \in \mathbb{S}^1$ has a closed orbit if and only if it is contained in a flipped interval $I$ on this section, or in its orbit by $F$ (its Rokhlin tower). But $I$ maps onto itself by the first return map which, in restriction to $I$ is equal to $F^d, d=2n+1 \in 2 \N+1$. Then the period of the billiard orbit is exactly $2d = 2 (2n+1)=4n+2$, and its combinatorial orbit is centrally symmetric.

The escaping trajectories correspond to the orbits of the points on the Poincaré-Rauzy section which are not flipped. By integrability, they correspond to rotations. The symbolic coding of the first return map to the section Poincaré-Rauzy defines the words $\omega_j, j=1,2$. For the case of rotation with a marked point ($3$-IET) as in point $3.$ of Definition \ref{def:simple}, its symbolic dynamics can be reduced to the case of $2$-IET by considering a smaller Poincaré section.
\end{proof}

In paragraph \ref{subs:properties_periodic} we have proven that the odd length codings correspond to the periodic orbits. The converse is true as well. 

\begin{proposition}\label{thm:period}
Consider any closed trajectory  $\delta$ in a triangle tiling billiard. Then the minimal period of its symbolic coding $\omega(\delta)$ has an odd length. Consequently, its period is equal to $4n+2$ for some $n \in \N^*.$
\end{proposition}
\begin{proof}
The existence of a periodic trajectory implies that the Rauzy induction stops for the corresponding $F \in \CETthree$, by Lemma \ref{lemma:minimality}. If the Rauzy induction stops then either one is in the conditions of Proposition \ref{thm:generic} (the lengths are independent over $\Q$), and there is nothing to prove, or the lengths $l_j$ are rationally dependent. In this case, we perturb the point in $C$ in such a way that this dependence is destroyed. If the perturbation is small enough, the symbolic coding (as well as the length) of the periodic trajectory stay fixed, by Proposition \ref{prop:Dianacomeback}. One can then go back to applying Proposition \ref{thm:generic}.
\end{proof}

\begin{remark}
The periods of drift-periodic orbits are not necessarily of length $4n+2$. For example, take a triangle with lengths of the sides $a,b,c$ equal to $5$, $7$ and $8$ and $\tau \approx 1/2$. One can verify that the orbit passing through the middle of the side $b$ is drift-periodic with a drift period $24$.
\end{remark}

By looking at the simulations of trajectories, one can suggest a much stronger property of closed orbits which is the following conjecture which was first stated in \cite{BDFI18}, see Conjecture 5.3.

\begin{conjecture}[Tree conjecture]\label{conj:tree}
Let $\Lambda$ be as a union of all vertices and edges of all drawn triangles in a periodic triangle tiling.
Take any periodic closed trajectory $\delta$ of a corresponding triangle tiling billiard. It incloses some bounded domain $U \subset \mathbb{R}^2$ in the plane, $\partial U = \delta$ and $U \cap \Lambda$ is an embedding of some graph in the plane. Then this graph is a \emph{tree}. 
\end{conjecture}
In \cite{BDFI18} this conjecture is proven for obtuse triangles (the tree in question is in this case a chain). This conjecture has an even stronger form corresponding to the fact that any trajectory (not necessarily closed) fills in the subset of the plane that it occupies.

\subsection{Description of the family $\mathrm{CET}^{3}_{1/2}$ and Arnoux-Rauzy family.}\label{subs:HW_subsec}

The generic behavior of a map $F_{\tau, l_1, l_2, l_3}$ from $\CETthree$ has been understood thanks to Theorem \ref{thm:CETthree}: it is integrable, and corresponding billiard trajectories are either closed or linearly escaping. Although, there exists a subfamily of maps in $\CETthree$ with minimal behavior. Indeed, the authors of \cite{BDFI18} remark that for $\tau=\frac{1}{2}$ and the vector of lengths being chosen as

\begin{equation}\label{eq:Dianas-parameters}
(l_1,l_2, l_3)=\left(\frac{1-x^3}{2},\frac{1-x^2}{2}, \frac{1-x}{2}\right),
\end{equation}
for $x$ the real solution of the algebraic equation
\begin{equation}\label{eq:tribonacci}
x+x^2+x^3=1,
\end{equation}
the symbolic behavior of the corresponding triangle tiling billiard trajectories seems to be quiet different from the expected generic behavior - they seem to escape to infinity non-linearly. The authors in \cite{BDFI18} notice that these trajectories "approach Rauzy fractal", though the statement in their paper is euristic (based on computer simulations). The object obtained in \cite{BDFI18}, i.e. the rescaled limit of combinatorial orbits for these parameters, is exactly the same than that defined in  \cite{HW18} outside the context of triangle tiling billiards.\footnote{For more on the relationship between the work \cite{HW18} and triangle tiling billiards, see in \cite{BDFI18}.} Hooper and Weiss already conjectured the convergence (up to rescaling and a uniform affine coordinate change) of these combinatorial orbits to the Rauzy fractal in Hausdorff topology, see the discussion at the end of Section 4 in \cite{HW18}.  This conjecture remains open, is one of our main motivations for future research, as well as the understanding of the non-linear escape in triangle tiling billiards. This motivates this paragraph's discussion and the following 

\begin{definition}[\cite{A88}, \cite{AR91}]\label{def:AR}
The \emph{Arnoux-Rauzy family} is a $3$-parametric family of the maps in $\mathrm{CET}^6$ defined as follows. The circle is identified with the interval $[0,1)$ and is cut into $6$ intervals of lengths $\frac{x_j}{2}, j=1,2,3$, $\sum_j x_j=1$ : 
\begin{equation*}
[0,1)=\left[0, \frac{x_1}{2}\right) \cup \left[\frac{x_1}{2}, x_1 \right) \cup \left[x_1, x_1+\frac{x_2}{2}\right) \cup \left[x_1+\frac{x_2}{2}, x_1+x_2\right) \cup \left[1-x_3, 1-\frac{x_3}{2}\right) \cup \left[1-\frac{x_3}{2},1\right).
\end{equation*}
Each transformation from Arnoux-Rauzy family is a composition of two maps. First, the intervals of equal lengths are exchanged. Second, the circle is rotated by $\frac{1}{2}$. 
\end{definition}

\begin{remark}
The two maps  in this definition are non-commuting involutions.
\end{remark}

The famous Arnoux-Yoccoz example (with the parameters $(x_1,x_2,x_3)=(x,x^2,x^3)$ with $x$ being the solution of \eqref{eq:tribonacci}), belongs to this Arnoux-Rauzy family. 
The Arnoux-Rauzy family has been extensively studied \cite{A88, AR91, AHS16, AHS16-1} but, as far as we know, it was never noticed that this family has a natural square root in the set of fully flipped interval exchange transformations.

\begin{proposition}\label{prop:the_squares_are_Rauzy}
The set of the squares
\begin{equation*}
\left\{ 
F^2 \mid \; F_{\tau, l_1, l_2, l_3} \in \CETthree, \tau=\frac{1}{2}, l_1, l_2, l_3 < \frac{1}{2}
\right\}
\end{equation*}
 is exactly Arnoux-Rauzy family of interval exchange transformations on the circle. 
\end{proposition}

\begin{proof}
The proof is a direct computation. Without loss of generality, we can suppose that $l_3 \leq l_2 \leq l_1$ by changing the orientation and possibly replacing $\tau$ by $1-\tau$. Consider a following subdivision of initial intervals of continuity:
\begin{align*}
 I_1=(0, \frac{1}{2}-l_2) \cup (\frac{1}{2}-l_2,l_1)=: J_2^+ \cup J_3^-\\
 l_2=(l_1, \frac{1}{2}+l_1-l_3) \cup (\frac{1}{2}+l_1-l_3, l_1+l_2)=:J_3^+ \cup J_1^-\\
 l_3=(l_1+l_2, l_2+\frac{1}{2}) \cup (l_2+\frac{1}{2},1)=:J_1^+ \cup J_2^-.
 \end{align*}

Obviously, $F$ can be seen as IET with flips on the \emph{interval} $[0,1)$ with intervals of continuity $J_2^+, J_3^-, J_3^+, J_1^-, J_1^+, J_2^-.$ These intervals are exchanged following the permutation
$$
\begin{pmatrix} \overline{J_2^+} & \overline{J_3^-} & \overline{J_3^+} &  \overline{J_1^-} & \overline{J_1^+} & \overline{J_2^-}\\  \overline{J_3^+} & \overline{J_2^-} & \overline{J_1^+} & \overline{J_3^-} & \overline{J_2^+} & \overline{J_1^-}
\end{pmatrix}.
$$
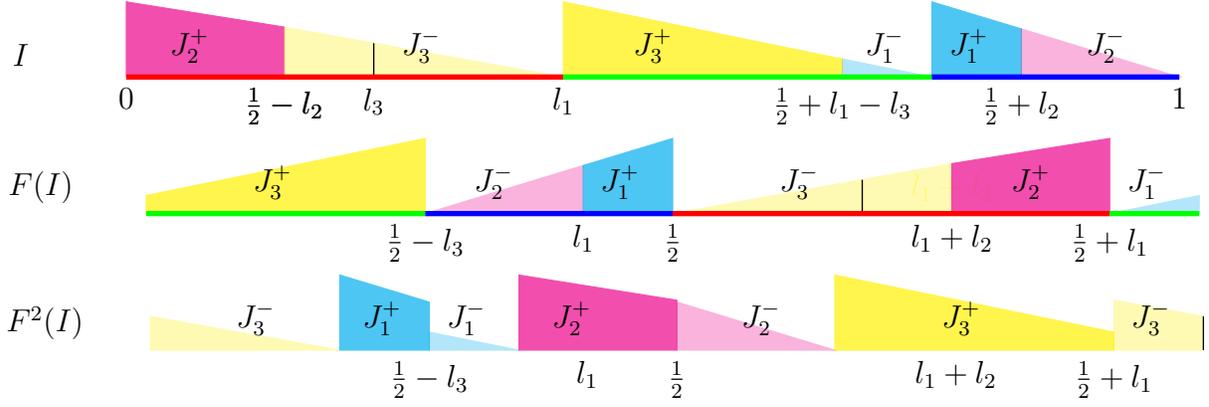
\begin{figure}
\centering
\begin{tikzpicture}[xscale=14]
\path[draw] (-0.1,0) node[above]{$I$} --cycle;
\path[draw, fill=magenta,magenta, opacity=0.7] (0,0)--(0.15,0)--(0.15,0.665)--(0,1)--cycle;
\path[draw, fill=yellow,yellow, opacity=0.3] (0.15,0)--(.415,0)--(0.15,0.665)--cycle;
\path[draw=black, line width=0.5] (0.235,0) node[below]{$l_3$}--(0.235,.45) -- cycle;
\path[draw, fill=yellow,yellow, opacity=0.7] (.415,0)--(.68,0)--(0.68, 0.24)--(.415,1)--cycle;
\path[draw, fill=cyan,cyan, opacity=0.3] (0.68,0)--(.765,0)--(0.68, 0.24)--cycle;
\path[draw, fill=cyan,cyan, opacity=0.7] (.765,0)--(0.85,0)--(.85,0.64)--(.765,1)--cycle;
\path[draw, fill=magenta,magenta, opacity=0.3] (0.85,0)--(.85,0.64)--(1,0)--cycle;
%\path[draw] (0.5,0) node[below]{$\frac{1}{2}$} -- cycle;
\path[draw] (0.15,0) node[below]{$\frac{1}{2}-l_2$} --cycle;
\path[draw] (0.15,0) node[below]{$\frac{1}{2}-l_2$} --cycle;
\path[draw] (0.68,0) node[below]{$\frac{1}{2}+l_1-l_3$} --cycle;
\path[draw] (0.85,0) node[below]{$\frac{1}{2}+l_2$} --cycle;
\path[draw] (0,0) node[below]{$0$} --cycle;
\path[draw] (1,0) node[below]{$1$} --cycle;
\path[draw] (.415,0) node[below]{$l_1$} --cycle;
\draw[-][draw=red, line width=2] (0,0) -- (.415,0);
\draw[-][draw=green, line width=2]  (.415,0) -- (0.765,0);
%node[below]{$1-l_3$};
\draw[-][draw=blue, line width=2]  (0.765,0) -- (1,0);

\path[draw] (0.06,0.05) node[above]{$J^+_2$} -- cycle;
\path[draw] (0.28,0.05) node[above]{$J^-_3$} -- cycle;
\path[draw] (0.5,0.05) node[above]{$J^+_3$} -- cycle;
\path[draw] (0.72,0.05) node[above]{$J^-_1$} -- cycle;
\path[draw] (0.8,0.05) node[above]{$J^+_1$} -- cycle;
\path[draw] (0.93,0.05) node[above]{$J^-_2$} -- cycle;

\end{tikzpicture}

\begin{tikzpicture}[xscale=14]
\path[draw] (-0.1,0) node[above]{$F(I)$} --cycle;
\path[draw, fill=yellow,yellow, opacity=0.7] (0,0)--(0,0.24)--(0.265,1)--(.265,0)--cycle;
\path[draw, fill=magenta,magenta, opacity=0.3] (0.265,0)--(0.415,0)--(0.415,0.64)--cycle;
\path[draw, fill=cyan,cyan, opacity=0.7] (0.415,0)--(0.415,0.64)--(0.5, 1)--(0.5,0)--cycle;
\path[draw, fill=yellow,yellow, opacity=0.3] (.5,0)--(.765,.665) node[below]{$l_1+l_2$}--(0.765, 0)--cycle;
\path[draw=black, line width=0.5] (.68,0) --(.68,.45) -- cycle;
\path[draw, fill=magenta,magenta, opacity=0.7] (0.765,0)--(0.765,0.665)--(0.915,1)--(0.915,0)--cycle;
\path[draw, fill=cyan,cyan, opacity=0.3] (0.915,0)--(1,0.24)--(1, 0)--cycle;
\path[draw] (0.5,0) node[below]{$\frac{1}{2}$} -- cycle;
\path[draw] (0.265,0) node[below]{$\frac{1}{2}-l_3$} --cycle;
%\path[draw] (0.15,0) node[below]{$\frac{1}{2}-l_2$} --cycle;
\path[draw] (0.915,0) node[below]{$\frac{1}{2}+l_1$} --cycle;
\path[draw] (0.765,0) node[below]{$l_1+l_2$} --cycle;
%\path[draw] (0,0) node[below]{$0$} --cycle;
%\path[draw] (1,0) node[below]{$1$} --cycle;
\path[draw] (.415,0) node[below]{$l_1$} --cycle;
\draw[-][draw=red, line width=2] (0.5,0) -- (.915,0);
\draw[-][draw=green, line width=2]  (.915,0) -- (1,0);
\draw[-][draw=green, line width=2]  (0,0) -- (.265,0);
\draw[-][draw=blue, line width=2]  (0.265,0) -- (.5,0);

\path[draw] (0.12,0.05) node[above]{$J^+_3$} -- cycle;
\path[draw] (0.33,0.05) node[above]{$J^-_2$} -- cycle;
\path[draw] (0.45,0.05) node[above]{$J^+_1$} -- cycle;
\path[draw] (0.62,0.05) node[above]{$J^-_3$} -- cycle;
\path[draw] (0.84,0.05) node[above]{$J^+_2$} -- cycle;
\path[draw] (0.95,0.05) node[above]{$J^-_1$} -- cycle;
\end{tikzpicture}

\begin{tikzpicture}[xscale=14]
\path[draw] (-0.1,0) node[above]{$F^2(I)$} --cycle;
\path[draw, fill=yellow,yellow, opacity=0.3] (.915,0)--(.915,0.665)--(1,0.45)--(1,0)--cycle;
\path[draw=black, line width=0.5] (1,0) --(1,.45) -- cycle;
\path[draw, fill=yellow,yellow, opacity=0.3] (0,0)--(0,0.45)--(.18,0)--cycle;
\path[draw, fill=cyan,cyan, opacity=0.7] 
(0.18,0)--(0.18,1)--(0.265, 0.64)--(0.265,0)--cycle;
\path[draw, fill=cyan,cyan, opacity=0.3] 
(0.265,0)--(0.265,0.24)--(0.35,0)--cycle;
\path[draw, fill=magenta,magenta, opacity=0.7] 
(0.35,0)--(0.35,1)--(.5,0.665)--(.5,0)--cycle;
\path[draw, fill=magenta,magenta, opacity=0.3] 
(.5,0)--(.5,.64)--(.65,0)--cycle;
\path[draw, fill=yellow,yellow, opacity=0.7] 
(.65,0)--(.65, 1)--(.915, .24)--(.915,0)--cycle;
\path[draw] (0.5,0) node[below]{$\frac{1}{2}$} -- cycle;
\path[draw] (0.265,0) node[below]{$\frac{1}{2}-l_3$} --cycle;
%\path[draw] (0.15,0) node[below]{$\frac{1}{2}-l_2$} --cycle;
\path[draw] (0.915,0) node[below]{$\frac{1}{2}+l_1$} --cycle;
\path[draw] (0.765,0) node[below]{$l_1+l_2$} --cycle;
%\path[draw] (0,0) node[below]{$0$} --cycle;
%\path[draw] (1,0) node[below]{$1$} --cycle;
\path[draw] (.415,0) node[below]{$l_1$} --cycle;

\path[draw] (0.1,0.05) node[above]{$J^-_3$} -- cycle;
\path[draw] (0.22,0.05) node[above]{$J^+_1$} -- cycle;
\path[draw] (0.3,0.05) node[above]{$J^-_1$} -- cycle;
\path[draw] (0.4,0.05) node[above]{$J^+_2$} -- cycle;
\path[draw] (0.58,0.05) node[above]{$J^-_2$} -- cycle;
\path[draw] (0.77,0.05) node[above]{$J^+_3$} -- cycle;
\path[draw] (0.95,0.05) node[above]{$J^-_3$} -- cycle;

%\draw[-][draw=red, line width=2] (0,0) -- (.415,0);
%\draw[-][draw=green, line width=2]  (.415,0) -- (0.765,0);
%%node[below]{$1-l_3$};
%\draw[-][draw=blue, line width=2]  (0.765,0) -- (1,0);
\end{tikzpicture}
\caption[]{The action of the map $F$ and its square $T=F^2$ on the subdivision of $\mathbb{S}^1$ by six intervals $J_j^{\pm}, j=1,2,3$. The three levels of the Figure correspond to the iterations of $F$ ($F^0=\mathrm{Id}, F^1=F, F^2=T$). One can see that $F^2 \in \mathrm{CET}^6$ and $F^2 \in \mathrm{IET}^7$. Indeed, one has to make a cut $J_3^-=(\frac{1}{2}-l_2, l_3) \cup (l_3,l_1)$ in order to pass from the CET on $6$ intervals to an IET on $7$ intervals. The intervals of the same length in the subdivision neighbour each other on the first level as well as on the last level, after applying $T$ but their order in a couple is reversed, exactly defining the Arnoux-Rauzy family.}\label{fig:squaredmap}
\end{figure}

Applying once more the map $F$ to the image of the interval, one obtains the final result : $T=F^2$ is such that $F^2 \in \mathrm{CET}^6$ and $F^2 \in \mathrm{IET}^7$. The order of the corresponding permutation of $6$ intervals on the circle is represented on the Figure \ref{fig:squaredmap}.

Define $I_{j,k}$ as 
\begin{equation}\label{eq:symbolicdefinitiontwosteps}
I_{j,k}=\left\{
x \in I_j \left| \right. F(x) \in I_k
\right\}.
\end{equation}
for $j,k=1,2,3$. Then one can see that exactly 
\begin{align*}
J_1^{-}=I_{2,3}, \;\; J_1^{+}=I_{3,1};\\
J_2^{-}=I_{3,1},  \; \; J_2^{+}=I_{1,3};\\
J_3^{-}=I_{1,2}, \;\;  J_3^{+}=I_{2,1}
\end{align*}

and that these intervals come by couples of equal length : 
$$|I_{1,2}|=|I_{2,1}|=\frac{1}{2}-l_3, |I_{1,3}|=|I_{3,1}|=\frac{1}{2}-l_2, |I_{3,2}|=|I_{2,3}|=\frac{1}{2}-l_1.$$

Combinatorics and the lengths defined the map $T=F^2$ and show that this is exactly a map from the Arnoux-Rauzy family and that all the maps are represented. The parameters $x_j$ in Defintion \ref{def:AR} are given by the affine change from the lengths of flipped intervals $I_j, j=1,2,3$: 
\begin{equation}\label{eq:relation_x_and_l}
x_j=1-2l_j, j=1,2,3.
\end{equation}
Here $x_1 \geq x_2 \geq x_3$.
\end{proof}

\begin{remark}
The fact that the intervals split into couples of equal length for the square of the map, is geometrically obvious from the symmetry, for any map in $\mathrm{CET}^n_{\frac{1}{2}}$, Indeed, the lengths  correspond to the lengths of the intervals that an inscribed polygon and its reflection with respect to its circumcenter cut out on the circumscribed circle (the intervals $I_{i,j}$ and $I_{j,i}$ being centrally symmetric, they have equal lengths). Although, the combinatorics has to be precised, see Proposition above.
\end{remark}

Let us note that for any $F \in \CETthreehalf$ corresponding to an obtuse triangle its dynamics is completely periodic.

\begin{proposition}\label{prop:obtuse}
For any $F \in \CETthreehalf$ such that there exists $l_j>\frac{1}{2}$, $F$ is completely periodic.
\end{proposition}
\begin{proof}
One can suppose $l_1>\frac{1}{2}$. Then, the intervals $I_1:=(0, l_1-\frac{1}{2})$ and $I_2:=(\frac{1}{2}, l_1)$ are invariant by $F$ and consist of periodic orbits of length $2$. The dynamics of $F$ restricted on the interval $[0,1]\setminus \left( I_1 \cup I_2 \right) $ one can see that this dynamics, is equivalent to the dynamics of an $3$-IET with flips with combinatorics 
\begin{equation}\label{eq:combi}
\begin{pmatrix}
\bar{A}&\bar{B}&\bar{C}\\
\bar{B}&\bar{C}&\bar{A}
\end{pmatrix}.
\end{equation}
Such a system is integrable by Proposition\ref{prop:ietfthree_integrability} and, as can be explicitely verified in this case, via a more precise study of the corresponding component of the Rauzy graph, completely periodic. See Figure \ref{fig:component_ABC_BCA} for illustration.

\begin{figure}
\centering
\includegraphics[scale=0.5]{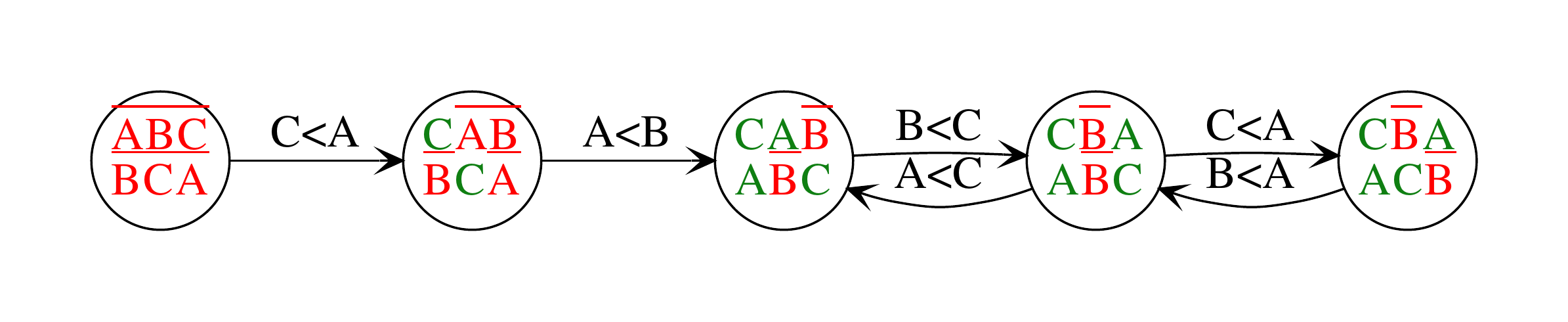}
\caption{Connected component of the vertex \eqref{eq:combi} in the Rauzy graph of $\IETFthree$. }\label{fig:component_ABC_BCA}
\end{figure}
\end{proof}

\subsection{Non-linearly escaping orbits and the Rauzy gasket.}
The hero of this paragraph is the following object. 

\begin{definition}[\cite{AS13}]\label{defn:Rauzy_gasket}
Consider a point $(x_1, x_2, x_3) \in \R^3_+$ such that $x_1+x_2+x_3=1$. Apply the following algorithm: 
\begin{itemize}
\item[1.] If one of the entries, say $x_1$, is greater then the sum of the two smaller ones, obtain a new point $(x_1-x_2-x_3, x_2, x_3)$ with positive entries. Analogically, for $x_2>x_1+x_3 $ (or $x_3>x_1+x_2$), one defines a new point $(x_1,x_2-x_1-x_3,x_3)$ (or $(x_1,x_2,x_3-x_1-x_2)$).
\item[2.] Rescale a new point so that the sum of the entries is equal to $1$.
\end{itemize}
The \emph{Rauzy gasket} $\boldsymbol{\mathcal{R}}$ is the set of points in the simplex on which this algorithm can be applied infinitely many times, i.e. the step $1.$ is always possible.
\end{definition}

The Rauzy gasket  $\boldsymbol{\mathcal{R}}$ has zero Lebesgue measure \cite{AS13} and its properties were extensively studied in \cite{AR91, AHS16, AHS16-1}. 
The relationship of the set $\boldsymbol{\mathcal{R}}$ to the triangle tiling billiards is the following. Proposition \ref{prop:the_squares_are_Rauzy} shows that the study of the dynamics of trajectories of acute triangle tiling billiards passing throught the circumcenter is equivalent to the study of the dynamics of Arnoux-Rauzy family. It is well known \cite{AR91} that a corresponding $F \in \CETthreehalf$ is minimal if and only if the corresponding triple belongs to the Rauzy gasket,  $(x_1, x_2, x_3) \in \boldsymbol{\mathcal{R}}$. Here $x_j$ and $l_j$ are related by \eqref{eq:relation_x_and_l}. One can ask ourselves, what is the dynamics of the trajectories corresponding to the parameters from the Rauzy gasket ? We conjecture that all of these trajectories escape non-linearly to infinity, and we prove this for a full measure set of points in  $\boldsymbol{\mathcal{R}}$ (with respect to the Avila-Hubert-Skripchenko measure defined in \cite{AHS16}). Moreover, there are no other non-linearly escaping trajectories, other than that coming from the Rauzy gasket and passing through the circumcenters.

%\begin{definition}
%We say that the trajectory $\delta$ is \emph{spiraling} if does not stay at bounded distance from a line. 
%
% if for the corresponding symbolic orbit $\omega$, there exist strictly monotonous subsequences of odd numbers $\{N_k\}_{k=1}^{\infty}, \{M_l\}_{l=1}^{\infty}, N_k, M_l \in 2\N+1$ such that the scalar products with basis vectors $e_x=(1,0)$ and $e_y=(0,1)$ alternate on these sequences:
%\begin{align*}
%\left(\sum_{j=0}^{N_k} \overrightarrow {\Omega_j \Omega_{j+1}}, e_x \right)>0, \left(\sum_{j=0}^{M_l} \overrightarrow {\Omega_j \Omega_{j+1}}, e_y \right)>0, k, l \in 2\Z;\\
%\left(\sum_{j=0}^{N_k} \overrightarrow {\Omega_j \Omega_{j+1}}, e_x \right)<0,\left(\sum_{j=0}^{M_l} \overrightarrow {\Omega_j \Omega_{j+1}}, e_y \right)<0,  k,l \in 2\Z+1.
%\end{align*}
%\end{definition}

\begin{proposition}\label{thm:when_escape}
Suppose that there exists a pair $(\delta, \Delta)$ such that a trajectory $\delta$ in a triangle tiling billiard with the tiles congruent to $\Delta$, is non-linearly escaping. Then
\begin{itemize}
\item [1.] the triangle $\Delta$ is acute,
\item [2.] the trajectory $\delta$ passes through the circumcenters of all the triangles that it crosses,
\item[3.] the triple of $x_j$ defined by \eqref{eq:relation_x_and_l}, where $l_j$ are the normalized angles of $\Delta$, belongs to the Rauzy gasket $\mathcal{R}$.
%the normalized angles of the triangle $(l_1, l_2, l_3):=(\frac{\alpha_1}{\pi},\frac{\alpha_2}{\pi},\frac{\alpha_3}{\pi})$ such that 

Moreover, if the pair $(\delta, \Delta)$ are such as in $1.-3.$ above and the correponding map in $\CETthree$ is uniquely ergodic, then the trajectory $\delta$ is %spiraling and hence, 
non-linearly escaping.
\end{itemize}
\end{proposition}

\begin{proof}
By Proposition \ref{prop:obtuse}, $\Delta$ is acute. The existence of a non-linearly escaping trajectory implies that $\tau=\frac{1}{2}$ for a corresponding $F \in \CETthree$ by Proposition \ref{thm:generic}. The point 3. follows from \cite{AR91} where it was proven that the only Arnoux-Rauzy minimal maps are those with parameters $x_j$ in the set $\mathcal{R}$.

Let us prove by contradiction that a trajectory is non-linearly escaping. Consider $T = F^2$ and let us consider its coding as in equation \eqref{eq:possible directions}. Let denote by $f$ the map with values in $\R^2$ that is this approximated displacement only depending on two consecutive letters of the symbolic coding. If the trajectory is -linearly escaping, there is a positive constant $C$ such that 
\begin{equation}\label{eq:linear}
\vert \vert \sum_{j=0}^{N} \overrightarrow {\Omega_j \Omega_{j+1}}\vert \vert > CN.
\end{equation}
By unique ergodicity of the map $T$, and by applying the ergodic theorem, we have
\begin{equation}\label{eq:ergodic_theorem}
\lim_{N\to \infty} \frac{1}{N}\vert \vert \sum_{j=0}^{N} \overrightarrow {\Omega_j \Omega_{j+1}}\vert \vert = \int_{\Sph^1}f(x) dx.
\end{equation}
%By symmetry, the interval with displacement $\overrightarrow {AB}$ (resp.  $\overrightarrow {AC}$, $\overrightarrow {BC}$) has the same length as the interval $\overrightarrow {BA}$ resp.  $\overrightarrow {CA}$, $\overrightarrow {CB}$). Thus $\int_{\Sph^1}f(x) dx =0$ which leads to a contradiction.
Note that here, one can denote $I_A:=I_1, I_B:=I_2, I_C:=I_2$ and analogously to Definition \eqref{eq:symbolicdefinitiontwosteps}, one can denote the intervals $I_{i,j}$ with the labels $i,j \in \mathcal{A}$, for the alphabet $\mathcal{A}=\{A,B,C\}$. Now we calculate explicitely the right-hand side of \eqref{eq:ergodic_theorem}:
\begin{equation*}
\int_{\Sph^1} f(x) dx = \sum_{i,j \in \mathcal{A}} \overline{i j} |I_{ij}|=
 \sum_{i,j \in \mathcal{A}} \left(
\overrightarrow{0j}-\overrightarrow{0i}
 \right)
  |I_{ij}|= \sum_j \overrightarrow{0j} \sum_i |I_{ij}| - \sum_{i} \overrightarrow{0i} \sum_j 
|I_{ij}| = 0, 
\end{equation*}
since $\sum_i |I_{ij}|=I_j, \sum_j |I_{ij}|=|I_i|$ by definition \eqref{eq:symbolicdefinitiontwosteps}.  Thus this calculation and the relation \eqref{eq:linear} give a contradiction in the relation \eqref{eq:ergodic_theorem}.
\end{proof}

This Proposition proves, for example, that the trajectory with parameters \ref{eq:Dianas-parameters} is indeed not linearly escaping.

We didn't manage to prove the non-linear escaping behavior for all the points in the set  $\boldsymbol{\mathcal{R}}$ but only for those which correspond to uniquely ergodic maps. 
%The unique ergodicity holds for almost every point in $\boldsymbol{\mathcal{R}}$ by \cite{AHS16}.

\section{Perspectives and open questions.}\label{sec:perspectives}
\subsection{Combinatorial orbits of non-linearly escaping trajectories.}\label{subs:Hooper-Weiss}

From the simulations, one can see that non-linearly escaping trajectories of triangle tiling billiards have a very interesting behavior. One can hope to understand it better even though Proposition \ref{thm:when_escape} proves a spiraling behavior for most of these trjactories (with respect to an appropriate measure).

\begin{problem}
Describe the combinatorial orbits of non-linearly escaping trajectories of the triangle tiling billiards. 
\end{problem}
This direction of research was already discussed in paragraph \ref{subs:HW_subsec} but we add it here for completeness. To some extent, a part of the work on this problem was done in \cite{AHS16}. 
%As the previous Section shows, the combinatorial orbits of non-linear
%In \cite{HW} Hooper and Weiss study the growing fractal curve which corresponds exactly to the Hooper-Weiss curve of a singular trajectory of a triangle tiling with parameters \eqref{eq:Dianas-parameters}. We would like to note that the Hooper-Weiss curve is interesting to study for any parameters in $C \times \mathbf{S}^1$ that correspond to non-linearly escaping trajectories (exactly those described by Theorem \ref{thm:when_escape}).

%The combonatorial o
%In terms of the work \cite{HW}, the change of the parameter $\tau$ of a triangle tiling billiard corresponds to the change of the vertical (relative) distance between the singularities on the translation surface and the change of $l_j$ corresponds to the horizontal distance change. In their work \cite{HW}, only the parameter $\tau$ is changed and the points $l_j$ are fixed as in \eqref{eq:Dianas-parameters}. We think that it can be interesting to study the general behavior of all Hooper-Weiss curves. 
%\label{sec:exceptional_behavior}
\subsection{Squares of fully flipped interval exchange transformations and $\SAF$ invariant.}
For a fully flipped interval exchange transformation $F \in \FETn [0,1)$ its square $T=F^2$ doesn't have any of its intervals flipped, $T \in {\mathrm{IET}^{2n-1}} [0,1)$. Let 
\begin{equation*}
\mathcal{F}:=\left\{
T \in {\mathrm{IET}^{2n-1}} \left|\right. \exists F \in \FETn : F^2=T
\right\}.
\end{equation*}
%and
%\begin{equation*}
%\Ftau:=\left\{
%T \in {\mathrm{IET}^{2n-1}} \left\|\right. \exists F \in \CETn : F^2=T
%\right\}
%\end{equation*}

Note that the set $\mathcal{F}$ is a set of half of the dimension of 
$ {\mathrm{IET}^{2n-1}} [0,1)$. 
%
%The set of all $(2n-1)$-IETs is $2n-2$-dimensional although the set $\mathcal{F}$ of IETs that come from flipped $n$-IETs is $n-1$-dimensional. So $\mathcal{F}$ is the set of half a dimension. 

%\boxed{\textbf{Is it a Lagrangian subset? Sasha Skrip's work?...}}

\begin{problem}
Describe the set $\mathcal{F}$ in terms of the invariants of interval exchange transformations.
\end{problem}

We still do not understand very well what exactly is this subset $\mathcal{F}$ in the set of all interval exchange transformations $\mathcal{F} \subset {\mathrm{IET}^{2n-1}}$. In this paragraph, we prove that $\mathcal{F}$ lies in the kernel of the SAF invariant.

\begin{definition}
Take $T$ is an interval exchange transformation, $T \in \mathrm{IET}^n$ such that $\lambda_i$ are the lengths of its intervals $I_i$ of continuity, and for every interval $I_i$ the map $T$ is a translation by some number $t_i$:
\begin{equation*} 
\left.T\right|_{I_i} (x) = x+t_i, 1 \leq i \leq n.
\end{equation*}

%$[a_{i-1},a_i)$ being the intervals of continuity. Let $\lambda_i=a_i-a_{i-1}, i \in \{1, \ldots, m \}$ be their lengths.
Then the \emph{Sah-Arnoux-Fathi invariant }of $T$ is a $\Q$-bilinear form which is given by
\begin{equation}\label{def:SAF}
\SAF(T)=\sum_{i=1}^m l_i \otimes_{\Q} t_i
\end{equation}
\end{definition}

The SAF invariant of an IET was defined by Arnoux in his thesis \cite{A81}, where he proved the following

\begin{theorem}\cite{A81}
Consider $T \in \IETn [0,1)$. Then its SAF-invariant has the following properties.
\begin{itemize}
\item[1.] If $T$ is periodic then $\SAF(T)=0$.
\item[2.] For $T=T_{\alpha}$ being a rotation of angle $\alpha$,
\begin{equation*} 
 \SAF(T_{\alpha})=1 \otimes \alpha - \alpha \otimes 1
 \end{equation*}
 And, consequently, $T_{\alpha}$ is periodic ($\alpha \in \Q$) if only if $SAF(T_{\alpha})=0$.
\item[3.] Let $T_1$ and $T_2$ be two first return maps of a directional flow on some translation surface $M$ with transversals $\tau_1, \tau_2 \subset M$ correspondingly. Then $\SAF(T_1)=SAF(T_2)$. In other words, $\SAF$ is an invariant of the flow. 
\item[4.] $\SAF(T)$ is invariant by the Rauzy induction.
\end{itemize}
\end{theorem}

A $\SAF$ invariant is a very powerful tool for working with interval exchange transformations and translation surfaces. For example, for a surface of genus $2$, if $SAF=0$ then the translation flow is not minimal, by the works of Arnoux, Boshernitzan and Caroll Calta, McMullen (see \cite{A88}, \cite{BC}, \cite{Ca},  \cite{M12}).

Although, in higher genus the nullity of SAF is not equivalent to the absence of minimality. Indeed, the Arnoux-Yoccoz example \cite{A88} already mentionned in paragraph \ref{subs:HW_subsec} is the first example of an IET that has zero $\SAF$-invariant and is minimal. Its suspension dynamics corresponds to a translation flow on the surface of genus $3$. 

\begin{proposition}
For any map $F \in \FETn$ its square $T=F^2$ has zero SAF invariant.
\end{proposition}

\begin{proof}
Divide the initial interval of definition of $F$ into the intervals $I_i$ of continuity of $F$. On each of these intervals, one writes 
%The interval $[0,1)$ is a union of $m$ intervals $I_j=[a_{j-1}, a_j), j \in \{1, \ldots, m \}$. 
%On each interval $I_j$ the restriction
 $F \left| \right. _{I_j}=-x+\tau_j, \tau_j \in \R$. We subdivide each of these intervals into the subsets $I_{j,k}$ (which are either segments or  empty sets) defined as in \eqref{eq:symbolicdefinitiontwosteps}. We have then $\cup_{k=1}^m I_{j,k}= I_j$ and $\cup_{j=1}^m I_{j,k}=F^{-1}(I_k)$. Also, we have $\sum_k |I_{j,k}|=\lambda_j$ and also $\sum_j |I_{j,k}|=\lambda_k$ since $F$ is a bijection outside the singularities.

The restriction of $T=F^2$ on each of these intervals is a shift: 
\begin{equation*}
T\left|\right._{I_{j,k}}(x)= - (-x+\tau_j) + \tau_k = x - \tau_j + \tau_k, \; \forall j,k.
\end{equation*}
One can finally calculate
\begin{multline*}
\SAF(T)=\sum_{j,k=1}^m |I_{j,k}| \otimes (\tau_k-\tau_j)=- \sum_j \left( \sum_k |I_{j,k}| \otimes \tau_j \right) + \sum_k \left( \sum_j |I_{j,k}| \otimes \tau_k \right) =\\
= - \sum_j \left( \sum_k |I_{j,k}|  \right) \otimes \tau_j + \sum_k \left( \sum_j |I_{j,k}| \right) \otimes \tau_k = -\sum_j \lambda_j \otimes \tau_j + \sum_k \lambda_k \otimes \tau_k = 0.
\end{multline*}
\end{proof}

A version of this Proposition in the context of flows on translation surfaces has been proven in \cite{S18}. There B. Strenner shows that any pseudo-Anosov map that is a lift of a pseudo-Anosov homeomorphism of a \emph{nonorientable} surface has a vanishing $\SAF$ invariant. He also gives an explicit construction of this nonorientable surface for the Arnoux-Yoccoz interval exchange. His construction can be directly repeated for any $F \in \CETthreehalf$ that defines a pseudo-Anosov map. 

We hope that the better understanding of the role of the nullity of SAF invariant for fully flipped interval exchange transformations can help to find the geometric proof of Lemma \ref{lemma:main_vector_preserved_lemma} as well as, maybe, solve the Conjecture \ref{conj:tree}.
% see in the following paragraph \ref{subs:tree}.

\subsection{Minimality of interval exchange transformations with flips.}\label{subs:minimality_and_connections}
The combinatorial study of the modified Rauzy graphs in connection to minimality seems a promising area of research, see e.g. \cite{DM17} for the combinatorial approach of standard Rauzy graphs. The modified Rauzy graphs are very different from standard Rauzy graphs: for example, if there exists a path from one vertex to another in a modified graph, there is not necessarily a path backward (which was yet true for the standard Rauzy graphs).

\smallskip

We would like to formulate a general
\begin{conjecture}\label{conj:minimality}
For any $n$, if $F \in \CETn$ is minimal on $\mathbb{S}^1$ then $\tau=\frac{1}{2}$.
\end{conjecture}

\begin{theorem}\label{ref:thm_4}
Conjecture \ref{conj:minimality} is true for $n=4$.
\end{theorem}

\begin{proof}
Proof is analogical to the case when $n=3$, and also uses an analogue of Lemma \ref{lemma:main_vector_preserved_lemma} whose proof is computer assisted. 
\end{proof} 

\begin{remark}
The calculations for the proof of Theorem \ref{ref:thm_4} are heavier than that of the the proof of Theorem~\ref{thm:first_ingredient} since they involve the study of the connected component of the permutation

\begin{equation}\label{eq:minimal_permutation_4}
\begin{pmatrix}
\overline{A}&\overline{B}&\overline{C}&\overline{D}&\overline{E} \\
\overline{B}&\overline{D}&\overline{E}&\overline{A}&\overline{C}
\end{pmatrix}
\end{equation}

in a bigger Rauzy graph (that of $\IETFfive$). The permutation \eqref{eq:minimal_permutation_4} is the only combinatorial type of a minimal permutation in $\FETfive$ coming from the family $\CETfour$ (up to the change of the orientation of the initial circle) by adapting the argument in the Remark of paragraph \ref{subs:combinatorial_proof}. An invariant analogue to that of Lemma \ref{lemma:main_vector_preserved_lemma} exists for the vector $v^{\perp}:=(1,-1,1, -1,-1)^T$, and hence the proof of Theorem \ref{ref:thm_4} is repeated along the lines of the proof of Theorem \ref{thm:first_ingredient}.
\end{remark}

Conjecture \ref{conj:minimality} is related to a following problem.

\begin{problem}
Describe all minimal maps in the family $\IETFn$.
\end{problem}

\smallskip

We manage to give the answer to this problem for $n=4$ as a corollary of the study of the invariants of the modified Rauzy graph for $\IETFfour$.

\begin{theorem}
A map $F \in \IETFfour$ is minimal if and only if the following two conditions hold:
\begin{itemize}
\item[1.] there exists a Rauzy path in the Rauzy graph of $\IETFfour$ that connects a map $F_0$ with the combinatorics given by \eqref{eq:basic_combinatorial_type_CET_3} to $F$. In other words, $\exists N$ such that $\mathcal{R}^N F_0 = F$;
\item[2.] for a corresponding map $F_0$ permuting the intervals $I_i$ labeled by $i \in \mathcal{A}=\{A,B,C,D\}$, with the combinatorics \eqref{eq:basic_combinatorial_type_CET_3}, the following Rauzy condition on the lengths $\lambda_i, i \in \mathcal{A}$ holds :
\begin{align*}
\lambda_A+\lambda_B+\lambda_C+\lambda_D=\lambda,\\
\frac{1}{\lambda}\left(
\lambda_B+\lambda_C+\lambda_D-\lambda_A, \lambda_A-\lambda_B-\lambda_C+\lambda_D, \lambda_A+\lambda_B+\lambda_C-\lambda_D
\right)\in \boldsymbol{\mathcal{R}}.
\end{align*}
\end{itemize}
\end{theorem}

\begin{proof}
 First, by explicit computer-assisted search we conclude that the only minimal component (a component corresponding to minimal transformations) in the quotient Rauzy graph is that of the permutation \eqref{eq:basic_combinatorial_type_CET_3}. This and the results obtained above (namely, Propositions \ref{thm:generic} and \ref{thm:when_escape}) conclude the proof. The condition on the lengths $\lambda_i, i \in \mathcal{A}$ follows from the following relations. First, $(x_1,x_2,x_3) \in \boldsymbol{\mathcal{R}}$ by Proposition \ref{thm:when_escape}, where $l_j$ and $x_j$ are related by \eqref{eq:relation_x_and_l}. Second, the relation between $l_j, j=1,2,3$ and $\lambda_i, i \in \mathcal{A}$ is given by $[l_1:l_2:l_3]=[\lambda_A: \lambda_B+\lambda_C: \lambda_D]$, see for example Figure \ref{fig:combinatoric_type}.
\end{proof}

\subsection{Quadrilateral tilings.}\label{sec:polygonal_tilings}
The Definition \ref{def:reflection_in_a_circumcircle:2} of a system of reflections in a circumcircle can be generalized in an obvious way to any inscribed $n$-polygon. Although, not any polygon tiles the plane.
A tiling by congruent polygones in the plane can be achieved only for $n=3,4,5$ and $6$. All of these tilings are classified, with a final achievement by M. Rao in proving the classification of all pentagonal tiling families into $15$ (already known) classes \cite{R17}. In this paragraph, we concentrate on the case of quadrilateral tilings but we mention this interesting

\begin{problem}
Study the behavior of tiling billiards on all the tilings of a plane by congruent pentagons (hexagons).
\end{problem}

\begin{problem}
Any quadrilateral tiles a plane in a periodic way. Study the behavior of tiling billiards in such a quadrilateral tiling. For example, the parallelogramm tiling.
\end{problem}

We think that for the parallelogramm tiling, a trajectory of a billiard will behave in a way closer to that described in \cite{DH18}, and fill densely the open subsets of the plane. It is also interesting to see how the tiling billiard passes from one behavior to the other once the geometry of the tiling changes (from trihexagonal tiling to triangle tiling, or from tiling into parallelograms to square tiling).

Here, we restrict ourselves to the case of the quadrilaterals inscribed in circles (cyclic quadrilaterals), in order to the analogue of Proposition \ref{prop:Dianacomeback} to hold. The class of tilings we study is \emph{periodic cyclic quadrilateral tilings } on the plane, defined as follows:  for any cyclic quadrilateral in a tiling a neighbouring tile is obtained by a central symmetry in the middle of the edge, all the tiles being congruent. As before, one can associate to any such tiling the map $F_{\tau, l_1, l_2, l_3, l_4} \in \CETfour$. The lengths $l_j$ correspond to the lengths of the sides of the inscribed quadrilateral, and $\tau$ to the direction of the trajectory in a folded system, as before.

Following the strategy of the proof of Proposition \ref{thm:generic}, by applying Theorem \ref{ref:thm_4} and Proposition \ref{thm:CETfour}, we show that qualitatively periodic cyclic quadrilateral tilings have the same behavior as triangle tiling billiards, see Figure \ref{fig:quadrilaterals_simulations}. 

\begin{theorem}\label{thm:minimality_4}
Consider the set of all cyclic quadrilaterals with the set of lengths $l_j$ such that $l_j, j=1,..,4$ are independent over $\Q$. Consider any trajectory $\delta$ in a corresponding tiling that doesn't pass through circumcenters of the tiles. Then the following holds:
\begin{itemize}
\item[1.] The trajectory $\delta$ is not drift-periodic;
\item[2.] if $\delta$ is closed, its symbolic orbit is equal to $\overline{s^2}$ for some word $s$ in the alphabet of sides $\mathcal{A}:=\{a,b,c,d\}$ of odd length. Consequently, the period of $\delta$ is equal to $4n+2$ for some $n \in \N^*$;
\item[3.] if $\delta$ is not closed, it is linearly escaping. Its symbolic coding can be described by an infinite word $\omega$ that can be represented as an infinite concatenation of the words $w_1$ and $w_2$ forming a sturmian sequence.
\end{itemize}
\end{theorem}

As we have seen in Section \ref{sec:integrability_section}, starting from $n=5$, the family $\CETn$ exhibits non-integrable behavior in open sets. Maybe this is related to the fact, that most of the inscribed pentagons never tile the plane? ...

For the quadrilaterals, one can ask the same questions about exceptional behavior of the trajectories (related to the problems discussed in the previous paragraph), as well as generalize Conjecture \ref{conj:tree}.

\begin{problem}
Study the set of exceptional (possibly non-linearly escaping) trajectories in the periodic cyclic quadrilateral tilings. This set is a subset of the trajectories passing through the circumcenters of the tiles. In other words, are the parameters $l_j, j=1, \ldots, 4$, for which the maps in $F_{\frac{1}{2}, l_1, l_2, l_3, l_4} \in \CETfour$ are minimal ?
\end{problem}

The answer to this question can probably be given in the terms of Theorem \ref{thm:minimality_4} and introduce a higher dimensional analogue of the Rauzy gasket.

%The Figure \ref{fig:quadrilaterals_simulations} shows two examples of generic behaviors of trajectories in periodic cyclic quadrilateral tilings: periodic and linearly escaping behaviors, related to integrability of the maps in $\CETfour$. 

\begin{conjecture}[Tree conjecture for quadrilateral tilings]\label{conj:tree_quadrilateral}
Let $\Lambda$ be as a union of all vertices and edges of all drawn triangles in a periodic cyclic quadrilateral tiling.
Take any periodic closed trajectory $\delta$ of a corresponding billiard. It incloses some bounded domain $U \subset \mathbb{R}^2$ in the plane, $\partial U = \delta$ and $U \cap \Lambda$ is an embedding of some graph in the plane. Then this graph is a \emph{tree}. 
\end{conjecture}

 \begin{figure}
\centering

\includegraphics[scale=0.4]{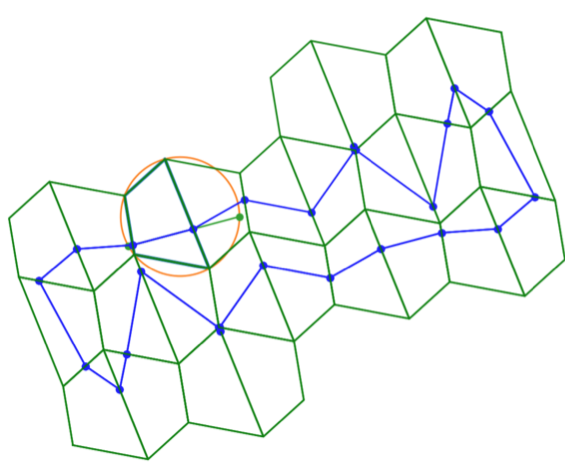}
\includegraphics[scale=0.4]{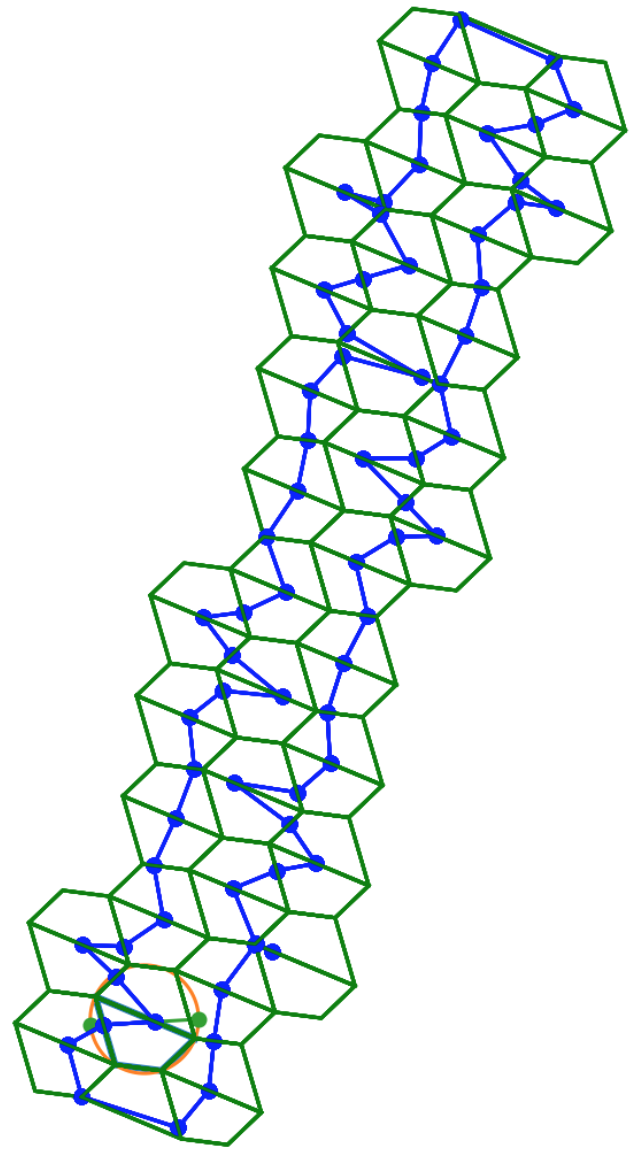}
\includegraphics[scale=0.4]{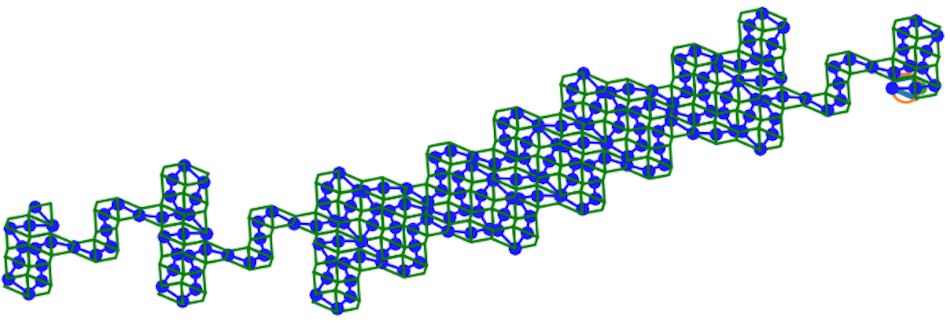}
\includegraphics[scale=0.3]{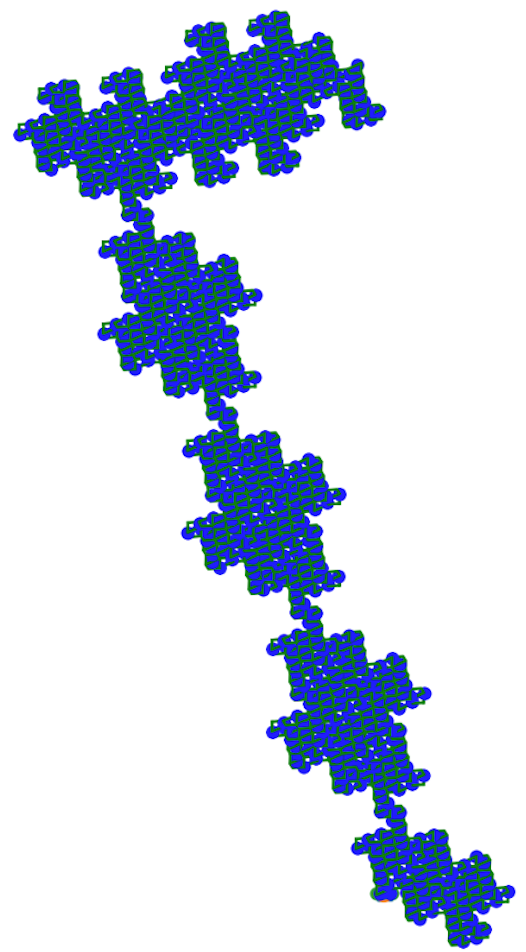}

\caption[]{Periodic cyclic quadrilateral tiling billiards. The two trajectories on the top are closed and the two on the bottom are linearly escaping. A program simulating quadrilateral tiling billiards dynamics has been written by the second author in collaboration with Ilya Schurov.}\label{fig:quadrilaterals_simulations}
\end{figure}

\begin{landscape}
\begin{figure}
\centering
\includegraphics*[scale=0.8 ]{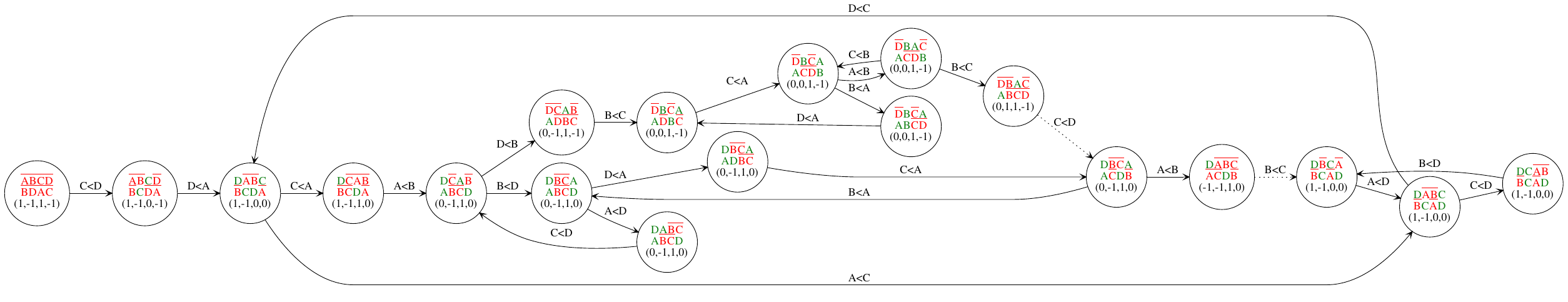}
\caption[]{
%\textbf{OLGA ! CHANGE THE PICTURE AND THE TEXT}
This is a connected component of the permutation \eqref{eq:basic_combinatorial_type_CET_3} in the quotient Rauzy graph of $\IETFfour$, i.e it is exactly the same component as that represented on Figure \ref{pic:graph_for_four}. This component has $19$ vertices (for the non-irreducible vertices are only represented on this picture). Here this component is drawn with some additional information. To each vertex $[\sigma]$ of this quotient Rauzy graph  one associates a vector $A_{[\sigma]} {v}^{\perp}$ with the components labeled by the letters of the alphabet $\mathcal{A}=\{A,B,C,D\}$. It happens that each of the components is either equal to $0,1$ or $1-1$.The non-flipped intervals have their corresponding components equal to $0$ and the flipped intervals have their corresponding components equal to $\pm 1$. The initial permutation\eqref{eq:basic_combinatorial_type_CET_3} has its corresponding vector equal to $v^{\perp}=(1,-1,1,-1)$.}
\label{fig:part_graph_4_connected_component_equivalence}
\end{figure}
\end{landscape}

\section{Appendix: the Rauzy graphs.}
This Appendix includes a few examples of modified Rauzy graphs (or parts of these graphs) that we constructed for this work, in collaboration with Paul Mercat. We put the links for downloading some of them which are too big to be inserted in the paper in the following list.

\begin{center}
\textbf{Rauzy graphs of interval exchange transformations with flips: some examples.}
\end{center}

For the study of triangle tiling billiards and the minimality properties of the associated interval exchange transformations with flips, we studied the full connected component of the permutation $\sigma$, defined by \eqref{eq:basic_combinatorial_type_CET_3}, in the Rauzy graph of $\IETFfour$. In Figure \ref{pic:graph_for_four} we represent a quotient of this connected component. Here are some additional figures:
\begin{itemize}
\item[•] In Figure \ref{fig:part_graph_4_connected_component_equivalence}, we represent this quotient with an additional information on the values of the invariant of the vertex of the graph given by a vector with entries in the set $\{0,1,-1\}$ (see the discussion in Section \ref{sec:one-half}),
\item[•] for downloading the full connected component of $\sigma$, follow the link \url{https://drive.google.com/file/d/14xbZ5a63PetQPGxBcmGfHfMjLkEsN9fn/view?usp=sharing},
\item[•] for the same full connected component of $\sigma$ with the additional information on the combinatorial invarant, follow the link 

\url{https://drive.google.com/file/d/1gbV_QETrQEJy1QaWyqcxM2Ob4HVIphhj/view?usp=sharing}.
\end{itemize}

In order to understand the behaviour of the maps in the family $\CETfour$ and prove the integrability of almost all the maps in the family (Proposition \ref{thm:CETfour}), one needs to study the connected component of the permutation $\sigma$ defined by \eqref{eq:combinatorics_CETfour}. The number of the (irreducible) vertices in this component is equal to $8 222$. For obvious reasons, we do not draw the entire graph here. The corresponding component in the quotient graph is much smaller (it has only $130$ vertices) and can be downloaded here, 
\url{https://drive.google.com/file/d/12i8mg8PjIvykq7daxxYhqQXLaH8Ts5S6/view?usp=sharing}, with the additional combinatorial data associated to it.

\begin{center}
\textbf{Acknowledgements.}
\end{center}
We would like to thank the organizers of the conferences \emph{Teichmüller Space, Polygonal Billiard, Interval Exchanges} at CIRM in February 2017 (where the work on the project started) as well as \emph{Teichmüller Dynamics, Mapping Class Groups and Applications at Institut Fourier} in June 2018\footnote{O.P-R's lecture based on this paper is available at the Youtube channel of Institut Fourier here: \url{https://www.youtube.com/watch?v=I91c-g_BzbM}\href{https://www.youtube.com/watch?v=I91c-g_BzbM}{}} (where the work on the project continued) that provided wonderful environment for research and exchange. We are both very grateful to Diana Davis for her enthousiastic talk in February $2017$ in CIRM that introduced us to tiling billiards as well as for the graphic representation of IET's with flips (as in Figure \ref{fig:fullyflipped}) that we use throughout this article.

The second author benefits from the support of the French government “Investissements d’Avenir” program ANR-11-LABX-0020-01 - Centre Henri Lebesgue \& Région Bretagne - dispositif SAD. During the work on this article she was also supported by the grant L'Oréal-UNESCO for Women in Science $2016$. 

This article wouldn't exist without the help of Paul Mercat who wrote the program that drew modified Rauzy graphs for the maps in $\CETthree$ and $\CETfour$. The proof of Lemma \ref{lemma:main_vector_preserved_lemma} which is crucial for our work, is for now computer assisted. The needed calculations were done by the program that Paul wrote. 

The second author thanks Ilya Schurov for his help on the program for the trajectories in quadrilateral tilings. We would like to thank Pat Hooper and Alexander St Laurent for their program that draws the tilling billiard trajectories, accessible on-line \cite{HSL}, and Shigeki Akiyama for suggesting us a new representation of tiling billiards as the systems of tangent reflections. 
%\begin{figure}
%\centering
%%\includegraphics[scale=0.25]{graph_5.eps}
%\includegraphics[scale=0.68]{bdeac6.pdf}
%\caption[]{A part of the connected component of the permutation 2. in the modified Rauzy graph for $5$ interval exchange transformations with flips (for $6$ or less steps of Rauzy induction process). The colors of the letters correspond to values of the vector $\boldsymbol{k}$: $k_i=-1$ if and only if the color of a corresponding letter in $\mathcal{A}$ is red and $k_i=1$ if and only if the color is green. The vertices that are drawn correspond to irreducible permutations.}\label{fig:part_of_the_graph_for_5}
%\end{figure}

%
%\begin{landscape}
%\begin{figure}
%\includegraphics*[scale=1.1]{ABCDE_BDEAC_mod.pdf}
%\end{figure}
%\end{landscape}

\Addresses


\begin{thebibliography}{99}
\bibitem[ABB11]{ABB11} P. Arnoux, J. Bernat, X. Bressaud, \emph{Geometrical models for substitutions}, Exp. Math. 20, 97–127 (2011)
\bibitem[A88]{A88} P. Arnoux, \emph{Un exemple de semi-conjugaison entre un \'{e}change d’intervalles et une translation sur le tore}, Bull. Soc. Math. France 116, 489--500 (1988)
\bibitem[A81]{A81} P. Arnoux, \emph{Un invariant pour les échanges d'intervalles et les flots sur les surfaces}, doctoral thesis (1981)
\bibitem[AR91]{AR91} P.Arnoux, G. Rauzy, \emph{Représentation géométrique des suites de complexité }$2n+1$, Bulletin de la SMF., 119: 2, 199--215 (1991)
\bibitem[AS13]{AS13} P.Arnoux, S. Starosta, \emph{The Rauzy gasket}, Birkhäuser Boston. Further Developments in Fractals and Related Fields, Springer Science+Business Media New York, 1--23, Trends in Mathematics (2013)
\bibitem[AHS16]{AHS16} A. Avila, P. Hubert, A. Skripchenko, \emph{Diffusion for chaotic plane sections of $3$-periodic surfaces}, Inventiones mathematicae,volume 206, issue 1, 109--146 (2016)
\bibitem[AHS16-1]{AHS16-1} A.Avila, P. Hubert, A. Skripchenko, \emph{On the Hausdorff dimension of the Rauzy gasket}, Bulletin de la société mathématique de France, 144 (3), pp.539 - 568 (2016)
\bibitem[BC]{BC} M. Boshernitzan, C. Carroll, \emph{ An extension of Lagrange's theorem to interval exchange transformations over quadratic fields}, J. Anal. Math. 72 (1997), 21--44.
\bibitem[BDFI18]{BDFI18} P.Baird-Smith, D.Davis, E.Fromm, S.Iyer, \emph{Tiling billiards on triangle tilings, and interval exchange transformations}, preprint, \url{http://www.swarthmore.edu/NatSci/ddavis3/triangle_tiling_billiards.pdf} (2018)
%\bibitem{inv1} K. Chang. Light Fantastic: Flirting With Invisibility. The New York Times, June 12 (2007)
\bibitem[BL]{BL} C.Boissy, E. Lanneau \emph{Dynamics and geometry of the Rauzy-Veech induction for quadratic differentials}, Ergodic Theory and Dynamical Systems (2008)
\bibitem[Ca]{Ca} K. Calta, \emph{ Veech surfaces and complete periodicity in genus two}, J. Amer. Math. Soc. 17 (2004), no. 4, 871--908.

\bibitem[DH18]{DH18} Diana Davis, W. Patrick Hooper, \emph{Periodicity and ergodicity in the trihexagonal tiling}, accepted pending revision in Commentarii Mathematici Helvetici (2018)
\bibitem[DDRSL16]{DDRSL16} D. Davis, K. DiPietro, J. Rustad, A. St Laurent, \emph{Negative refraction and tiling billiards, to appear in Advances in Geometry} (2016)
\bibitem[D16]{D16} V. Delecroix, \emph{Interval exchange transformations}, Lecture Notes, Salta (Argentina) (2016)
\bibitem[DDL09]{DDL09} R. De Leo, I. Dynnikov, \emph{Geometry of plane sections of the infinite regular skew polyhedron }$\{4, 6 \left|\right. 4\}$, Geom. Dedicata 138, 51–67 (2009)
\bibitem[D97]{D97} I. Dynnikov \emph{Semiclassical motion of the electron. a proof of the Novikov conjecture in general position and counterexamples}, In: Solitons, Geometry and Topology: on the Cross road, Translations of the AMS, Ser. 2, 179, AMS, Providence, 45–73 (1997)
\bibitem[G16]{G16} P. Glendinning, \emph{Geometry of refractions and reflections through a biperiodic medium}, Siam J. Appl. Math., Society for Industrial and Applied Mathematics 76: 4, 1219–1238 (2016)
\bibitem[GPR]{GPR} \emph{Je voudrais vous parler de mathématiques...}, short film co-created by C. Goudron and O. 
Paris-Romaskevich, for a competition Symbiose 48 hour film project, scientific documentary festival \emph{PariScience2018}, 
\url{https://vimeo.com/297265239}
(2018) 
\bibitem[HSL]{HSL} P. Hooper, Alexander St Laurent, \emph{Negative Snell law tiling billiards trajectory simulations}, 
\url{http://awstlaur.github.io/negsnel/}
\bibitem[HW18]{HW18} W. Patrick Hooper, B. Weiss,\emph{ Rel leaves of the Arnoux-Yoccoz surfaces}, Selecta Mathematica, 24:2, 875-934 (2018)
\bibitem[K75]{K75} M. Keane, \emph{Interval exchange transformations} Math. Z. 141, 25-31 (1975)
\bibitem[KZ03]{KZ03} M. Kontsevich, A. Zorich \emph{Connected components of the moduli spaces of Abelian differentials with prescribed singularities}, Inventiones mathematicae, 153:3, 631--678 (2003)
\bibitem[LPV07]{LPV07} J. H. Lowenstein, G. Poggiaspalla, and F. Vivaldi, \emph{Interval exchange transformations over algebraic number fields: the cubic Arnoux-Yoccoz model}, Dynamical Systems, 22(1), 73--106 (2007)
\bibitem[MMY10]{MMY10} S. Marmi, P. Moussa, J.-C. Yoccoz ,\emph{Affine interval exchange maps with a wandering interval}, Proc. London Math. Soc. (3) 100, 639--669 (2010)
\bibitem[DM17]{DM17}  Q. de Mourgues, \emph{A combinatorial approach to Rauzy-type dynamics}, Université Paris 13, thesis (2017)
\bibitem[MF]{MF} A. Mascarenhas, B. Fluegel \emph{Antisymmetry and the breakdown of Bloch's theorem for light}, unpublished draft (2015)
\bibitem[M03]{M03} C. McMullen, \emph{Teichm\"uller geodesics of infinite complexity}. Acta Math. 191 (2003), no. 2, 191--223
\bibitem[M15]{M12} C. McMullen, \emph{Cascades in the dynamics of measured foliations}, Annales Scientifiques de l'École Normale Supérieure, (4) 48 (2015), no. 1, 1--39
\bibitem[N89]{N89} A. Nogueira, \emph{Almost all interval exchange transformations with flips are nonergodic}, Ergodic Theory Dynam. Systems 9:3, 515-525 (1989)
\bibitem[N82]{N82} S.P.Novikov, \emph{The Hamiltonian formalism and multivalued analogue of Morse theory},
(Russian) Uspekhi Mat. Nauk 37: 5, 3–49 (1982); translated in Russian Math. Surveys 37:5, 1–56 (1982)
\bibitem[R17]{R17} M. Rao, \emph{Exhaustive search of convex pentagons which tile the plane} (2017)
\bibitem[R79]{R79} G. Rauzy, \emph{Échanges d’intervalles et transformations induites}, Acta Arith., 34(4):315–328, (1979)
\bibitem[SSS01]{SSS01} R. A. Shelby, D. R. Smith, S. Schultz \emph{Experimental Verification of a Negative Index of Refraction}, Science, Vol. 292 no. 5514, 77–79 (2001)
\bibitem[S18]{S18} B. Strenner, \emph{Lifts of pseudo-Anosov homeomorphisms of nonorientable surfaces have vanishing SAF invariant}, Mathematical Research Letters, 25:2 (2018)
\bibitem[SPW04]{SPW04} D. Smith, J. Pendry, M. Wiltshire  \emph{Metamaterials and negative refractive index}, Science, Vol. 305, 788–792 (2004)
\bibitem[ST18]{ST18} A. Skripchenko, S. Troubetzkoy, \emph{On the Hausdorff dimension of minimal interval exchange transformations with flips}, Journal London Mathematical Society, to appear.
\bibitem[VZZ08]{VZZ08} J. Valentine, S. Zhang, T. Zentgraf, E. Ulin-Avila, D. A. Genov, G. Bartal and X. Zhang, \emph{Three-dimensional optical metamaterial with a negative refractive index}, Nature, 455 (2008)
\bibitem[Z84]{Z84} A. Zorich, \emph{A Problem of Novikov on the Semiclassical Motion of an Electron in a Uniform Almost Rational Magnetic Field}, Russ. Math. Surv. 39 (5), 287–288 (1984)
\end{thebibliography}
\end{document}